\documentclass[12pt, reqno]{amsart}
\usepackage{floatrow}
\usepackage{amssymb, amsfonts, amsthm, amsmath, amscd}
\usepackage{graphicx}
\usepackage[latin2]{inputenc}
\usepackage{t1enc}
\usepackage[mathscr]{eucal}
\usepackage{indentfirst}
\usepackage{pict2e}
\usepackage{epic}
\numberwithin{equation}{section}
\usepackage{stmaryrd}
\usepackage{scalerel,stackengine}
\usepackage[margin=2.5cm]{geometry}
\usepackage{mathtools}
\usepackage[colorlinks = true,
            linkcolor = blue,
            urlcolor  = blue,
            citecolor = blue,
            anchorcolor = blue]{hyperref}
\usepackage[titletoc]{appendix}
\usepackage[T1]{fontenc}
\usepackage[numbers,square]{natbib}
\usepackage{adjustbox}
\usepackage{dsfont}
\usepackage{times}
\usepackage{enumerate}
\usepackage{setspace}
\usepackage{leftidx}
\usepackage{multirow}

\usepackage{tikz-cd}
\usepackage{pgf,tikz}
\usetikzlibrary{math}
\usetikzlibrary{arrows,decorations.markings}
\usetikzlibrary{patterns,intersections}
\usetikzlibrary{calc}
\usetikzlibrary{decorations.pathreplacing}

\makeatletter
\newcommand*\bulletsmall{\mathpalette\bulletsmall@{.5}}
\newcommand*\bulletsmall@[2]{\mathbin{\vcenter{\hbox{\scalebox{#2}{$\m@th#1\bullet$}}}}}
\makeatother

\makeatletter
\DeclareRobustCommand\bigop[1]{%
  \mathop{\vphantom{\sum}\mathpalette\bigop@{#1}}\slimits@
}
\newcommand{\bigop@}[2]{%
  \vcenter{%
    \sbox\z@{$#1\sum$}%
    \hbox{\resizebox{\ifx#1\displaystyle.9\fi\dimexpr\ht\z@+\dp\z@}{!}{$\m@th#2$}}%
  }%
}
\makeatother
\newenvironment{myproof}[2]{\paragraph{\textit{Proof of {#1} }{#2}.}}{\hfill$\square$}
\newenvironment{myproof2}[1]{\paragraph{\textit{Proof of {#1}}}}{\hfill$\square$}
\newenvironment{mystep}[1]{\paragraph{\textbf{Step {#1}}}}{}
\newenvironment{mystepf}{\paragraph{\textbf{Final step}}}{}

\theoremstyle{plain}
\newtheorem{theorem}{Theorem}[section]
\newtheorem{lemma}[theorem]{Lemma}
\newtheorem{corollary}[theorem]{Corollary}
\newtheorem{proposition}[theorem]{Proposition}

\theoremstyle{definition}
\newtheorem{definition}[theorem]{Definition}
\newtheorem{assume}[theorem]{Assumption}
\newtheorem{remark}[theorem]{Remark}

\newtheorem{condition}[theorem]{Condition}
\newtheorem{problem}[theorem]{Problem}
\newtheorem{example}[theorem]{Example}

\def\CC{\mathbb{C}}
\def\NN{\mathbb{N}}
\def\QQ{\mathbb{Q}}
\def\RR{\mathbb{R}}
\def\ZZ{\mathbb{Z}}
\def\ds{\displaystyle}

\def\a{\alpha}

\def\g{\gamma}
\def\s{\sigma}
\def\G{\Gamma}
\def\w{\omega}
\def\l{\lambda}
\def\ra{\rightarrow}
\def\tt{\theta}

\def\ol{\overline}
\def\O{\Omega}

\def\M{\mathcal{M}}

\def\gg{\mathfrak{g}}

\def\UU{\mathcal{U}}
\def\t{\mathfrak{t}}

\def\OO{\mathcal{O}}

\def\X{\mathcal{X}}
\def\Y{\mathcal{Y}}
\def\Ham{Hamiltonian }
\def\e{\epsilon}
\def\J{\mathcal{J}}
\def\ul{\underline}
\def\S{\Sigma}
\def\sle{strip-like end }
\def\pt{\partial}
\def\I{\mathcal{I}}

\def\HH{\mathbb{H}}

\def\cl{a}

\def\mr{\mathring}
\def\AA{\mathcal{A}}

\def\wc{\mathcal{C}}

\def\monoconst{\tau}

\def\capping{moment capping }

\def\D{\Delta}
\def\MC{\Phi_{MC}}
\def\holo{pseudoholomorphic }
\def\P{P}
\def\PP{\mathcal{P}}
\def\lo{\varphi}
\newcommand{\Nov}[2]{\Lambda^{#2}_{#1}}
\newcommand{\Se}[1]{\Phi_{Se}^{#1}}
\def\H{\mathcal{H}}
\def\Q{Q^{\vee}}
\def\BS{BS}
\def\FF{\mathcal{F}}
\def\CCC{\mathcal{C}}

\def\Lo{\mathcal{L}}
\def\GG{\mathcal{G}}
\def\eit{e^{i\tt}}
\def\cu{\g}
\newcommand{\bighash}{\DOTSB\bigop{\mathrm{\#}}}
\def\ii{e}
\def\TT{T}
\newcommand{\See}[2]{\Phi_{#1}^{#2}}

\def\ppp{p}
\def\XX{X^-\times X}
\DeclareMathOperator{\qcp}{qcp}
\def\vp{\varphi}
\def\HH{\mathcal{H}}
\DeclareMathOperator{\fiber}{fiber}
\DeclareMathOperator{\Crit}{Crit}
\DeclareMathOperator{\codim}{codim}
\DeclareMathOperator{\im}{Im}
\DeclareMathOperator{\ind}{Ind}

\DeclareMathOperator{\id}{id}
\DeclareMathOperator{\ad}{ad}
\DeclareMathOperator{\Aut}{Aut}
\DeclareMathOperator{\pr}{pr}

\DeclareMathOperator{\Ad}{Ad}

\DeclareMathOperator{\Lie}{Lie}
\DeclareMathOperator{\crit}{Crit}

\DeclareMathOperator{\Span}{span}
\DeclareMathOperator{\Aff}{Aff}
\DeclareMathOperator{\rk}{rk}
\begin{document}
\title[Quantum characteristic classes, moment correspondences and $Ham(G/L)$]{Quantum characteristic classes, moment correspondences and the Hamiltonian groups of coadjoint orbits}

\author{Chi Hong Chow}
\address{The Institute of Mathematical Sciences and Department of Mathematics, The Chinese University of Hong Kong, Shatin, Hong Kong}
\email{chchow@math.cuhk.edu.hk}

\begin{abstract}
For any coadjoint orbit $G/L$, we determine all useful terms of the associated Savelyev-Seidel morphism defined on $H_{-*}(\Omega G)$. Immediate consequences are: (1) the dimension of the kernel of the natural map $\pi_*(G)\otimes \mathbb{Q}\rightarrow \pi_*(Ham(G/L))\otimes \mathbb{Q}$ is at most the semi-simple rank of $L$, and (2) the Bott-Samelson cycles in $\Omega G$ which correspond to Peterson elements are solutions to the min-max problem for Hofer's max-length functional on $\Omega Ham(G/L)$.

The proof is based on Bae-Chow-Leung's recent computation of Ma'u-Wehrheim-Woodward morphism for the moment correspondence associated to $G/T$ where $T$ is a maximal torus, the computation of Abbondandolo-Schwarz isomorphism for $G$, and two theoretical results including the coincidence of the above Savelyev-Seidel and Ma'u-Wehrheim-Woodward morphisms, and a Leray-type spectral sequence relating Savelyev-Seidel morphisms for $G/L$ and $G/T$.

These ingredients also allow us to obtain an alternative proof of Peterson-Woodward's comparison formula which relates the quantum cohomology of $G/T$ to that of $G/L$.
\end{abstract}

\subjclass[2010]{53D35, 53D40, 57R17, 57S05}
\keywords{coadjoint orbit, Hamiltonian group, Hofer geometry, quilted Floer theory, Seidel representation}

\maketitle


\section{Introduction}\label{intro}
Let $G$ be a compact connected semi-simple Lie group. By $G/L$ we mean the coadjoint orbit passing through a point of the dual of its Lie algebra $\gg$ whose associated KSS symplectic form is monotone. Since the $G$-action on $G/L$ is Hamiltonian, it induces a group homomorphism $G\ra \H:=Ham(G/L)$ from $G$ to the group of \Ham diffeomorphisms of $G/L$.
\begin{theorem}\label{A} The dimension of the kernel of the natural map
\[ \pi_*(G)\otimes\QQ\ra\pi_*(\H)\otimes\QQ\]
is at most the rank of $L/Z(L)$ where $Z(L)$ is the center of $L$. In particular, this map is injective if $L$ is equal to a maximal torus in $G$.
\end{theorem}

\begin{remark}\label{kedra} The case when $L$ is equal to a maximal torus in $G$ is a corollary of a result of K\k{e}dra \cite{K} based on the previous work of Reznikov \cite{Reznikov}, K\k{e}dra-McDuff \cite{KM} and Gal-K\k{e}dra-Tralle \cite{GKT}. Notice that his result does not imply our general case because it holds only for coadjoint orbits lying in a Zariski open subset of $\gg^{\vee}$.
\end{remark}

Our second result is about the Hofer geometry of $\H$ \cite{Hofer}. Let $\{\varphi_t\}_{t\in [0,1]}$ be a loop in $\H$. Then $\{\varphi_t\}$ has a \textit{generating Hamiltonian} which is a smooth family $\{H_t\}_{t\in [0,1]}$ of Hamiltonians $H_t:G/L\ra \RR$ satisfying
\[ \dot{\varphi}_t = X_{H_t}\circ \varphi_t.\]
Such a family is unique if we further impose the normalization condition
\[ \int_{G/L} H_t\w^{\frac{1}{2}\dim G/L} =0\]
for any $t$. Define
\[L^+(\{\varphi_t\}):= \int_0^1\max_{G/L} H_t ~dt.\]
This defines a length functional on $\O\H$, called the \textit{max-length functional} \cite{S_QCC}. As in Riemannian geometry, a natural problem is the min-max problem for $L^+$. More precisely, given a homology class $\a\in H_*(\O\H;\ZZ)$, what is the infimum of
\[\max_{\G}~ (L^+\circ f)\]
where $f:\G\ra \O\H$ runs over all smooth cycles which represent $\a$? Moreover, is this value attained by an explicit representative? Our theorem deals with homology classes which come from $\O G$. Recall that the additive group $H_*(\O G;\ZZ)$ is completely known: Bott-Samelson \cite{BS} constructed an explicit basis $\{x_q\}_{q\in \Q}$, indexed by the unit lattice $\Q$ of a maximal torus $T$ in $G$ (say contained in $L$), such that each $x_q$ is represented by a cycle, called \textit{Bott-Samelson cycle}, whose domain is a smooth projective variety which has a structure of iterated $\mathbb{P}^1$-bundles.

The following theorem does not hold for all $q\in \Q$ (unless $L=T$) but a subset consisting of those which are, roughly speaking, concentrated near the faces of the Weyl chambers corresponding to $G/L$. The precise notion is \textit{Peterson elements} with respect to the canonical fibration $G/T\ra G/L$ which will be defined in Definition \ref{peterelementdef} below.
\begin{theorem}\label{B} Suppose $q\in \Q$ is a Peterson element with respect to $G/T\ra G/L$. For any smooth cycle $f:\G\ra \O\H$ representing $x_q$, we have
\[\max_{\G}~(L^+\circ f)\geqslant C_q\]
where $C_q:=\max_{G/L}\langle q,-\rangle$ is a constant determined by $q$. Moreover, the equality holds for the associated Bott-Samelson cycle.
\end{theorem}

\begin{remark}\label{savelyev} Theorem \ref{B} is an extension of a result of Savelyev \cite{S_JDG} which deals with the case of $G/T$ and the classes $x_q$ for $q$ lying in the interior of the dominant chamber\footnote{Strictly speaking, the classes $x_q$ he considered are not represented by the Bott-Samelson cycles but the descending submanifolds with respect to the standard energy functional on $\O G$. Yet, the combinatorics describing these bases are the same, namely the affine Weyl group associated to $\O G$.}.
\end{remark}

As is well-known, the standard tool for solving problems of the above types is \textit{Seidel morphisms} or their variants. See for example the work \cite{ASLag, MSl, S_JDG}. Since higher dimensional cycles in $\O G$ enter our situation, we consider Savelyev's generalization of Seidel morphisms which is a ring map\footnote{To avoid unnecessary technical issues, our map $\Se{}$ is defined only on the subring of $H_{*}(\O Ham(X,\w);\ZZ)$ generated by classes represented by smooth cycles.}
\[ \Se{}:H_{*}(\O Ham(X,\w);\ZZ)\ra QH^*(X;\ZZ)\]
associated to any compact monotone symplectic manifold $(X,\w)$ where the source is given the Pontryagin product. Roughly speaking, it is defined by counting pairs $(\vp,u)$ consisting of $\vp\in \O Ham(X,\w)$ which lies in a given cycle as the input of $\Se{}$ and a pseudoholomorphic section $u$ of the \Ham fibration $P_{\vp}(X)$ over $S^2=D_-\cup D_+$ defined by
\begin{equation}\label{Hamfib}
 P_{\vp}(X):= (D_-\times X)\cup (D_+\cup X)/(e^{-i\tt},x)\sim (\eit,\vp(\eit)\cdot x)
\end{equation}
where $D_{\pm}$ are two copies of the unit disk which are glued to form the 2-sphere $S^2$. To prove Theorem \ref{A}, we introduce a Novikov ring $\ZZ[\Nov{X}{\PP}]$ associated to any group homomorphism $\vp:\PP\ra \O Ham(X,\w)$ where $\PP$ is possibly infinite dimensional, defined as the group algebra of the group $\Nov{X}{\PP}$ of homotopy classes of pairs $(p,u)$ consisting of $p\in \PP$ and a section $u$ of $P_{\vp_p}(X)$. The group structure on $\Nov{X}{\PP}$ is given by the standard gluing operation (e.g. \cite{S_QCC})
\[ \pi_2^{section}(P_{\vp_1}(X)) \times\pi_2^{section}(P_{\vp_2}(X))\ra \pi_2^{section}(P_{\vp_1\circ\vp_2}(X))   \]
for any $\vp_1,\vp_2\in \O Ham(X,\w)$ where $\pi_2^{section}$ denotes the set of section classes and $\circ$ the pointwise composition. It is straightforward to modify the definition of $\Se{}$ to obtain a ring map
\[ \Se{\PP}: H_{-*}(\PP;\ZZ)\ra QH^*(X;\ZZ[\Nov{X}{\PP}]).\]
The key point of the proof of Theorem \ref{A} and Theorem \ref{B} is to have good knowledge of $\Se{\O G}$ for $X=G/L$. Our strategy consists of three steps.

\vspace{0.3cm}
\begin{mystep}{1} We relate $\Se{\O G}$ to another ring map constructed by a different theory. Suppose $(X,\w)$ admits a Hamiltonian $G$-action for a compact connected Lie group $G$. Introduced by Weinstein \cite{W}, the \textit{moment correspondence} associated to $(X,\w)$ is defined by
\begin{equation}\label{moment}
C:=\{(g,\mu(x),x,g\cdot x)|~g\in G, x\in X\}\subset T^*G\times X\times X
\end{equation}
where $\mu:X\ra \gg^{\vee}$ is the moment map, and $T^*G$ is identified with $G\times \gg^{\vee}$ by left multiplication. It is a Lagrangian correspondence from $T^*G$ to $X^-\times X$ which geometrically composes with the cotangent fiber $L:=T_e^*G$ at the identity element $e\in G$ to give the diagonal $\D$. By Ma'u-Wehrheim-Woodward's \textit{quilted Floer theory} \cite{MWW} and the work of Evans-Lekili \cite{EL}, $C$ induces a ring homomorphism
\begin{equation}\label{MC}
\MC: HW^*(L,L)\ra HF^*(\D,\D).
\end{equation}
It is well-known that the source of $\MC$ has a topological model: Abbondandolo-Schwarz \cite{AS} constructed a ring isomorphism
\begin{equation}\label{AS}
\FF: H_{-*}(\O G)\ra HW^*(L,L).
\end{equation}
See also the work of Abouzaid \cite{Ab_JSG} who constructed the inverse of $\FF$. As for the target of $\MC$, we have the ring isomorphism of Piunikhin-Salamon-Schwarz \cite{PSS}
\[PSS: HF^*(\D,\D)\ra QH^*(X).\]
Since we have enlarged $QH^*(X)$ by tensoring it with $\ZZ[\Nov{X}{\O G}]$, we enlarge $HF^*(\D,\D)$ by introducing a natural notion of capping disks and modify the definition of $\MC$ and $PSS$ correspondingly. See Section \ref{cappingdisk} for the details.
\begin{theorem}\label{MC=Se} $PSS\circ \MC\circ \FF=\Se{\O G}$.
\end{theorem}

The proof, given in Section \ref{proveMC=Se}, is to consider the closed string analogue of the problem which is nothing but a special case of (a family version of) a theorem of Wehrheim-Woodward \cite{WW3}, proved by \textit{annulus-shrinking} \cite{WW2}, stating that quilted invariants are compatible with geometric compositions. The open string case then follows from the closed string case, since the open-closed map $\mathcal{OC}:HF^*(\D,\D)\ra QH^*(X^-\times X)$ is injective.
\end{mystep}

\bigskip
\begin{mystep}{2} We compute $PSS\circ \MC\circ \FF$ for $X=G/T$. It is based on  the analysis of $\FF$ and the work of Bae-Chow-Leung \cite{BCL} which computed $PSS\circ \MC$. Let us define some notations before stating the result. Recall $G$ is a compact connected Lie group which is semi-simple\footnote{That means the center $Z(G)$ of $G$ is discrete.} and $T$ is a maximal torus in $G$. Put $\gg:=\Lie(G)$ and $\t:=\Lie(T)$. Fix an $\Ad$-invariant metric $\langle -,-\rangle$ on $\gg$. Denote by $W$ the Weyl group of the pair $(G,T)$. It is well-known that $W$ acts simply transitively on the set of Weyl chambers in $\t$. Fix one of these chambers $\wc_0$ and call it \textit{dominant}. Denote by $\Q\subset \t$ the unit lattice of the torus $T$. For any $q\in\Q$, there is a unique element $w_q\in W$ such that $w_q\wc_0$ is the Weyl chamber the segment from $q$ to any interior point of $\wc_0$ first hits.

To describe what $PSS\circ \MC\circ \FF$ looks like, we specify a basis for each of the sources and targets as follows. As mentioned above, the group $H_*(\O G;\ZZ)$ has a basis $\{x_q\}_{q\in\Q}$ where each $x_q$ is represented by the Bott-Samelson cycle $\BS_q$ associated to $q$. See Definition \ref{BSq} for the construction of $\BS_q$. As for $QH^*(G/T;\ZZ[\Nov{G/T}{\O G}])$, it is well-known that, over $\ZZ$ but not $\ZZ[\Nov{G/T}{\O G}]$, it has a basis consisting of Schubert classes $\s_w$, $w\in W$. Each $\s_w$ is the Poincar\'e dual of the stable (i.e. descending) submanifold passing through the critical point $x_w$ of the Morse function $\langle a,-\rangle$ on $G/T$ for an element $a\in\mr{\wc}_0$, where $x_w:=w(x_0)$ and $x_0$ is the unique intersection point of $G/T$ and $\wc_0$. By Lemma \ref{GPNovforOO},
\[\Nov{G/T}{\O G}=\left\{ \left. A_q^{G/T}:=[\vp_q, s_{\vp=\vp_q,u_{\pm}\equiv x_0}]\right|~q\in\Q\right\}\]
where $\vp_q:\eit\mapsto \exp(\tt\cdot q/2\pi)\in G$ and $s_{\vp,u_{\pm}}$ is the section of $P_{\vp}(G/T)$ defined in Section \ref{Hamfibsection}. It follows that $QH^*(G/T;\ZZ[\Nov{G/T}{\O G}])$ has a basis $\{\s_w \TT^{A_q^{G/T}}\}_{(w,q)\in W\times \Q}$ where $\TT$ is the Novikov variable.

\begin{theorem}(=Theorem \ref{computeG/T}) \label{computeG/Tintro} For any $q\in\Q$, we have
\[(PSS\circ\MC\circ\FF)(x_q)\in\pm\s_{w_q}\TT^{A^{G/T}_{w_q^{-1}(q)}}+\bigoplus_{\substack{w'\in W\\ \ell'(w')<\ell'(w_q)}}\ZZ[\Nov{G/T}{\O G}]\cdot \s_{w'}\]
where $\ell':W\ra \RR$ is a function defined in Definition \ref{perturblengthdef}\footnote{Roughly speaking, it is a small perturbation of the standard length function $\ell$ on $W$.}
\end{theorem}
\end{mystep}

\begin{mystepf} We carry out the computation for general coadjoint orbits $G/L$. This time, we do not examine $PSS\circ \MC\circ\FF$ but prove a theoretical result expressing $\Se{\O G}$ for $G/L$ in terms of $\Se{\O G}$ for $G/T$. This will give what we want since we know enough about the latter map, by Theorem \ref{MC=Se} and Theorem \ref{computeG/Tintro}. As we will deal with $\Se{\O G}$ for more than one Hamiltonian $G$-manifolds, we use $\See{G/L}{\O G}$ in place of $\Se{\O G}$ if we talk about $G/L$, etc. By $G/L$, we mean the coadjoint orbit passing through a point $y_0\in\wc_0$ (while $G/T$ is the coadjoint orbit passing through an interior point $x_0\in\mr{\wc}_0$). The points $x_0$ and $y_0$ are required to lie in certain rays in order for the resulting coadjoint orbits to be monotone. See Section \ref{fromG/T} for more detail.

There is a unique $G$-equivariant \Ham fibration $\pi:G/T\ra G/L$ sending $x_0$ to $y_0$. Given $\vp\in \O G$. Since $\pi$ is $G$-equivariant, it induces a map $\pi_{\vp}:P_{\vp}(G/T)\ra P_{\vp}(G/L)$ commuting with the projections onto $S^2$. For any section $u$ of $P_{\vp}(G/L)$, the pre-image $\pi_{\vp}^{-1}(\im(u))$ of the image of $u$ with respect to $\pi_{\vp}$ is a \Ham fibration over $\im(u)\simeq S^2$ with fibers isomorphic to $L/T$. Moreover, any section of this fibration induces a section of $P_{\vp}(G/T)$ via the inclusion. Notice that everything we have just introduced is defined up to homotopy. As the first step towards the computation of $\See{G/L}{\O G}$, we solving the following

\begin{problem} \label{peterG/Lask} For any $A\in \Nov{G/L}{\O G}$, determine $\psi_A\in\pi_0(\O Ham(L/T))$ such that
\[ \pi_{\vp_A}^{-1}(\im(u_A))\simeq P_{\psi_A}(L/T)\]
for a (and hence all) representative $(\vp_A,u_A)$ of $A$. Determine also the map
\[ \PP_A:\pi_2^{section}(P_{\psi_A}(L/T))\ra \pi_2^{section}(P_{\vp_A}(G/T)).\]
\end{problem}

\noindent We answer Problem \ref{peterG/Lask} here and give a complete proof in Section \ref{fromG/T}. Denote by $R$ the set of roots for the pair $(G,T)$. For any $y\in\t$, define $R_y:=\{\a\in R|~\a(y)=0\}$, $\t_y:=\bigcap_{\a\in R_y}\{\a=0\}$,
\begin{align*}
\Q_{R_y} &:= \Span_{\ZZ}\{\a^{\vee}|~\a\in R_y\}\\
P_{R_y}^{\vee}&:= \{q\in \Q_{R_y}\otimes \RR|~\a(q)\in\ZZ\text{ for all }\a\in R_y\}
\end{align*}
where $\a^{\vee}:=\frac{2\a}{\langle\a,\a\rangle}\in\t$ is the coroot associated to a root $\a\in R$. Notice that $\Q_{R_y}\otimes\RR$ is equal to the orthogonal complement $\t_y^{\perp}$ of $\t_y$ in $\t$. Denote by $\pi_{\t_y^{\perp}}:\t\ra\t_y^{\perp}$ the orthogonal projection. By Lemma \ref{GPNovforOO},
\[\Nov{G/L}{\O G}=\left\{\left. A_{q+\Q_{R_{y_0}}}^{G/L}:=[\vp_q, s_{\vp=\vp_q, u_{\pm}\equiv y_0}] \right|~q+\Q_{R_{y_0}}\in\Q/\Q_{R_{y_0}}\right\}\]
where $\vp_q:\eit\mapsto \exp(\tt\cdot q/2\pi)\in G$. Observe that the fiber $L/T$ is a coadjoint orbit of the group $L^{\ad}:=L/Z(L)$ which has a maximal torus $\exp(\t_{y_0}^{\perp})$ with unit lattice $P^{\vee}_{R_{y_0}}$. By the same lemma, we have
\[ \Nov{L/T}{\O L^{\ad}}=\left\{ \left. A^{L/T}_q \right|~q\in P^{\vee}_{R_{y_0}}\right\}\]
where $A^{L/T}_q$ defined similarly as above.
\begin{proposition}(=Proposition \ref{peterG/L}) \label{peterG/Lintro} For any $q+\Q_{R_{y_0}}\in\Q/\Q_{R_{y_0}}$,
\[ \psi_{A'} = \pi_{\t_{y_0}^{\perp}}(q) + \Q_{R_{y_0}}\in P^{\vee}_{R_{y_0}}/\Q_{R_{y_0}}\simeq \pi_1(L^{\ad})\simeq \pi_0(\O L^{\ad}) \]
where $A':=A^{G/L}_{q+\Q_{R_{y_0}}}$, and for any $\tilde{q}\in P^{\vee}_{R_{y_0}}$ with $\tilde{q}\in\pi_{\t_{y_0}^{\perp}}(q)+\ Q_{R_{y_0}}\simeq \pi_2^{section}(P_{\psi_{A'}}(L/T))$,
\[ \PP_{\psi_{A'}}\left(A^{L/T}_{\tilde{q}}\right)= A^{G/T}_{\tilde{q}-\pi_{\t_{y_0}^{\perp}}(q)+q}.\]
\end{proposition}

\noindent Proposition \ref{peterG/Lintro} is necessary for the following key argument which expresses $\See{G/L}{\O G}$ in terms of $\See{G/T}{\O G}$. It is a version of Leray spectral sequences in Floer theory (for example, the work of Oancea \cite{SSFloer1}). We show that, by choosing perturbation data for $P_{\vp}(G/T)$ and $P_{\vp}(G/L)$ suitably, the signed count of \holo sections of $P_{\vp}(G/T)$ yielding a term in $\See{G/T}{\O G}$ is equal to the signed count for a term in $\See{G/L}{\O G}$ times the signed count for a term in $\See{L/T}{\O L^{\ad}}$. The key point is that the signed count for $\See{L/T}{\O L^{\ad}}$ is $\pm 1$ if we make specific choices of section classes in $\Nov{G/T}{\O G}$ and $\Nov{G/L}{\O G}$ which is done with help of Proposition \ref{peterG/Lintro}.

The outcome is the following. Recall there are also Schubert classes $\s_{wW_{y_0}}$ for $G/L$ indexed by the cosets $wW_{y_0}$ in $W/W_{y_0}$ where $W_{y_0}$ is the Weyl group of $L$. Fix $a\in \mr{\wc}_0$ which is sufficiently close to the origin. Define
\[ \deg^{L/T}(q):= \sum_{\substack{\a\in R_{w_q(y_0)}\\ \a(w_q(x_0))>0}} \lfloor \a(q+a)\rfloor.\]

\begin{theorem}(=Theorem \ref{computeG/L}) \label{computeG/Lintro} Let $q\in \Q$.
\begin{enumerate}[(a)]
\item If $\deg^{L/T}(q)=0$, then
\[ \See{G/L}{\O G}(x_q) = \pm \s_{w_qW_{y_0}} \TT^{A^{G/L}_{w_q^{-1}(q)+\Q_{R_{y_0}}}}+\cdots \]
where $\cdots$ is a finite sum of terms which do not cancel with the first term.
\item If $\deg^{L/T}(q)=0$ and $w_q=\ii$, then
\[ \See{G/L}{\O G}(x_q) = \pm  \TT^{A^{G/L}_{q+\Q_{R_{y_0}}}}.\]
\end{enumerate}
\end{theorem}

\begin{definition}\label{peterelementdef} Any element $q\in\Q$ with $\deg^{L/T}(q)=0$ is called a \textit{Peterson element}.
\end{definition}

\begin{remark}
The term ``Peterson element'' comes from Peterson's conjectural formula \cite{Peter}, proved by Woodward \cite{Wformula}, relating the quantum cohomology of $G/T$ to that of $G/L$. An alternative proof based on the techniques developed in this paper is given in Appendix \ref{PWproof}. Peterson also conjectured an explicit isomorphism (after localization) between $QH^*(G/T;\ZZ[\Nov{G/T}{\O G}])$ and the Pontryagin ring $H_{-*}(\O G;\ZZ)$. This conjecture was proved by Lam-Shimozono \cite{LS}. In the same paper, they also obtained a similar result for general $G/L$\footnote{In this case, the source is replaced with the quotient of $H_{-*}(\O G;\ZZ)$ by the ideal generated additively by $x_q$ with $\deg^{L/T}(q)\ne 0$.} using Peterson-Woodward's formula. Theorem \ref{computeG/Lintro} should be considered parallel to their argument of obtaining the case of $G/L$ from that of $G/T$.
\end{remark}
\end{mystepf}

We now prove Theorem \ref{A} and Theorem \ref{B}, assuming the results stated in our strategy.
\begin{myproof}{Theorem}{\ref{A}} Since $\pi_1(G)$ is finite, we may assume $G$ is simply connected without affecting the result. Denote by $\mathcal{K}$ the universal covering group of the identity component of $\O\HH$. Since $\pi_0(\O G)\simeq \pi_1(\O G)=0$, the group homomorphism $\O G\ra \O\HH$ has a lift $\O G\ra\mathcal{K}$. By the same reason, it suffices to prove the result for the induced map $\pi_*(\O G)\otimes \QQ\ra\pi_*(\mathcal{K})\otimes \QQ$. We have the commutative diagram
\begin{equation}\label{pfA}
\begin{tikzcd}
 \pi_{-*}(\O G)\otimes \QQ \arrow{r}   \arrow{d}[left]{hur_{\O G}\otimes \QQ} &[1.5em]  \pi_{-*}(\mathcal{K})\otimes \QQ\arrow{d}{hur_{\mathcal{K}}\otimes \QQ}   \\
  H_{-*}(\O G;\QQ) \arrow{r}   \arrow{d}[left]{\See{G/L}{\O G}} &[1.5em]  H_{-*}(\mathcal{K};\QQ) \arrow{d}{\See{G/L}{\mathcal{K}}}  \\
QH^*(G/L;\QQ[\Nov{G/L}{\O G}])   \arrow{r}  &[1.5em] QH^*(G/L;\QQ[\Nov{G/L}{\mathcal{K}}])
\end{tikzcd}
\end{equation}
where the horizontal arrows are some natural homomorphisms and the vertical arrows in the upper square are the rational Hurewicz maps. The result then follows from
\begin{enumerate}[(i)]
\item $H_*(\O G;\QQ)$ is the symmetric algebra of $\pi_*(\O G)\otimes\QQ$ with canonical map $hur_{\O G}\otimes \QQ$.
\item Theorem \ref{computeG/Lintro}(b) which implies that the transcendence degree of the image of $\See{G/L}{\O G}$ is at least $\rk(G)-\rk(L^{\ad})$ where $\rk$ denotes the rank of a compact Lie group.
\item The bottom horizontal arrow of \eqref{pfA} is injective. (This is the reason why we do not consider $\O\HH$ but $\mathcal{K}$.)
\item The dimension of $\pi_*(\O G)\otimes\QQ$ is equal to $\rk(G)$.
\end{enumerate}
\end{myproof}

\begin{myproof}{Theorem}{\ref{B}} Let $q\in\Q$. Suppose $\deg^{L/T}(q)=0$. By Theorem \ref{computeG/Lintro}(a), $\See{G/L}{\O G}(x_q)$ contains a non-zero term $\pm\s_{w_qW_{y_0}}\TT^{A'}$ where $A':=A^{G/L}_{w_q^{-1}(q)+\Q_{R_{y_0}}}$. Let $f:\G\ra \O\HH$ be a smooth cycle representing $x_q$. By a standard argument (see \cite{S_JDG}), we have
\[ \max_{\G}~(L^+\circ f)\geqslant - \CCC(A')\]
where $\CCC$ is the coupling class\footnote{
The coupling class $\CCC\in H^2(P_{\vp}(X);\RR)$ of any \Ham fibration $P_{\vp}(X)$ is characterized by (1) $\CCC$ restricts to $[\w_X]$ in the fibers; and (2) $\pi_*\CCC^{\frac{1}{2}\dim X+1}=0$. The existence and uniqueness of $\CCC$ are proved in \cite{GScf, MScf}.} which acts naturally on $\Nov{G/L}{}$. An straightforward computation gives
\[ \CCC(A') = -\langle w_q^{-1}(q),y_0 \rangle = - \langle q,w_q(y_0) \rangle.\]
To determine $C_q:=\max_{G/L}\langle q,-\rangle$, observe that the function $\langle q,-\rangle $ is Bott-Morse and its critical values can be determined by looking at its restriction to the intersection $\t\cap G/L=W\cdot y_0$. It is then easy to deduce that the maximum value is $\langle q,w_q(y_0) \rangle$. This proves the first part of the theorem.

It remains to show that the Bott-Samelson cycle $\BS_q$ is a minimizer of $L^+$. By definition, the image of $\BS_q$ consists of broken geodesics in $G$ each obtained by applying the $G$-action to different portions (with respect to a fixed partition) of the loop $\vp_q:t\mapsto \exp(tq)$, without tearing off the curve. Such an operation does not change the value of $L^+$. It follows that $L^+\circ \BS_q$ is constantly equal to $L^+(\vp_q)$. Notice that the normalized generating \Ham for $\vp_q$ is given by $H_t\equiv \langle q,-\rangle$, and we have seen that its maximum value is equal to $\langle q,w_q(y_0) \rangle$ so that $L^+(\vp_q)=\langle q,w_q(y_0) \rangle$.
\end{myproof}

\vspace{.2cm}
Finally, we point out that Theorem \ref{A} and Theorem \ref{B} hold for non-monotone coadjoint orbits as long as one is able to define $\Se{\O Ham(X,\w)}$ for these spaces, most likely by virtual theory. Indeed $\See{G/L}{\O G}$ is well-defined, without applying virtual theory, even for non-monotone $G/L$. This relies on the existence of a $G$-equivariant integrable almost complex structure on $G/L$ which depends only on the topology of $G/L$. It is compatible with any KSS forms and satisfies that all moduli spaces of genus-zero stable curves are regular and that the evaluation maps at any points of the domain curves are submersions. Using the fact that $\See{G/L}{\O G}$ is invariant under deformation of symplectic forms, we see that it is independent of which KSS form used. Since Theorem \ref{computeG/Lintro} holds for a particular form, namely the monotone one, we conclude that the same theorem holds for all coadjoint orbits.
\section*{Organization of the paper}
We prove Theorem \ref{MC=Se} in Section \ref{proveMC=Se}, after recalling the definition of Savelyev-Seidel and Ma'u-Wehrheim-Woodward morphisms in Section \ref{SSmor} and Section \ref{QFTsection} respectively. In Section \ref{OG}, we recall the necessary materials from Lie theory including coadjoint orbits and Bott-Samelson cycles, and prove Theorem \ref{computeG/Tintro} and \ref{computeG/Lintro}. In Appendix \ref{wrapped}, we give a brief summary of the wrapped Floer theory on cotangent bundles. In Appendix \ref{gluing}, we recall a gluing result of Wehrheim-Woodward in Appendix \ref{gluing}. In Appendix \ref{PWproof}, we prove Peterson-Woodward's comparison formula.
\section*{Acknowledgements}
I would like to thank my supervisor Conan Leung for his encouragement and continuous support; Changzheng Li for bringing my attention to his paper with Yongdong Huang and a paper of Peter Magyar; Yasha Savelyev for answering my questions about his papers and invaluable comments on the draft of this paper; and Mohammed Abouzaid, Cheuk Yu Mak and Jarek K\k{e}dra for useful discussions.
\section{Savelyev-Seidel morphisms}\label{SSmor}
\subsection{Hamiltonian fibrations and their sections}\label{Hamfibsection}
Let $(X,\w_X)$ be a compact monotone symplectic manifold. For any loop $\vp\in\Lo Ham(X,\w_X)$, there is an associated \Ham fibration $P_{\vp}(X)$ over $S^2$ with fibers $(X,\w_X)$ defined by
\[ P_{\vp}(X):= (D_-\times X)\cup (D_+\cup X)/(e^{-i\tt},x)\sim (\eit,\vp(\eit)\cdot x)\]
where $D_{\pm}$ are two copies of the unit disk which are glued to form the 2-sphere $S^2$. A typical way of constructing sections of $P_{\vp}(X)$ is as follows: Given a pair $u_{\pm}:D\ra X$ satisfying
\begin{equation}\label{pair}
u_+(\eit)=\vp(\eit)\cdot u_-(\eit)\quad \text{for any }\tt,
\end{equation}
define $s_{\vp,u_{\pm}}$ to be the section of $P_{\vp}(X)$ by
\[ s_{\vp,u_{\pm}}(z):=\left\{
\begin{array}{rl}
u_-(\overline{z})& z\in D_-\subset S^2\\
u_+(z)&z\in D_+\subset S^2
\end{array}
\right. .\]
This gives a bijective correspondence between the set of sections of $P_{\vp}(X)$ and the set of pairs $u_{\pm}$ satisfying \eqref{pair}.

\begin{definition} \label{sectionclassdef} Let $\vp:\PP\ra \Lo Ham(X,\w_X)$ be a group homomorphism where $\PP$ is a Fr\'echet Lie group. Define $\Nov{X}{\PP}$ to be the set of homotopy classes of pairs $(p,u)$ consisting of $p\in \PP$ and a section $u$ of $P_{\vp_p}(X)$. It is naturally endowed with a group structure defined by the standard gluing construction
\[ \pi_2^{section}(P_{\vp_1}(X)) \times \pi_2^{section}(P_{\vp_2}(X)) \ra \pi_2^{section}(P_{\vp_1\circ\vp_2}(X))\]
for any $\vp_1,\vp_2\in \Lo Ham(X,\w_X)$ where $\pi_2^{section}$ denotes the set of section classes and $\circ$ the pointwise composition. See \cite{S_QCC} for more detail.
\end{definition}
\subsection{The definition} We recall the definition of Savelyev-Seidel morphism
\[\Se{\PP}: H_{-*}^{cycle}(\PP;\ZZ)\ra QH^*(X;\ZZ[\Nov{X}{\PP}])\]
where the domain is the subring generated by homology classes which are represented by smooth cycles. Let $\ppp:\G\ra \PP$ be a smooth cycle, giving rise to a smooth family $\{P_{\vp_{\ppp(\g)}}(X)\}_{\g\in \G}$ of \Ham fibrations over $S^2$ with fibers $(X,\w_X)$. Denote by $\I_{\G}(X)$ the space of families $I=\{I^{\g}\}_{\g\in\G}$ of $\w_X$-compatible almost complex structures $I^{\g}$ on $X$, and by $\J_{\ppp}(X)$ the space of families $J=\{J^{\g,z}\}_{(\g,z)\in \G\times S^2}$ of $\w_X$-compatible almost complex structures on the fibers of $\{P_{\vp_{\ppp(\g)}}(X)\}_{\g\in \G}$. For any $J\in\J_{\ppp}(X)$ and any families $\HH=\{\HH^{\g}\}_{\g\in\G}$ of \Ham connections on $\{P_{\vp_{\ppp(\g)}}(X)\}$, denote by $J_{\HH}=\{J^{\g}_{\HH}\}_{\g\in\G}$ the family of almost complex structures on $\{P_{\vp_{\ppp(\g)}}(X)\}$ characterized by the conditions
\begin{itemize}
\item each $J^{\g}_{\HH}$ restricts to $J^{\g,z}$ on each fiber;
\item the horizontal distributions $\HH^{\g}$ are preserved by $J^{\g}_{\HH}$; and
\item the projection $(P_{\vp_{\ppp(\g)}}(X), J^{\g}_{\HH})\ra (S^2,j)$ is holomorphic for any $\g\in \G$, where $j$ is the standard complex structure on $S^2$.
\end{itemize}

\begin{definition} Given $A\in \pi_2(X)$ and $I=\{I^{\g}\}\in \I_{\G}(X)$, define $\M^{simple}(A,\G,I)$ to be the moduli space of pairs $(\g,u)$ with $\g\in \G$ and $u:S^2\ra X$ satisfying
\[ \left\{
\begin{array}{ll}
u\text{ is simple and }(j,I^{\g})\text{-holomorphic; and}\\
u_*[S^2]=A.
\end{array}
\right.\]
\end{definition}
\noindent Notice that we do not take the quotient by the automorphism group of $(S^2,j)$. Fix two distinct marked points on $S^2$, call them ``in'' and ``out'', and denote by
\[ev_{in}, ev_{out} :\M^{simple}(A,\G,I)\ra X\]
the evaluation maps at these two marked points.

\begin{definition} Given $A_1,\ldots,A_k\in \pi_2(X)$, define
\begin{align*}
& \M^{chain}(A_1,\ldots,A_k,\G,I)\\
\subseteq~ & \D_{\G}\times_{\G^k}\left( \M^{simple}(A_1,\G,I)\times_{(ev_{out},ev_{in})} \cdots\times_{(ev_{out},ev_{in})}\M^{simple}(A_k,\G,I) \right)
\end{align*}
to be the open subset consisting of chains of simple holomorphic spheres in $X$ any two components of which have distinct images, where $\D_{\G}$ is the small diagonal of $\G^k$.
\end{definition}
\noindent We denote by
\[ev_{in},ev_{out}:\M^{chain}(A_1,\ldots,A_k,\G,I)\ra X\]
the evaluation maps at the marked point ``in'' of the component in $\M^{simple}(A_1,\G,I)$ and the marked point ``out'' of the component in $\M^{simple}(A_k,\G,I)$ respectively.

\begin{definition} Given $A\in \Nov{X}{\PP}$, $J\in\J_{p}(X)$ and families $\HH=\{\HH^{\g}\}_{\g\in\G}$ of \Ham connections on $\{P_{\vp_{\ppp(\g)}}(X)\}$, define $\M^{section}(A,\ppp,\HH,J)$ to be the moduli space of pairs $(\g,u)$ where $\g\in\G$ and $u:S^2\ra P_{\vp_{\ppp(\g)}}(X)$ satisfying
\[ \left\{
\begin{array}{ll}
u\text{ is a section and is }(j,J^{\g}_{\HH})\text{-holomorphic; and}\\
A=[\vp_{\ppp(\g)},u].
\end{array}
\right.\]
\end{definition}
\noindent We denote by
\[ev_0:\M^{section}(A,\ppp,\HH,J)\ra X \]
the map sending each $(\g,u)$ to $u(0)$ ($0\in S^2$) where the target $X$ is the fiber of $P_{\vp_{\ppp(\g)}}(X)$ over $0$.

Let $h:N\ra X$ be a pseudocycle and $(\HH,J)$ be given.
\begin{condition} \label{GPdefSeidel} For any $A_0\in\Nov{X}{\PP}$ and $A_1,\ldots,A_k\in\pi_2(X)$ satisfying
\[\dim\G +\dim N + 2c_1^v(A_0)+2\sum_{i=1}^k c_1(A_i)\leqslant 0\]
(where $c_1^v$ is the vertical first chern class), the fiber product
\[ \D_{\G}\times_{\G^2} \left( \M^{section}(A_0,\ppp,\HH,J)\times_{(ev_0,ev_{in})}\M^{chain}(A_1,\ldots,A_k,\G,\{J^{\g,0}\})\times_{(ev_{out},h)}N\right),\]
as well as the similar fiber product with $h$ replaced by a given smooth map (whose domain has dimension at most $\dim N-2$) covering the limit set\footnote{It is a terminology about pseudocycles. See \cite{MS} for the definition.} of $h$, is regular.
\end{condition}

\noindent The following proposition is standard.
\begin{proposition} For generic $(\HH,J)$, Condition \ref{GPdefSeidel} is satisfied. Moreover, the signed count of the fiber product $\M^{section}(A,\ppp,\HH,J)\times_{(ev_0,h)}N$ for any $A\in \Nov{X}{\PP}$ satisfying $\dim\G+\dim N+2c_1^v(A)=0$ is well-defined for and independent of any $(\HH,J)$ satisfying Condition \ref{GPdefSeidel}. \hfill$\square$
\end{proposition}

\noindent Choose a pair $(\HH,J)$ satisfying Condition \ref{GPdefSeidel}. Let $h:N\ra X$ be a pseudocycle.
\begin{definition} Define
\[\langle \Se{\PP}([\ppp]),[h]\rangle := \sum_{ \substack{A\in\Nov{X}{\PP} \\ \dim\G+\dim N+2c_1^v(A)=0 } } \langle \Se{\PP}([\ppp]),[h]\rangle_A ~\TT^A \]
where $\langle -,-\rangle$ is the natural pairing $H^*(X)\otimes H_*(X)\ra \ZZ$, $T$ is the Novikov variable of $\ZZ[\Nov{X}{\PP}]$ and
\[\langle \Se{\PP}([\ppp]),[h]\rangle_A:=\#\left(\M^{section}(A,\ppp,\HH,J) \times_{(ev_0,h)}N \right).\]
This defines an element $\Se{\PP}([\ppp])\in QH^*(X;\ZZ[\Nov{X}{\PP}])$. By standard cobordism arguments, $\Se{\PP}([\ppp])$ is independent of which smooth cycle $\ppp$ representing the same homology class of $H_{-*}(\PP;\ZZ)$.
\end{definition}
\subsection{A Leray-type spectral sequence}\label{Leray}
Let $F\hookrightarrow X\xrightarrow{\pi} Y$ be a fiber bundle where $F$, $X$ and $Y$ are compact simply connected smooth manifolds. Assume there exist symplectic forms $\w_X$ and $\w_Y$ on $X$ and $Y$ respectively such that $(X,\w_X)$ and $(Y,\w_Y)$ are monotone, and every fiber of $\pi$ is a symplectic submanifold of $(X,\w_X)$. Then $\pi:(X,\w_X)\ra (Y,\w_Y)$ is a Hamiltonian fibration. Denote by $\w_F$ the induced symplectic form on any particular fiber $F$.

Let $G$ be a compact connected Lie group. Suppose there are \Ham $G$-actions on $X$ and $Y$ respectively such that the projection $\pi:X\ra Y$ is $G$-equivariant.
Since $\pi$ is $G$-equivariant, it induces, for any $\vp\in\O G$, a bundle map $\pi_{\vp}:P_{\vp}(X)\ra P_{\vp}(Y)$ which restricts to $\pi$ on every fiber, and hence a map $\pi_*:\Nov{X}{\O G}\ra \Nov{Y}{\O G}$.

Let $A\in \Nov{Y}{\O G}$. Choose a representative $(\vp_A,u_A)$ of $A$ where $\vp_A\in \O G$ and $u_A$ is a smooth section of $P_{\vp_A}(Y)$. Then the pre-image $\pi_{\vp_A}^{-1}(\im(u_A))$ is a \Ham fibration over $\im(u_A)\simeq S^2$ with fibers $(F,\w_F)$\footnote{The \Ham connection is induced by a connection 1-form on $P_{\vp_A}(X)$.} so it is isomorphic to $P_{\psi_A}(F)$ for some $\psi_A\in \O Ham(F,\w_F)$. By an abuse of notation, we denote by the same symbol the connected component of $ \O Ham(F,\w_F)$ which contains the loop $\psi_A$. Moreover, every section of $P_{\psi_A}(F)$ induces a section of $P_{\vp_A}(X)$ via the inclusion. Thus we get a map
\[ \PP_A:\pi_2^{section}(P_{\psi_A}(F))\ra \pi_2^{section}(P_{\vp_A}(X)).\]

\begin{lemma} \label{chernclassformula} For any $A\in \Nov{Y}{\O G}$ and $B\in \pi_2^{section}(P_{\psi_A}(F))$, we have
\[c_1^v(\PP_A(B))=c_1^v(A)+c_1^v(B). \]
\end{lemma}
\begin{proof}
Let $u:S^2\ra P_{\psi_A}(F)$ be a section representing $B$. Recall we have an inclusion $\iota: P_{\psi_A}(F)\hookrightarrow P_{\vp_A}(X)$ commuting with the projections onto $S^2$, and that $\pi_{\vp_A}\circ \iota \circ u$ is a section of $P_{\vp_A}(Y)$ representing $A$. The result follows from the short exact sequence
\[ 0\ra u^*T^vP_{\psi_A}(F)\ra (\iota\circ u)^*T^vP_{\vp_A}(X)\ra (\pi_{\vp_A}\circ \iota \circ u)^*T^v P_{\vp_A}(Y)\ra 0 \]
where $T^v$ denotes the vertical tangent bundle with respect to the fibration structure over $S^2$.
\end{proof}

Now suppose we have a smooth cycle $\vp:\G\ra \O G$, a fiber bundle $N_F\hookrightarrow N_X\ra N_Y$ and smooth cycles $h_F:N_F\ra F$, $h_X:N_X\ra X$ and $h_Y:N_Y\ra Y$ fitting into the commutative diagram
\begin{equation}\nonumber
\begin{tikzcd}
  N_F \arrow[hookrightarrow]{r}   \arrow{d}[left]{h_F} &[1em] N_X \arrow{r}   \arrow{d}[left]{h_X}    &[1em]  N_Y  \arrow{d}[left]{h_Y}  \\
  F \arrow[hookrightarrow]{r}  &[1em] X \arrow{r}{\pi}  &[1em]  Y
\end{tikzcd}.
\end{equation}
\begin{assume} \label{GPassume} There exist smooth families $I_X=\{I^{\g}_X\}_{\g\in\G}\in \I_{\G}(X)$ and $I_Y=\{I^{\g}_Y\}_{\g\in\G}\in \I_{\G}(Y)$ such that
\begin{enumerate}
\item $\pi:(X,I^{\g}_X)\ra (Y,I^{\g}_Y)$ is holomorphic for any $\g\in \G$;
\item for any $A_1,\ldots,A_k\in\pi_2(X)$, the fiber product
\[\M^{chain}(A_1,\ldots,A_k,\G,I_X)\times_{(ev_{out},h_X)}N_X\]
is regular; and
\item the similar regularity condition holds for $Y$.
\end{enumerate}
\end{assume}

\noindent The following theorem allows us to express $\Se{\O G}$ for $G/L$ in terms of $\Se{\O G}$ for $G/T$. Since it deals with more than one symplectic manifolds at the same time, we drop the subscript in $\Se{\O G}$ and replace it by the space for which the morphism is defined.
\begin{theorem} \label{GPmain} Under Assumption \ref{GPassume} and the assumption that the $G$-actions on $X$ and $Y$ are transitive, for any $A_X\in \Nov{X}{\O G}$ and $A_Y\in \Nov{Y}{\O G}$ such that $\pi_*A_X=A_Y$ and
\[\dim\G+\dim N_X+2c_1^v(A_X)= \dim\G+\dim N_Y+2c_1^v(A_Y)=0,\]
we have
\[ \langle \See{X}{\O G}([\vp]),[h_X] \rangle_{A_X} = \langle\See{Y}{\O G}([\vp] ),[h_Y]  \rangle_{A_Y}\cdot \sum_{ \substack{ B\in \pi_2^{section}(P_{\psi_{A_Y}}(F))\\ \PP_{A_Y}(B)=A_X }} \langle \See{F}{\O G}(\psi_{A_Y}),[h_F] \rangle_B  .\]
\end{theorem}

\noindent The proof of Theorem \ref{GPmain} relies on choosing perturbation data on $\{P_{\vp_{\g}}(X)\}_{\g\in\G}$ and $\{P_{\vp_{\g}}(Y)\}_{\g\in\G}$ suitably.

\begin{definition} Let $\HH_X$ and $\HH_Y$ be \Ham connections on $P_{\vp}(X)$ and $P_{\vp}(Y)$ respectively. We say that they are \textit{$\pi$-compatible} if $(\pi_{\vp})_*\HH_X=\HH_Y$.
\end{definition}

\noindent We show that $\pi$-compatible \Ham connections exist if $\vp\in\O G$. Recall there exists a unique normalized generating Hamiltonian $\{H_{\tt}^{\vp,X}\}_{\tt\in [0,2\pi]}$ (resp. $\{H_{\tt}^{\vp,Y}\}_{\tt\in [0,2\pi]}$) of $\vp$ acting on $X$ (resp. $Y$), i.e. a family of smooth functions $H_{\tt}^{\vp,X}:X\ra \RR$ (resp. $H_{\tt}^{\vp,Y}:Y\ra \RR$) satisfying, for any $\tt$,
\[ \dot{\vp}_{\tt} = X_{H_{\tt}^{\vp,X}} \circ \vp_{\tt}\]
and
\[ \int_X H_{\tt}^{\vp,X} \w_X^{\dim X/2} =0\]
(resp. similar equalities for $Y$). Pick a cut-off function $\chi:[0,1]\ra [0,1]$ such that $\chi|_{[0,1/3]}\equiv 0$ and $\chi|_{[2/3,1]}\equiv 1$ Define $\s_{X,\pm}\in\O^2(D_{\pm}\times X;\RR)$ and $\s_{Y,\pm}\in\O^2(D_{\pm}\times Y;\RR)$ by
\begin{align}
\s_{X,-} :=\w_X &\quad\text{and}\quad \s_{X,+}:=\w_X-d(\chi(r)H_{\tt}^{\vp,X}d\tt)\nonumber \\
\s_{Y,-} :=\w_Y &\quad\text{and}\quad \s_{Y,+}:=\w_Y-d(\chi(r)H_{\tt}^{\vp,Y}d\tt). \label{compatibleform}
\end{align}
Then $\s_{X,-}$ and $\s_{X,+}$ (resp. $\s_{Y,-}$ and $\s_{Y,+}$) can be glued to a closed 2-form $\s_X$ (resp. $\s_Y$) on $P_{\vp}(X)$ (resp. $P_{\vp}(Y)$). We will show in a moment that $\s_X$ and $\s_Y$ define $\pi$-compatible \Ham connections. Starting with $\s_X$ and $\s_Y$, we can construct more $\pi$-compatible \Ham connections by \Ham perturbations. Let $K_X$ (resp. $K_Y$) be a 1-form on $S^2$ with values in the spaces of smooth functions on the fibers of $P_{\vp}(X)$ (resp. $P_{\vp}(Y)$). Define $\tilde{K}_X\in \O^1(P_{\vp}(X);\RR)$ by
\[ \iota_{\eta}\tilde{K}_X := (\iota_{(\pi_X)_*\eta}K_X)(p)\]
for any $p\in P_{\vp}(X)$ and $\eta\in T_{p}P_{\vp}(X)$, where $\pi_X:P_{\vp}(X)\ra S^2$ is the projection. Define $\tilde{K}_Y\in \O^1(P_{\vp}(Y);\RR)$ similarly.

\begin{definition} We say that $K_X$ and $K_Y$ are \textit{$\pi$-compatible} if for any $z\in S^2$ and $v\in T_zS^2$,
\[ \pi_*X_{\iota_v K_X}=X_{\iota_v K_Y}\]
where we identify the fibers of $P_{\vp}(X)$ and $P_{\vp}(Y)$ over $z$ with $X$ and $Y$ respectively.
\end{definition}

\begin{lemma} \label{compatible} Suppose $K_X$ and $K_Y$ are $\pi$-compatible. Then the 2-forms $\s_X+d\tilde{K}_X$ and $\s_Y+d\tilde{K}_Y$ define a pair of $\pi$-compatible \Ham connections on $P_{\vp}(X)$ and $P_{\vp}(Y)$. In particular, $\s_X$ and $\s_Y$ themselves define a pair of $\pi$-compatible \Ham connections.
\end{lemma}
\begin{proof} By restricting the 2-forms to $D_{\pm}\times X$ and $D_{\pm}\times Y$, and noticing that the normalized generating Hamiltonians $H_{\tt}^{\vp,X}$ and $H_{\tt}^{\vp,Y}$ satisfy the condition we are going to impose, it suffices to show the following: Let $f_X,g_X:D\times X\ra \RR$ and $f_Y,g_Y:D\times Y\ra \RR$ be smooth functions. Suppose for any $z=x+iy\in D$,
\[ \pi_*X_{f_X(z,-)} = X_{f_Y(z,-)}\quad\text{and}\quad\pi_*X_{g_X(z,-)}=X_{g_Y(z,-)}.\]
Then the 2-forms
\[\w_X+d(f_Xdx+g_Xdy)\quad\text{and}\quad w_Y+d(f_Ydx+g_Ydy)\]
define a pair of $\pi$-compatible \Ham connections on $D\times X$ and $D\times Y$ respectively. This follows from a simple observation that the horizontal distributions defined by these 2-forms are given by the graphs of the linear map $TD\ra TX$ defined by
\[ a\pt_x +b\pt_y\mapsto aX_{f_X}+bX_{g_X}\]
and the linear map $TD\ra TY$ defined by
\[ a\pt_x +b\pt_y\mapsto aX_{f_Y}+bX_{g_Y}\]
respectively.
\end{proof}

\begin{myproof}{Theorem}{\ref{GPmain}} Fix smooth families
\[J_X=\{J_X^{\g,z}\}_{(\g,z)\in \G\times S^2}\in \J_{\vp}(X)\quad\text{and}\quad J_Y=\{J_Y^{\g,z}\}_{(\g,z)\in \G\times S^2}\in \J_{\vp}(Y)\]
such that $I_X=\{J_X^{\g,0}\}$ and $I_Y=\{J_Y^{\g,0}\}$ where $I_X$ and $I_Y$ are given in Assumption \ref{GPassume}, and the restriction of $\pi_{\vp_{\g}}$ to any fiber is holomorphic with respect to $J_X^{\g,z}$ and $J_Y^{\g,z}$. Let $\HH_X=\{\HH_X^{\g}\}_{\g\in\G}$ and $\HH_Y=\{\HH_Y^{\g}\}_{\g\in\G}$ be smooth families of \Ham connections on $\{P_{\vp_{\g}}(X)\}_{\g\in\G}$ and $\{P_{\vp_{\g}}(Y)\}_{\g\in\G}$ respectively such that for any $\g\in \G$, $\HH_X^{\g}$ and $\HH_Y^{\g}$ are $\pi$-compatible.

\begin{lemma} \label{GPeasylemma} For any $\g\in\G$, $\pi_{\vp_{\g}}: (P_{\vp_{\g}}(X), J_{\HH_X^{\g}})\ra (P_{\vp_{\g}}(Y), J_{\HH_Y^{\g}})$ is holomorphic.
\end{lemma}
\begin{proof} Obvious.
\end{proof}

Now take $\HH_X^{\g}$ and $\HH_Y^{\g}$ to be the Hamiltonian connections defined by the closed 2-forms
\[ \s^{\g}_X+d\tilde{K}^{\g}_X\quad\text{and}\quad \s^{\g}_Y+d\tilde{K}^{\g}_Y\]
in Lemma \ref{compatible} for some Hamiltonian perturbations $K^{\g}_X$ and $K^{\g}_Y$, where $\s^{\g}_X$ and $\s^{\g}_Y$ are defined by \eqref{compatibleform} for the loop $\vp_{\g}$. We require that
\begin{align*}
K^{\g}_X|_{D_-}\equiv 0 &\quad\text{and}\quad K^{\g}_X|_{D_+}=\langle \ol{K}^{\g}_X,\mu_X\rangle\\
K^{\g}_Y|_{D_-}\equiv 0 &\quad\text{and}\quad K^{\g}_Y|_{D_+}=\langle \ol{K}^{\g}_Y,\mu_Y\rangle
\end{align*}
where $\mu_X$ (resp. $\mu_Y$) is the moment map of $(X,\w_X)$ (resp. $(Y,\w_Y)$) and $\ol{K}^{\g}_X,\ol{K}^{\g}_Y\in \O^1(D_+;\gg)$ are some $\gg$-valued 1-forms with support contained in the interior of $D_+$.

Since the $G$-actions on $X$ and $Y$ are transitive, by assumption, it follows that Condition \ref{GPdefSeidel} is satisfied for $(\HH_X,J_X)$ (resp. $(\HH_Y,J_Y)$) for generic $\{\ol{K}^{\g}_X\}_{\g\in\G}$ (resp. $\{\ol{K}^{\g}_Y\}_{\g\in\G}$). Recall we have to require $\HH^{\g}_X$ and $\HH^{\g}_Y$ to be $\pi$-compatible. It is equivalent to imposing $\ol{K}^{\g}_X\equiv \ol{K}^{\g}_Y$. Therefore, by taking the intersection of the two residual subsets consisting of regular $\{\ol{K}^{\g}_X\}$ and $\{\ol{K}^{\g}_Y\}$, we conclude that there exist regular $\HH_X$ and $\HH_Y$ which are $\pi$-compatible.

By lemma \ref{GPeasylemma}, $\{\pi_{\vp_{\g}}\}_{\g\in\G}$ induces a map
\[ f:\M^{section}(A_X,\vp,\HH_X,J_X)\times_{(ev_0,h_X)}N_X\ra\M^{section}(A_Y,\vp,\HH_Y,J_Y)\times_{(ev_0,h_Y)}N_Y. \]
It suffices to show that
\begin{equation}\label{GPunion}
\fiber(f)\simeq \bigcup_{\substack{ B\in\pi_2^{section}(P_{\psi_{A_Y}}(F))\\\PP_{A_Y}(B)=A_X }} \M^{section}(B,\psi_{A_Y},\HH_F,J_F)\times_{(ev_0,h_F)}N_F
\end{equation}
for some $(\HH_F,J_F)$ satisfying Condition \ref{GPdefSeidel} for the cycle $h_F$. (Here $\psi_{A_Y}$ is regarded as a zero-dimensional cycle.) Let $(\g,u)\in\M^{section}(A_Y,\vp,\HH_Y,J_Y)\times_{(ev_0,h_Y)}N_Y$. Then $\pi_{\vp_{\g}}^{-1}(\im(u))$ is a compact almost complex submanifold of $(P_{\vp_{\g}}(X),J_{\HH_X^{\g}})$ fibering over $\im(u)\simeq S^2$. By definition, it is isomorphic to $P_{\psi_{A_Y}}(F)$. Moreover, it has a \Ham connection $\HH_F$ induced by a connection 1-form associated to $\HH_X^{\g}$. Define $J_F\in \J_{pt}(F)$ to be the restriction of $J_X$. Then $J_{\HH_F}$ is equal to the restriction of $J_{\HH_X^{\g}}$. It is not hard to see that the pre-image $f^{-1}((\g,u))$ is precisely the union in \eqref{GPunion}. It remains to verify that $(\HH_F,J_F)$ satisfies Condition \ref{GPdefSeidel} for $h_F$.

Let $B_0\in\pi_2^{section}(P_{\psi_{A_Y}}(F))$ and $B_1,\ldots,B_k\in\pi_2(F)$ satisfying
\begin{equation}\label{GPdim<0}
\dim N_F+2c_1^v(B_0)+2\sum_{i=1}^kc_1(B_i)\leqslant 0.
\end{equation}
We have to show that the fiber product
\begin{equation}\label{GPfiberprod}
 \M^{section}(B_0,\psi_{A_Y},\HH_F,J_F)\times_{(ev_0,ev_{in})}\M^{chain}(B_1,\ldots,B_k,\{pt\},J_F^0)\times_{(ev_{out},h_F)}N_F
\end{equation}
is regular where $J_F^0$ is the restriction of $J_F$ on the fiber over $0\in S^2$. We have a canonical embedding
\begin{align}
& \M^{section}(B_0,\psi_{A_Y},\HH_F,J_F)\times_{(ev_0,ev_{in})}\M^{chain}(B_1,\ldots,B_k,\{pt\},J_F^0)\times_{(ev_{out},h_F)}N_F \label{GPembed} \\
\hookrightarrow~& \D_{\G}\times_{\G^2} \left( \M^{section}(\PP_{A_Y}(B_0),\vp,\HH_X,J_X)\times_{(ev_0,ev_{in})}\M^{chain}(\iota_*B_1,\ldots,\iota_*B_k,\G,I_X)\times_{(ev_{out},h_X)}N_X\right) \nonumber
\end{align}
where $\iota_*:\pi_2(F)\ra \pi_2(X)$ is induced by the inclusion $\iota:F\hookrightarrow X$. By Lemma \ref{chernclassformula} and the assumption $\dim\G+\dim N_Y+2c_1^v(A_Y)=0$, we have
\begin{align}
& \dim\G+\dim N_X+2c_1^v(\PP_{A_Y}(B_0))+2\sum_{i=1}^kc_1(\iota_*B_i) \nonumber\\
= ~& \dim N_F+2c_1^v(B_0)+2\sum_{i=1}^kc_1(B_i) \label{GPequal}
\end{align}
which is non-positive by \eqref{GPdim<0}. Therefore, if \eqref{GPdim<0} is strict, the fiber product \eqref{GPfiberprod} is empty, by the embedding \eqref{GPembed} and Condition \ref{GPdefSeidel} for $(\HH_X,J_X)$. If \eqref{GPdim<0} is an equality, then $k=0$, since there is a real two-dimensional symmetry for each sphere bubble. We show that the fiber product \eqref{GPfiberprod} is regular by showing that the kernel of the linearization of the Cauchy-Riemann operator is zero, and this follows from the linearization of the embedding \eqref{GPembed}, the equality \eqref{GPequal} and the regularity of the target of \eqref{GPembed}. This completes the proof of the theorem.
\end{myproof}

\section{Moment correspondences and quilted Floer theory}\label{QFTsection}
\subsection{Moment correspondences} Let $G$ be a compact connected Lie group and $(X,\w_X)$ a compact monotone Hamiltonian $G$-manifold with moment map $\mu_X:X\ra \gg^{\vee}$. Weinstein \cite{W} defined a Lagrangian correspondence $T^*G\xrightarrow{C} \XX$, called the \textit{moment correspondence}, defined by
\[C:=\{ (g,\mu_X(x),x,g\cdot x)|~g\in G, x\in X\}\]
where $T^*G$ is identified with $G\times \gg^{\vee}$ by left multiplication. For any $g\in G$, define $L_g:=T^*_gG$ and
\[ \D_g:=\{(x,g\cdot x)|~x\in X\}.\]
For $g$ equal to the identity element $e$ of $G$, we also put $L:=L_e$ and $\D:=\D_e$. A key observation is that for any $g\in G$, the \textit{geometric composition} \cite{WW2} $L_g\circ C$ is embedded and is equal to $\D_g$.

In order to apply the quilted Floer theory of Ma'u-Wehrheim-Woodward \cite{MWW} to the moment correspondence $C$, it is necessary to establish a suitable monotonicity condition. The following lemma will do the job.

\begin{lemma}\label{monolemma} For any loop $\lo:\pt D\ra G$ and disks $u_{\pm}:D\ra X$ satisfying
\begin{equation}\label{monolemma1}
u_+(\eit)=\vp(\eit)\cdot u_-(\eit)\quad \text{for any }\tt,
\end{equation}
define
\[E(\lo,u_{\pm}):= -\int_D(u_-)^*\w_X+\int_D(u_+)^*\w_X-\int_0^{2\pi} \langle \lo^{-1}(d\lo/d\tt), \mu_X(u_-(\eit))\rangle d\tt\]
where $\langle-,-\rangle$ is the natural pairing $\gg\otimes \gg^{\vee}\ra\RR$. There exists a map $C:\pi_1(G)\ra \RR$ such that for any pair $(\lo,u_{\pm})$ satisfying \eqref{monolemma1},
\[ E(\lo,u_{\pm}) =\monoconst\cdot c_1^v(s_{\lo,u_{\pm}}) +C([\lo])\]
where $\monoconst$ is the monotonicity constant of $X$ and $s_{\lo,u_{\pm}}$ is the section of $P_{\lo}(X)$ defined in Section \ref{Hamfibsection}.
\end{lemma}
\begin{proof} We first prove that $E$ is invariant under homotopy of $(\lo,u_{\pm})$. The inclusion of $C\simeq G\times X$ into $(T^*G)^-\times \XX$ splits into two maps
\[ \begin{array}{cccc}
m:&G\times X&\ra& T^*G\simeq G\times \gg^{\vee}\\ \\ [-1em]
 &(g,x) &\mapsto &(g,\mu_X(x)).
\end{array}\]
and
\[ \begin{array}{cccc}
j:&G\times X&\ra& X\times X\\ \\ [-1em]
 &(g,x) &\mapsto &(x,g\cdot x).
\end{array}\]
We have $m^*d\l_G=j^*\w_{\XX}$ where $\l_G$ is the canonical Liouville form on $T^*G$. Notice that giving a pair $(\lo,u_{\pm})$ satisfying \eqref{monolemma1} amounts to giving a loop $\tilde{\lo}$ in $G\times X$ and a capping disk of $j\circ \tilde{\lo}$. The assertion then follows from Stokes' theorem.

Now, given two pairs $(\lo,u_{\pm})$ and $(\lo',u'_{\pm})$ satisfying \eqref{monolemma1} and a homotopy $\lo_s$ joining $\lo$ and $\lo'$. One can easily produce $\ol{u}_{\pm}$ such that $(\lo',u'_{\pm})$ is homotopic to $(\lo,\ol{u}_{\pm})$ and $s_{\lo,\ol{u}_{\pm}}=s_{\lo,u_{\pm}}\# v$ for some sphere $v$ lying in a fiber of $P_{\lo}(X)$. By the first part of the proof, we have
\begin{align*}
E(\lo',u'_{\pm})-\monoconst \cdot c_1^v(s_{\lo',u'_{\pm}}) &=E(\lo,u_{\pm})-\monoconst \cdot c_1^v(s_{\lo,u_{\pm}})    + \left( \int_{S^2} v^*\w_X - \monoconst \cdot c_1(v)\right)\\
&= E(\lo,u_{\pm})-\monoconst \cdot c_1^v(s_{\lo,u_{\pm}}) .
\end{align*}
\end{proof}
\begin{remark} It should be pointed out that Evans-Lekili \cite{EL} have already proved a similar result. Besides establishing a well-defined Floer theory, Lemma \ref{monolemma} is necessary for proving a gluing result in Appendix \ref{gluing}.
\end{remark}
\subsection{A notion of capping disks} \label{cappingdisk}
Parallel to the introduction of $\Nov{X}{\O G}$ for $QH^*(X)$, we introduce a notion of capping disks for the Floer cochain complex $CF^*(\D,\D)$. Since the outcome of quilted Floer theory applied to the moment correspondence is an $A_{\infty}$ homomorphism from the wrapped Floer cochain complex $CW^*(L,L)$ of the cotangent fiber $L$ to the Floer cochain complex $CF^*((L,C),(L,C))$ of the generalized Lagrangian $(L,C)$, it is necessary to do the same thing for $CF^*((L,C),(L,C))$ as well as the morphism spaces between $(L,C)$ and $\D$. For this reason, we define, following \cite{BCL},  an $A_{\infty}$ category $\AA$ as follows.

\begin{definition} \label{object} (objects) The objects of $\AA$ are
\[ (L,C)\quad\text{and}\quad\D.\]

Next, we define the morphism spaces of $\AA$. Notice that every generalized \Ham chord involved consists of up to three \Ham chords, each in either $T^*G$ or $\XX$. For example, the generalized \Ham chords for defining $CF^*(\D,(L,C))$ are of the form $(x_1,x_2)$ where $x_1$ is a time-1 \Ham chord in $\XX$ and $x_2$ is a time-$\delta$ \Ham chord in $T^*G$ (for a fixed $\delta >0$). They are required to satisfy
\begin{equation}\label{EL1}
x_1(0)\in \D,~(x_1(1),x_2(0))\in C^T,~x_2(\delta)\in L.
\end{equation}
For simplicity, we denote any generalized \Ham chords by $\ul{x}=(x_-,x_{\XX},x_+)$ where $x_{\XX}$ is a time-1 \Ham chord in $\XX$, and each of $x_-$ and $x_+$ is either a time-0 (i.e. constant) or a time-$\delta$ \Ham chord in $T^*G$. These \Ham chords are required to satisfy a condition analogous to \eqref{EL1}. In the above example, $x_{\XX}$ is equal to $x_1$ and $x_+$ is equal to $x_2$. As for $x_-$, recall that $\D=L\circ C$ is embedded so there is a unique point $z\in T^*G$ such that $(z,x_1(0))\in C$. We have $x_-\equiv z$.
\end{definition}

Give the unit disk $D$ a negative \sle near $-1$.
\begin{definition} Let $\ul{x}=(x_-,x_{\XX},x_+)$ be a generalized \Ham chord.
\begin{enumerate}
\item A \textit{\capping disk} for $\ul{x}$ is a pair $(\lo,u)$ consisting of maps $\vp:\pt D\setminus \{-1\}\ra G$ and $u:D\setminus \{-1\}\ra \XX$ which satisfies
\begin{enumerate}[(a)]
\item $u$ converges to $x_{\XX}$ at the negative strip-like end;
\item $u(\eit)\in\D_{\lo(\eit)}$ for any $\tt\ne \pi$; and
\item $\lim_{\tt\to\pi^-}\lo(\eit)$ is equal to the starting point of $\pi\circ x_+$ and $\lim_{\tt\to\pi^+}\lo(\eit)$ is equal to the ending point of $\pi\circ x_-$, where $\pi:T^*G \ra G$ is the projection.
\end{enumerate}
\item A \textit{homotopy of \capping disks} for $\ul{x}$ consists of homotopies $s\mapsto \lo_s$ and $s\mapsto u_s$ such that for any $s$, $(\lo_s,u_s)$ is a \capping disk for $\ul{x}$.
\end{enumerate}
\end{definition}

Let $\ul{L}$ and $\ul{L}'$ be any objects of $\AA$. Denote by $\X(\ul{L},\ul{L}')$ the set of generalized \Ham chords for the cyclic set $(\ul{L},(\ul{L}')^T)$ of Lagrangian correspondences.
\begin{definition} \label{morphismspace} (morphisms) The morphism space from $\ul{L}$ to $\ul{L}'$ is defined to be
\[ Hom_{\AA}^*(\ul{L},\ul{L}') := \bigoplus \ZZ\langle \ul{x},[\lo,u]\rangle \]
where the direct sum is taken over all $\ul{x}\in\X(\ul{L},\ul{L}')$ and all homotopy classes $[\lo,u]$ of \capping disks for $\ul{x}$.
\end{definition}

\noindent It is straightforward to extend the definition of the usual $A_{\infty}$ operations and the quilted invariants to the case where \capping disks are present. The general principle is to glue the \capping disks for input generalized \Ham chords and the patches of any \holo quilted surfaces of interest to us which are labelled by $\XX$.
\subsection{Quilted invariants associated to moment correspondences}\label{QFT}
The quilted invariants we are interested in are the following two linear maps
\begin{equation}\label{MCop}
\MC^{op}: HW^*(L,L)\ra Hom_{H(\AA)}^*((L,C),(L,C))\xrightarrow{Y_{\#}} Hom_{H(\AA)}^*(\D,\D)
\end{equation}
and
\begin{equation}\label{MCcl}
\MC^{cl}: SH^*(T^*G)\ra Hom(H_*(X)\otimes H_*(X),\ZZ[\Nov{X}{\Lo G}])
\end{equation}
where the first arrow of \eqref{MCop} and the arrow \eqref{MCcl} count configurations in Figure 1 and Figure 2 respectively, and the second arrow $Y_{\#}$ of \eqref{MCop} is the isomorphism induced by Lekili-Lipyanskiy's quasi-isomorphism \cite{LL}
\[ Y:(L,C)\ra \D\]
in $\AA$ defined by counting configurations in Figure 3.

\begin{center}
\begin{minipage}{8cm}
\begin{center}
\vspace{0.2cm}
\begin{tikzpicture}
\tikzmath{\x1 = 0.6; \x2 = 1.6; \x3=2; \x4=1; \x5=\x1+(\x2-\x4)/2; \x6=1; \x7=3; \x8=0.2; \x9=0.45; \y1=0.45;}

\draw [line width=\x9mm, pattern=north east lines, pattern color=gray]  (0,0) -- (\x3,0)  to[out=2,in=182] (\x3+\x7,\x5) -- (\x3+\x7+\x6,\x5) -- (\x3+\x7+\x6,\x5+\x4) -- (\x3+\x7,\x5+\x4) to[out=178,in=-2]  (\x3,\x1+\x2+\x1) -- (0,\x1+\x2+\x1) -- (0,\x1+\x2) -- (\x3-\x8,\x1+\x2) arc (90:-90:0.5*\x2) -- (\x3-\x8,\x1) -- (0,\x1) -- cycle;

\draw [->, line width=\x9mm] (0,0) -- (0,\x1) node [midway, left] {};
\draw [->, line width=\x9mm] (0,\x1) -- (0,\x1+\x2) node [midway, left] {$\ul{x}_{out}$};
\draw [->, line width=\x9mm] (0,\x1+\x2) -- (0,\x1+\x2+\x1) node [midway, left] {};

\draw [->, line width=\x9mm] (\x3+\x7+\x6,\x5) -- (\x3+\x7+\x6,\x5+\x4) node [midway, right] {$x_{in}$};

\node[anchor = south west] at (\x5+1.7*\x7/2,\x1+\x2+0.2*\x1/2) {$L$};
\node[anchor = north west] at (\x5+1.7*\x7/2,\x1-0.3*\x1/2)  {$L$};
\node[anchor = west] at (0.15*\x3, \x1+\x2/2) {$\XX$};
\node[anchor = east] at (\x3+0.95*\x7, \x1+\x2/2) {$T^*G$};
\node[anchor = east] at (\x3+0.4*\x7, \x1+\x2/2) {$C$};
\node at (0,1.5+\x1+0.5*\x2){} ;
\node at (0,-1.5+\x1+0.5*\x2){} ;
\end{tikzpicture}

\vspace{0.3cm}
\noindent \hypertarget{fig1}{FIGURE 1.}
\vspace{.2cm}
\end{center}
\end{minipage}
\begin{minipage}{8cm}
\begin{center}
\vspace{0.2cm}
\begin{tikzpicture}
\tikzmath{\x1 = 1.5; \x2 = 0.2; \x3=1; \x4=1; \x5=2.5; \x9=0.45; \y1=0.45; \y3=1.2*\x3;}
\tikzmath{\z1= 0.15*\x3 ; \z2=0.2; \z3=1.3*\x3; \z4=0;}

\draw [line width=0mm, pattern=north east lines, pattern color=gray]
(\x1,-\x3) -- (\x1+\x5,-\x3) --(\x1+\x5,\x3) -- (\x1,\x3)-- cycle;

\begin{scope}
\clip(\x1,-\y3)--(\x1,\y3)--(\x1+1.1*\x2,\y3)--(\x1+1.1*\x2,-\y3)--cycle;
\draw [line width=\y1mm, fill=white] (\x1,0) ellipse (\x2 and \x3);
\end{scope}

\begin{scope}
\clip(\x1,-\y3)--(\x1,\y3)--(\x1-1.1*\x2,\y3)--(\x1-1.1*\x2,-\y3)--cycle;
\draw [dashed, line width=\y1mm, fill=white] (\x1,0) ellipse (\x2 and \x3);
\end{scope}

\draw[dashed, line width=0mm, fill=white] (\x1+\x5,0) ellipse (\x2 and \x3);
\draw[dashed, line width=0mm, pattern=north east lines, pattern color=gray] (\x1+\x5,0) ellipse (\x2 and \x3);

\begin{scope}
\clip(\x1+\x5,-\y3)--(\x1+\x5,\y3)--(\x1+1.1*\x2+\x5,\y3)--(\x1+1.1*\x2+\x5,-\y3)--cycle;
\draw [line width=\y1mm] (\x1+\x5,0) ellipse (\x2 and \x3);
\end{scope}

\begin{scope}
\clip(\x1+\x5,-\y3)--(\x1+\x5,\y3)--(\x1-1.1*\x2+\x5,\y3)--(\x1-1.1*\x2+\x5,-\y3)--cycle;
\draw[dashed, line width=\y1mm] (\x1+\x5,0) ellipse (\x2 and \x3);
\end{scope}

\draw [line width=\x9mm] (0,-\x3) -- (\x1+\x5,-\x3);
\draw [line width=\x9mm] (0,\x3) -- (\x1+\x5,\x3);

\draw [->,line width=0.99*\x9mm] (\x1+\x2+\x5,-0.01) -- (\x1+\x2+\x5,0);


\node[anchor = west] at (\x1+0.4*\x5,0) {$T^*G$};
\node[anchor = east] at (0.8*\x1,0) {$\XX$};
\node[anchor = west] at (\x1+1*\x2,0) {$C$};
\node[anchor = west] at (\x1+1.4*\x2+\x5, 0) {$y_{in}$};

\node at  (-\x3,0) [circle,fill, inner sep=2pt]{};

\node[anchor = east] at (-0.2*\x3,-1.1*\x3){$h_-\times h_+$};

\draw  [line width=\y1mm] (0,\x3) arc (90:270:\x3);
\draw  [line width=\y1mm] (-\x3,0) arc (0:15:3);
\draw  [line width=\y1mm] (-\x3,0) arc (0:-15:3);



\node at (0,1.5){} ;
\node at (0,-1.5){} ;
\end{tikzpicture}

\vspace{0.3cm}
\noindent \hypertarget{fig2}{FIGURE 2.}
\vspace{.2cm}
\end{center}
\end{minipage}
\end{center}

\begin{center}
\begin{minipage}{8cm}
\begin{center}
\vspace{0.2cm}
\begin{tikzpicture}
\tikzmath{\x1 = 1; \x2 = 1.5; \x3=2; \x4=.5; \x5=1.1;  \x9=0.45; \y1=0.45;}

\draw [line width=\x9mm] (0,0) -- (\x4,0);
\draw [line width=\x9mm] (0,\x1) -- (\x4,\x1);
\draw [line width=\x9mm] (0,\x1+\x2) -- (\x4,\x1+\x2);

\draw  [line width=\y1mm]  (\x4,\x1+\x2) to[out=-2,in=135] (\x3,\x1+\x2/2);

\node (A) at (\x3,\x1+\x2/2) {};
\node (B) at ([shift=({-45:\x5})]A) {};

\draw [line width=\y1mm, pattern=north east lines, pattern color=gray]  (0,0) -- (\x4,0)  to[out=2,in=225] (B.center) to[out=45,in=0] (\x3,\x1+\x2/2)  to[out=-135,in=2] (\x4,\x1) -- (0,\x1) -- cycle;

\node[anchor = west] at (0.2*\x3, 0.6*\x1) {$T^*G$};
\node[anchor = west] at (0.2*\x3, -0.3*\x1) {$L$};
\node[anchor = west] at (0, \x1+\x2/2) {$\XX$};
\node[anchor = west] at (0.2*\x3, \x1+1.2*\x2) {$\D$};
\node[anchor = north west] at (0.8*\x3, 1.5*\x1) {$C$};
\node[anchor = east] at (0,0.5*\x1+0.5*\x2) {$\ul{x}_{out}$};
\node at  (A) [circle,fill,inner sep=2pt]{};

\draw [->, line width=\x9mm] (0,0) -- (0,\x1);
\draw [->, line width=\x9mm] (0,\x1) -- (0,\x1+\x2);

\node at (0,3){};
\node at (0,-0.5){};
\end{tikzpicture}

\vspace{0.3cm}
\noindent \hypertarget{fig3}{FIGURE 3.}
\vspace{.2cm}
\end{center}
\end{minipage}
\begin{minipage}{8cm}
\begin{center}
\vspace{0.2cm}
\begin{tikzpicture}
\tikzmath{\x1 = 2; \x2 = 1.5; \x3=1.2; \x4=1.2; \x5=1.8; \x6=1.2; \x7=1; \x8=0.75; \x9=1.2;}
\tikzmath{\y1=0.3;\y2=.45; \y3= \x1+2*\x2+\x3+\x4+\x5+\x6+\x8+\x9+\y1;}
\tikzmath{\z1=\y1+\x9;\z4=3;}
\draw [->, line width=\y2mm] (\y3,-0.5*\x7) -- (\y3,0.5*\x7) node [midway, right] {$x_{in}$};
\draw [line width=\y2mm]  (\y3,0.5*\x7) -- (\y3-\y1,0.5*\x7)  to[out=178,in=-2] (\y3-\y1-\x9,\x8) arc (90:270:\x8) to[out=2,in=182]  (\y3-\y1,-0.5*\x7) -- (\y3,-0.5*\x7) -- cycle;

\node at  (\y3-\y1-\x9-\x8,0) [circle,fill, inner sep=2pt]{};

\node[anchor = east] at (\y3-\y1-\x9-\x8,-0.35*\z4){$h$};
\draw  [line width=\y2mm] (\y3-\y1-\x9-\x8,0)  arc (0:15:\z4);
\draw  [line width=\y2mm] (\y3-\y1-\x9-\x8,0)  arc (0:-15:\z4);
\node[anchor = east] at (\y3-0.4*\z1, 0.95*\x7) {$\D$};
\node[anchor = east] at (\y3-0.15*\z1 ,0) {$\XX$};
\node at (\y3,1.75){};
\node at (\y3,-1.75){};

\end{tikzpicture}

\vspace{0.3cm}
\noindent \hypertarget{fig4}{FIGURE 4.}
\vspace{.2cm}
\end{center}
\end{minipage}
\end{center}
We will also need the isomorphism of Piunikhin-Salamon-Schwarz \cite{PSS}
\[ PSS: Hom_{H(\AA)}^*(\D,\D)\ra QH^*(X;\ZZ[\Nov{X}{\O G}])\]
defined by counting configurations in Figure 4.

\begin{definition} Define the ``quantum co-product''
\[\qcp: QH^*(X;\ZZ[\Nov{X}{\O G}])\ra Hom(H_*(X)\otimes H_*(X),\ZZ[\Nov{X}{\Lo G}]) \]
by
\[\langle \qcp(PD[h]), [h_-]\otimes [h_+]\rangle :=\sum_{A\in \pi_2(X)}\langle h,h_-,h_+\rangle_A^{GW} \TT^A\]
for any pseudocycles $h,h_-,h_+$ in $X$, where $PD$ is the Poincar\'e dual, $\langle -,-,-\rangle_A^{GW}$ is the 3-pointed genus-zero Gromov-Witten invariant on $X$, and the exponent $A$ of $\TT^A$ is regarded as an element of $\Nov{X}{\Lo G}$ via the canonical inclusion $\pi_2(X)\hookrightarrow \Nov{X}{\Lo G}$.
\end{definition}

\begin{lemma} \label{qcpinj} $\qcp$ is injective.
\end{lemma}
\begin{proof}
Define
\[ \begin{array}{cccc}
ev_{[X],-}:&Hom(H_*(X)\otimes H_*(X),\ZZ[\Nov{X}{\Lo G}])&\ra& Hom(H_*(X),\ZZ[\Nov{X}{\Lo G}]) \\ \\ [-1em]
 &f &\mapsto &f([X]\otimes -).
\end{array}\]
Then $ev_{[X],-}\circ\qcp$ is equal to the classical intersection pairing on pseudocycles in $X$ tensored with the natural map $\ZZ[\Nov{X}{\O G}]\ra \ZZ[\Nov{X}{\Lo G}]$. The result follows from the Poincar\'e duality and the injectivity of the latter map.
\end{proof}

\begin{proposition} \label{commQFT} The following diagram is commutative:
\begin{equation}\label{MCopcl}
\begin{tikzcd}
HW^*(L,L) \arrow{r}{PSS~\circ~ \MC^{op}} \arrow{d}[left]{\mathcal{OC}} &[2em]   QH^*(X;\ZZ[\Nov{X}{\O G}])  \arrow{d}{\qcp} \\
SH^*(T^*G) \arrow{r}{\MC^{cl}} &[2em] Hom(H_*(X)\otimes H_*(X),\ZZ[\Nov{X}{\Lo G}])
\end{tikzcd}
\end{equation}
where $\mathcal{OC}$ is the length-zero part of the open-closed map.
\end{proposition}

\begin{proof} We have a more complicated diagram:
\begin{equation}\nonumber
\begin{tikzcd}[column sep=2ex]
 HW^*(L,L) \arrow{d}[left]{\mathcal{OC}} \arrow{r}{} & Hom_{H(\AA)}^*((L,C),(L,C))\arrow[dr, "\mathcal{OC}"] \arrow{r}{Y_{\#}}& Hom_{H(\AA)}^*(\D,\D) \arrow{d}[left]{\mathcal{OC}}  \arrow{r}{PSS} &  QH^*(X;\ZZ[\Nov{X}{\O G}])   \arrow[dl, "\qcp"]\\
 SH^*(T^*G) \arrow{rr}{\MC^{cl}} &  & Hom(H_*(X)^{\otimes 2},\ZZ[\Nov{X}{\Lo G}]) &
\end{tikzcd}
\end{equation}

\noindent The commutativity of the leftmost square is proved by Ritter-Smith \cite{RS}. The commutativity of the middle triangle follows from the existence of factorizations of the open-closed maps for the objects $(L,C)$ and $\D$ into the full open-closed map
\[ \mathcal{OC}:HH_*(\AA)\ra Hom(H_*(X)\otimes H_*(X),\ZZ[\Nov{X}{\Lo G}])\]
defined by counting configurations in Figure 5. See \textit{loc. cit.} for more detail.

\begin{center}
\vspace{.3cm}
\begin{tikzpicture}
\tikzmath{\x1 = 0.2; \x2 = 3; \x3=2; \x4=.4; \x5=0.7; \x6=1; \x7=3; \x8=0.2; \x9=0.45; \y1=0.45;\y2=.5; \y3=0.5; \y4=0.8;}
\tikzmath{\z1=-\x5+3*\x1+\x4+\y3; \z2=-\x1+\x5+\x2-\x4-\y3; \z3=0.5*\x3; \z4=4;}

\draw  [line width=\y1mm] (\z3,\x1+\x2/2) arc (135:150:\z4);
\draw  [line width=\y1mm] (\z3,\x1+\x2/2) arc (135:120:\z4);



\draw [line width=\y1mm, pattern=north east lines, pattern color=gray]
(\x3+\x7+\x6,\x1+\x2 +\x5)  -- (\x3+\x7+\x6,\x1+\x1+\x2 +\x5) -- (\x3+\x7+\x6-\y4,\x1+\x1+\x2 +\x5) to[out=182,in=2] (\x3,\x1+\x2+\x1) -- (0,\x1+\x2+\x1) arc (90:270:\x1+\x2/2) -- (\x3,0) to[out=-2,in=178] (\x3+\x7+\x6-\y4,-\x5) -- (\x3+\x7+\x6,-\x5) -- (\x3+\x7+\x6,-\x5+\x1) -- (\x3+\x7+\x6-\y4,-\x5+\x1) to[out=178,in=-2] (\x3,\x1) -- (0,\x1) arc (270:90:\x2/2) -- (0,\x1+\x2) -- (\x3,\x1+\x2) to[out=2,in=182] (\x3+\x7+\x6-\y4,\x1+\x2 +\x5) -- (\x3+\x7+\x6,\x1+\x2 +\x5) -- cycle;

\draw [line width=\y1mm, pattern=north east lines, pattern color=gray]  (\x3+\x7+\x6,\x1+\x2 +\x5-\x4) -- (\x3+\x7+\x6-\y4,\x1+\x2 +\x5-\x4) arc (90:270:\x1+\y3/2) -- (\x3+\x7+\x6,-\x1+\x2 +\x5-\x4-\y3) -- (\x3+\x7+\x6,\x2 +\x5-\x4-\y3) -- (\x3+\x7+\x6-\y4,\x2 +\x5-\x4-\y3) arc (270:90:\y3/2) -- (\x3+\x7+\x6,\x2 +\x5-\x4)-- cycle;

\draw [line width=\y1mm, pattern=north east lines, pattern color=gray]  (\x3+\x7+\x6,3*\x1+\y3 -\x5+\x4) -- (\x3+\x7+\x6-\y4,3*\x1+\y3 -\x5+\x4) arc (90:270:\x1+\y3/2) -- (\x3+\x7+\x6,\x1 -\x5+\x4) -- (\x3+\x7+\x6,2*\x1 -\x5+\x4) -- (\x3+\x7+\x6-\y4,2*\x1 -\x5+\x4) arc (270:90:\y3/2) -- (\x3+\x7+\x6, 2*\x1 -\x5+\x4+\y3)-- cycle;


\draw [->, line width=\x9mm] (\x3+\x7+\x6,-\x5) --(\x3+\x7+\x6,-\x5+\x1) ;
\draw [->, line width=\x9mm] (\x3+\x7+\x6,-\x5+\x1)-- (\x3+\x7+\x6,-\x5+\x1+\x4) ;
\draw [->, line width=\x9mm] (\x3+\x7+\x6,-\x5+\x1+\x4)-- (\x3+\x7+\x6,-\x5+\x1+\x4+\x1) ;
\draw [->, line width=\x9mm] (\x3+\x7+\x6,-\x5+\x1+\x4+\y3+\x1)-- (\x3+\x7+\x6,-\x5+\x1+\x4+\x1+\y3+\x1) ;

\draw [->, line width=\x9mm] (\x3+\x7+\x6,\x1+\x2 +\x5)  -- (\x3+\x7+\x6,\x1+\x1+\x2 +\x5) ;
\draw [->, line width=\x9mm] (\x3+\x7+\x6,\x1+\x2+\x5-\x4) -- (\x3+\x7+\x6,\x1+\x2 +\x5);
\draw [->, line width=\x9mm] (\x3+\x7+\x6,\x2+\x5-\x4) -- (\x3+\x7+\x6,\x1+\x2+\x5-\x4);
\draw [->, line width=\x9mm] (\x3+\x7+\x6,\x2+\x5-\x4-\x1-\y3) -- (\x3+\x7+\x6,\x1+\x2+\x5-\x4-\x1-\y3);


\node[anchor = east] at (\x3+0.9*\x7, \x1+\x2/2) {$\XX$};

\node at  (\x3+\x7+\x6,0.5*\z1+0.5*\z2) [circle,fill,inner sep=1pt]{};
\node at  (\x3+\x7+\x6,0.25*\z1+0.75*\z2) [circle,fill,inner sep=1pt]{};
\node at  (\x3+\x7+\x6,0.75*\z1+0.25*\z2) [circle,fill,inner sep=1pt]{};

\node at  (\z3,\x1+\x2/2) [circle,fill,inner sep=2pt]{};
\node[anchor = east ] at (0.9*\z3,1.8*\x1+\x2/2) {$h_-\times h_+$};
\end{tikzpicture}

\vspace{.3cm}
\noindent \hypertarget{fig5}{FIGURE 5.}
\end{center}

Finally, the commutativity of the rightmost triangle is proved by moving the interior marked point of the \holo disks defining the open-closed map for $\D$ to the boundary. The resulting bubbled configuration consists of a component corresponding to $PSS$ and a disk bubble which corresponds to $\qcp$ after gluing the disks in the two factors of $\XX$ to form a sphere.
\end{proof}

\section{Proof of Theorem \ref{MC=Se}} \label{proveMC=Se}
Denote by $H_G$ the \Ham defining $SH^*(T^*G)$, and by $\I(T^*G)$ and $\I(X^-\times X)$ the space of $\w_G$-compatible almost complex structures on $T^*G$ which are of contact type and the space of $\w_{\XX}$-compatible almost complex structures on $X^-\times X$ respectively. Let $S$ be a compact domain of $T^*G$ containing the image of the moment correspondence $C$ under the canonical projection $\pr:T^*G\times X\times X \ra T^*G$, and denote by $C^{\infty}_S(T^*G)$ the space of smooth functions on $T^*G$ with support in $S$.

Let $\vp:\G\ra \Lo G$ be a smooth cycle. Choose perturbation data
\begin{center}
\begin{minipage}{6.5cm}
\begin{align*}
&K_{G}^{\g,\cl}\in\O^1(S^1\times [-\cl,0];C^{\infty}(T^*G))\\
&K_{\XX}^{\g,\cl}\in\O^1(D;C^{\infty}(\XX))
\end{align*}
\end{minipage}
\begin{minipage}{6.5cm}
\begin{align*}
& J_{G}^{\g,\cl}\in C^{\infty}(S^1\times [-\cl,0];\I(T^*G))\\
& J_{\XX}^{\g,\cl}\in C^{\infty}(D;\I(\XX))
\end{align*}
\end{minipage}
\end{center}
depending smoothly on $\g\in \G$ and $\cl\in (0,+\infty)$, and satisfying
\begin{enumerate}
\item $K_{G}^{\g,\cl}$ is written as $H_G d\tt+\ol{K}_{G}^{\g,\cl}$ where
\[\ol{K}_{G}^{\g,\cl}\in \O^1(S^1\times [-\cl,0];C_S^{\infty}(T^*G));\]
\item as $\cl\to 0^+$, the pair $(K_{\XX}^{\g,\cl}, J_{\XX}^{\g,\cl})$ converges to one which is equivalent to the perturbation data defining the moduli space $\M^{section}(A,\vp,\HH,J)$ for some $(\HH,J)$; and
\item for any large $\cl$, $(K_{G}^{\g,\cl}, J_{G}^{\g,\cl})$ is equal to the perturbation data for the gluing of $\MC^{cl}$ and $\GG$, where $\GG:H_{-*}(\Lo G;\ZZ)\ra SH^*(T^*G)$ is the Abbondandolo-Schwarz isomorphism. See Appendix \ref{wrapped} for its construction.
\end{enumerate}

Given a pair of pseudocycles $h_{\pm}:N_{\pm}\ra X$ and $A\in\Nov{X}{\Lo G}$. Define $\M_{\vp}(h_{\pm},A)$ to be the moduli space of quadruples $\ul{u}^{\g,\cl}=(\g,\cl,u_G,u_{\XX})$ where
\begin{enumerate}[(a)]
\item $\g\in \G,~\cl\in (0,+\infty)$; and
\item $u_{G}:S^1\times [-\cl,0]\ra T^*G$ and $u_{\XX}:D\ra \XX$ are maps
\end{enumerate}
satisfying
\begin{enumerate}[(i)]
\item $u_{G}(e^{i\tt},0)\in T_{\vp_{\g}(e^{i\tt})}^*G$ for any $\tt$;
\item $(u_{G}(\eit,-\cl),u_{\XX}(\eit))\in C$ for any $\tt$;
\item $u_{\XX}$ hits the pseudocycle $h_-\times h_+$ at $0\in D$;
\item the section $s_{\pi\circ u_G|_{S^1\times \{-\cl\}},(u_{\XX})_{\pm}}$ defined in Section \ref{Hamfibsection} represents $A$, where $u_{\XX}=((u_{\XX})_-,(u_{\XX})_+)$ and $\pi:T^*G\ra G$ is the projection;
\item \begin{minipage}{8cm} \[ \left\{ \begin{array}{rl}
(du_{G}-X_{K_{G}^{\g,\cl}})^{0,1}_{J_{G}^{\g,\cl}}~=& 0\\ \\ [-0.8em]
(du_{\XX}-X_{K_{\XX}^{\g,\cl}})^{0,1}_{J_{\XX}^{\g,\cl}} ~=&0
\end{array} \right.\text{; and} \] \end{minipage}
\item $\int |du_{G}-X_{K_{G}^{\g,\cl}}|^2 +\int |du_{\XX}-X_{K_{\XX}^{\g,\cl}}|^2<+\infty$ .
\end{enumerate}
See Figure 6.

\begin{center}
\vspace{.3cm}
\begin{tikzpicture}
\tikzmath{\x1 = 2; \x2 = 0.2; \x3=1; \x4=1; \x5=5; \x9=0.45; \y1=0.45; \y3=1.2*\x3;}
\tikzmath{\z1= 0.15*\x3 ; \z2=0.2; \z3=1.3*\x3; \z4=3;}

\draw [<->, line width=\z2mm] (\x1,-\z3) -- (\x1+\x5,-\z3);
\draw [line width=\z2mm] (\x1,-\z3+\z1) -- (\x1,-\z3-\z1);
\draw [line width=\z2mm] (\x1+\x5,-\z3+\z1) -- (\x1+\x5,-\z3-\z1);
\draw [line width=0mm, pattern=north east lines, pattern color=gray]
(\x1,-\x3) -- (\x1+\x5,-\x3) --(\x1+\x5,\x3) -- (\x1,\x3)-- cycle;

\begin{scope}
\clip(\x1,-\y3)--(\x1,\y3)--(\x1+1.1*\x2,\y3)--(\x1+1.1*\x2,-\y3)--cycle;
\draw [line width=\y1mm, fill=white] (\x1,0) ellipse (\x2 and \x3);
\end{scope}

\begin{scope}
\clip(\x1,-\y3)--(\x1,\y3)--(\x1-1.1*\x2,\y3)--(\x1-1.1*\x2,-\y3)--cycle;
\draw [dashed, line width=\y1mm, fill=white] (\x1,0) ellipse (\x2 and \x3);
\end{scope}

\draw[dashed, line width=0mm, fill=white] (\x1+\x5,0) ellipse (\x2 and \x3);
\draw[dashed, line width=0mm, pattern=north east lines, pattern color=gray] (\x1+\x5,0) ellipse (\x2 and \x3);

\begin{scope}
\clip(\x1+\x5,-\y3)--(\x1+\x5,\y3)--(\x1+1.1*\x2+\x5,\y3)--(\x1+1.1*\x2+\x5,-\y3)--cycle;
\draw [line width=\y1mm] (\x1+\x5,0) ellipse (\x2 and \x3);
\end{scope}

\begin{scope}
\clip(\x1+\x5,-\y3)--(\x1+\x5,\y3)--(\x1-1.1*\x2+\x5,\y3)--(\x1-1.1*\x2+\x5,-\y3)--cycle;
\draw[dashed, line width=\y1mm] (\x1+\x5,0) ellipse (\x2 and \x3);
\end{scope}

\begin{scope}
\clip(0,-\y3)--(0,\y3)--(-1.1*\x4,\y3)--(-1.1*\x4,-\y3)--cycle;
\draw[line width=\y1mm] (0,0) ellipse (\x4 and \x3);
\end{scope}
\draw [line width=\x9mm] (0,-\x3) -- (\x1+\x5,-\x3);
\draw [line width=\x9mm] (0,\x3) -- (\x1+\x5,\x3);



\node[anchor = west] at (\x1+0.4*\x5,0) {$T^*G$};
\node[anchor = east] at (0.7*\x1,0) {$\XX$};
\node[anchor = west] at (\x1+1*\x2,0) {$C$};
\node[anchor = north] at (\x1+\x5/2,-\z3) {$\cl$};

\node at  (-\x4,0) [circle,fill, inner sep=2pt]{};
\node at (\x1+\x5+\z4,0) {};
\node (A) at (-\z4-\x4,0) {};
\node[anchor = north ] at ([shift=({-15:\z4})]A) {$h_-\times h_+$};

\draw  [line width=\y1mm] (-\x4,0) arc (0:15:\z4);
\draw  [line width=\y1mm] (-\x4,0) arc (0:-15:\z4);
\node (C) at (\x1+\x5,0) {};
\node[anchor = north ] (D) at ([shift=({-60:\x3})]C) {};
\path[ name path=p1] (\x1+\x5,0) ellipse (\x2 and \x3);
\path[ name path=p2] (C.center) -- (D.center);
\path [name intersections={of=p1 and p2,by=E}];
\node at  (E) [circle,fill,inner sep=1.5pt]{};
\node[anchor = west, shift={(0.05,0)}] at (E) {$T^*_{\vp_{\g}(\eit)}G$};
\node at (0,1.5){} ;
\node at (0,-1.5){} ;
\node at (-5*\x2,0){} ;
\end{tikzpicture}

\vspace{.3cm}
\noindent \hypertarget{fig6}{FIGURE 6.}
\end{center}

\noindent For any integer $i$, define $\M^i_{\vp}(h_{\pm},A)$ to be the component of $\M_{\vp}(h_{\pm},A)$ which has virtual dimension $i$. For generic $\ol{K}_{G}^{\g,\cl}, K_{\XX}^{\g,\cl}, J_{G}^{\g,\cl}, J_{\XX}^{\g,\cl}$, $\M^i_{\vp}(h_{\pm},A)$ is regular for any $i\leqslant 1$. There are two types of boundary components of $\M^1_{\vp}(h_{\pm},A)$:
\begin{enumerate}
\item The set of $\ul{u}^{\g,\cl}$ with $\cl\to +\infty$. This corresponds to the coefficient of $\TT^A$ in $\langle (\MC^{cl}\circ\GG)([\vp]), [h_-]\otimes [h_+]\rangle$.
\item The set of $\ul{u}^{\g,\cl}$ with $\cl\to 0^+$. This corresponds to the signed count of $N_+\times_{(h_+,ev_{\infty})} \M^{section}(A,\vp,\HH,J)\times_{(ev_0,h_-)}N_-$ where $0,\infty\in S^2$ are two distinct marked points.
\end{enumerate}
The gluing issue for (1) is standard. For (2), it is an instance of annulus-shrinking, and the rigorous arguments are given by Wehrheim-Woodward \cite{WW2}. We will revise their result in Appendix \ref{gluing}.

Summing over $A$, we get
\begin{equation}\label{pfsemc1}
\langle (\MC^{cl}\circ\GG)([\vp]), [h_-]\otimes [h_+]\rangle = \sum_A\#\left(N_+\times_{(h_+,ev_{\infty})} \M^{section}(A,\vp,\HH,J)\times_{(ev_0,h_-)}N_-\right) \TT^A.
\end{equation}
If $\vp$ lands in $\O G$, then by Proposition \ref{commQFT} and Proposition \ref{commwrapped}, the LHS of \eqref{pfsemc1} is equal to
\[\langle (\qcp\circ PSS\circ \MC^{op}\circ \FF)([\vp]),[h_-]\otimes [h_+]\rangle.\]
By \cite[Lemma 4.9]{S_QCC}, the RHS of \eqref{pfsemc1} is equal to
\[\langle (\qcp\circ \Se{\Lo G})([\vp]), [h_-]\otimes [h_+]\rangle.\]
Theorem \ref{MC=Se} now follows from the injectivity of $\qcp$ (Lemma \ref{qcpinj}) and the fact that $H_*(\O G;\ZZ)$ is generated additively by smooth cycles (Theorem \ref{BSthm}).
\section{Coadjoint orbits and based loop groups}\label{OG}
\subsection{Preliminaries on Lie theory}\label{prelimLie}
Let $G$ be a compact connected Lie group which is semi-simple, that is, the center $Z(G)$ of $G$ is discrete. Fix a maximal torus $T$ in $G$. Denote by $\gg$ (resp. $\t$) the Lie algebra of $G$ (resp. $T$). The adjoint representation gives rise to a set $R$ of \textit{roots} which are linear forms $\a:\t\ra\RR$. By definition, we have $-R=R$ and the following root space decomposition
\begin{equation}\label{rootdecomp}
\gg = \t\oplus \bigoplus_{[\a]\in R/_{\pm}} \gg_{[\a]}
\end{equation}
where for each $[\a]\in R/_{\pm}$, $\gg_{[\a]}:=\bigcap_{x\in\t}\ker(\ad(x)^2+4\pi^2\a(x)^2\id)$ which is known to be two dimensional. The walls $\{\a=0\}$, $\a\in R$ divide $\t$ into a finite number of open domains. The closure of each domain is called a \textit{Weyl chamber}. Throughout the discussion, we fix one of them, call it the \textit{dominant chamber}, and denote it by $\wc_0$. Then $\wc_0$ determines a set $R^+$ of representatives of $R/_{\pm}$
\[ R^+:=\{\a\in R|~\a(\wc_0)=\RR_{\geqslant 0}\}.\]
The elements of $R^+$ are called \textit{positive roots}.

Give $\gg$ an $\Ad$-invariant metric $\langle-,-\rangle$. This makes the direct sum \eqref{rootdecomp} orthogonal. The subgroup $W$ of $\Aut(\t)$ generated by the reflections across $\{\a=0\}$, $\a\in R$ is called the \textit{Weyl group}. It is well-known that $W$ acts simply transitively on the set of Weyl chambers. The metric $\langle-,-\rangle$ identifies $\t$ with $\t^{\vee}$. For each $\a\in R$, define the \textit{coroot} associated to $\a$
\[\a^{\vee}:= \frac{2\a}{\langle \a,\a\rangle}\in \t\]
which is independent of which $\Ad$-invariant metric we use. Define the lattices
\begin{align*}
 \Q &:= \{q\in\t|~\exp(q)=e\in T\} \\
 \Q_0 &:= \Span_{\ZZ}\{\a^{\vee}|~\a\in R\}.
\end{align*}
Then $\Q_0\subseteq \Q$ and they are equal if and only if $G$ is simply connected, since $\pi_1(G)\simeq \Q/\Q_0$.

Since we will be concerned with the based loop group $\O G$, we also need the affine analogue of everything we have introduced. We consider not only the walls $\{\a=0\}$, $\a\in R$ but also the \textit{affine walls} $V_{\a,k}:=\{\a=k\}$ where $\a\in R^+$ and $k\in\ZZ$. The affine walls divide $\t$ into infinitely many open bounded domains, and the closure of each of them is called a \textit{Weyl alcove}. Denote by $\D_0$ the one which sits inside $\wc_0$ and contains the origin. We call it the \textit{dominant alcove}. It is known that $\D_0$ (as well as other Weyl alcoves) is a product of simplices and is a simplex if $G$ is simple. Each of the boundary walls of $\D_0$ which does not contain the origin is given by $V_{\a,1}$ for some $\a\in R^+$. We call these roots \textit{highest roots}. The \textit{affine Weyl group} $W^{af}$ is defined to be the subgroup of $\Aff(\t)$ generated by the reflections across any $V_{\a,k}$. Similarly, $W^{af}$ acts simply transitively on the set of alcoves. The following facts are standard.
\begin{lemma}\label{stdfacts} $~$
\begin{enumerate}
\item $\D_0\cap \Q_0=\{0\}$.
\item $W^{af}$ is generated by $W$ and the translations
\[ \{t_{\a^{\vee}}|~\a\in R^+\text{ is a highest root}\}\]
where $t_{\a^{\vee}}:x\mapsto x+\a^{\vee}$.
\item $\Q_0$ is the orbit through $0\in \t$ under the action of $W^{af}$.
\end{enumerate}
\hfill$\square$
\end{lemma}

For later use, we introduce a few more notations. Pick any $a\in \mr{\wc}_0$ sufficiently close to the origin.
\begin{definition}\label{perturblengthdef} Define a function $\ell':W\ra \RR$ by
\[\ell'(w):= \sum_{\a\in R^+}\{\a(w^{-1}(a))\}\]
where $\{x\}:=x-\lfloor x\rfloor$ is the fractional part of $x$.
\end{definition}

\noindent The reason for the notation $\ell'$ is justified by Lemma \ref{perturblengthlemma} below. Recall that the standard length function $\ell:W\ra\ZZ_{\geqslant 0}$ is defined to be the number of positive roots $\a$ for which the walls $\{\a=0\}$ separate $\wc_0$ and $w\wc_0$.

\begin{lemma} \label{perturblengthlemma} For any $w\in W$, we have
\[\ell'(w) = \ell(w) +\langle a,w(\rho)\rangle\]
where $\rho:=\sum_{\a\in R^+}\a\in\t$ is the sum of positive roots.
\end{lemma}
\begin{proof}
Indeed,
\begin{align*}
\ell'(w) &= \sum_{\a\in R^+} \{\a(w^{-1}(a))\}\\
&= \sum_{\a\in R^+} \a(w^{-1}(a)) - \sum_{\a\in R^+}\lfloor\a(w^{-1}(a)) \rfloor\\
&= \langle w^{-1}(a),\rho\rangle - \sum_{\substack{\a\in R^+\\ \a(w^{-1}(a))<0}} (-1)\\
&= \langle a,w(\rho)\rangle +\ell(w^{-1})\\
&= \langle a,w(\rho)\rangle +\ell(w).
\end{align*}
\end{proof}
\noindent Let $a$ be as before.
\begin{definition}\label{defforq} For any $q\in \Q$, define
\begin{enumerate}
\item $w_q\in W$ to be the unique Weyl group element such that $q+a\in w_q\wc_0$.
\item $\deg(q)$ to be the number of affine walls intersected by the open segment joining $-a$ and $q$.
\end{enumerate}
\end{definition}
\begin{lemma}\label{lemmaforq} For any $q\in \Q$, we have
\[\deg(q) = \langle q,w_q(\rho)\rangle -\ell(w_q)\]
where $\rho:=\sum_{\a\in R^+}\a\in\t$.
\end{lemma}
\begin{proof}
It is not hard to see that
\begin{equation}\label{proofforq}
\deg(q)=\sum_{\a\in R^+}\lfloor\a(w_q^{-1}(q+a))\rfloor.
\end{equation}
The summand in \eqref{proofforq} is equal to $\a(w_q^{-1}(q))+\lfloor \a(w_q^{-1}(a))\rfloor$. The sum corresponding to the first term is equal to $\langle q,w_q(\rho)\rangle$, and the sum corresponding to the second is equal to, as in the proof of Lemma \ref{perturblengthlemma}, $-\ell(w_q)$.
\end{proof}


\subsection{Coadjoint orbits}
Let $x\in \t$. Define $\OO_x:= G\cdot x\subset \gg^{\vee}$, the coadjoint orbit passing through $x$. Then $\OO_x\simeq G/G_x$ where $G_x$ is the centralizer of $x$. By Hopf's theorem, $G_x$ is connected. Define $R_x:=\{\a\in R|~\a(x)=0\}$ and put $\t_x:=\bigcap_{\a\in R_x}\{\a=0\}$. For any subset $S\subseteq R$, define
\begin{align*}
\Q_S&:=\Span_{\ZZ}\{\a^{\vee}|~\a\in S\}\\
P^{\vee}_S&:=\{q\in \Q_S\otimes\RR|~\a(q)\in\ZZ\text{ for all }\a\in S\}.
\end{align*}
Notice that $\Q_R=\Q_0$. Define $G_x^{\ad}:=G_x/Z(G_x)$ and $\gg_x^{\ad}:=\Lie(G_x^{\ad})$. Observe that $\gg_x^{\ad}$ can be identified with the orthogonal complement of $\t_x$ in $\gg_x:=\Lie(G_x)$ via the quotient map $\gg_x\ra \gg_x^{\ad}$. Denote by $\t_x^{\perp}$ the orthogonal complement of $\t_x$ in $\t$. Then $\t_x^{\perp}$ is a maximal torus in $\gg_x^{\ad}$ via the above identification and $R_x$ is the root system of $(\gg_x^{\ad},\t_x^{\perp})$.

Define a complex structure $I_{\OO_x,x}$ on $\ds T_x\OO_x=\bigoplus_{\a\in R,~\a(x)>0}\gg_{\a}$ by
\begin{equation}\label{CScoadj}
I_{\OO_x,x} :=\sum_{\substack{\a\in R\\\a(x)>0} } \frac{-1}{2\pi\a(x)} \ad(x)|_{\gg_{\a}}.
\end{equation}
It is well-known that $I_{\OO_x,x}$ can be extended to a unique $G$-invariant almost complex structure $I_{\OO_x}$ on $\OO_x$ which is integrable and compatible with the KSS symplectic form. Moreover, $(\OO_x,I_{\OO_x})$ is a smooth projective variety and admits a holomorphic action by the complexification $G_{\CC}$ of $G$.

Recall $\wc_0$ is the dominant chamber. Pick $a\in \mr{\wc}_0$. Then $\langle a,-\rangle$ defines a Morse function on $\OO_x$ with critical set $\crit_{\OO_x}=W\cdot x$. For any $z\in\crit_{\OO_x}$, define $\mathcal{S}_z^{\OO_x}$ (resp. $\UU_z^{\OO_x}$) to be the stable (resp. unstable) submanifold\footnote{They are also called descending and ascending submanifolds respectively.} passing through $z$ with respect to this Morse function and the metric $g_{\OO_x}(-,-):=\w_{\OO_x}(-,I_{\OO_x}(-))$. It is well-known that these submanifolds are pseudocycles and the sets $\{[\mathcal{S}_z^{\OO_x}]\}_{z\in\crit_{\OO_x}}$ and $\{[\UU_z^{\OO_x}]\}_{z\in\crit_{\OO_x}}$ are $\ZZ$-bases of $H_*(\OO_x;\ZZ)$.

\begin{definition} Define the \textit{Schubert classes} $\s_z^{\OO_x}\in H^*(\OO_x;\ZZ)$ to be the Poincar\'e dual of $[\mathcal{S}_z^{\OO_x}]$.
\end{definition}

\noindent Since the Morse function $\langle a,-\rangle$ and the metric $g_{\OO_x}$ satisfy the Morse-Smale condition (another well-known fact), it follows that $\{\s_z^{\OO_x}\}_{z\in\crit_{\OO_x}}$ is the dual basis of $\{[\UU_z^{\OO_x}]\}_{z\in\crit_{\OO_x}}$ with respect to the natural pairing $H^*(\OO_x)\otimes H_*(\OO_x)\ra\ZZ$.
\begin{remark}\label{BSresolve} The closure of each stable or unstable submanifold admits a resolution, called the \textit{Bott-Samelson resolution} \cite{BS}. It follows that the homology classes $[\mathcal{S}_z^{\OO_x}]$ and $[\UU_z^{\OO_x}]$, $z\in\crit_{\OO_x}$ can in fact be represented by smooth cycles. Since we will not need their explicit forms, we do not recall their construction. (But we will do so for the based loop group $\O G$ in the next subsection.)
\end{remark}

Recall $\Nov{\OO_x}{\O G}$ is defined to be the group of homotopy classes of pairs $(\vp, u)$ where $\vp\in\O G$ and $u$ is a section of $P_{\vp}(\OO_x)$.

\begin{lemma} \label{GPNovforOO} There is a bijective correspondence between $\Nov{\OO_x}{\O G}$ and $\Q/\Q_{R_x}$.
\end{lemma}
\begin{proof} Recall the construction of $s_{\vp,u_{\pm}}$ which defines a bijective correspondence between the set of sections of $P_{\vp}(\OO_x)$ and pairs $(u_-,u_+)$ satisfying \eqref{pair}. It follows that $\Nov{\OO_x}{\O G}$ is in a bijective correspondence with the set of homotopy classes of pairs $(\vp,u_{\pm})$ satisfying \eqref{pair}. To determine the latter set, consider the moment correspondence associated to the \Ham $G$-manifold $\OO_x$, expressed as the graph of a symplectic quotient of $T^*G$:
\begin{equation}
\begin{array}{ccc}
G\times \OO_x&\ra& \OO_x\times \OO_x \\ \\ [-1em]
 (g,y) &\mapsto &(y,g\cdot y).
\end{array}
\end{equation}
It is a fiber bundle with fibers isomorphic to $G_x$. By the homotopy lifting property, we see that the set we want to determine is equal to $\pi_1(G_x)$. More precisely, for any based loop $\vp\in \O G_x$, the corresponding homotopy class of pairs is given by $(\vp,u_{\pm}\equiv x)$. It remains to show that $\pi_1(G_x)\simeq \Q/\Q_{R_x}$. This is standard but we still give a proof here.

The fibration $T\hookrightarrow G_x\ra G_x/T$ induces a short exact sequence
\begin{equation}\label{GPSES}
0\ra \pi_2(G_x/T)\ra \pi_1(T)\ra \pi_1(G_x)\ra 0
\end{equation}
(using $\pi_2(G_x)=\pi_1(G_x/T)=0$). Notice that $G_x/T\simeq \widetilde{G_x^{\ad}}/T'$ where $\widetilde{G_x^{\ad}}$ is the universal cover of $G_x^{\ad}$ and $T'$ is the maximal torus in $\widetilde{G_x^{\ad}}$ covering $\exp(\t_x^{\perp})\simeq T/Z(G_x)\subseteq G_x^{\ad}$. Our result follows from the well-known fact that $\pi_1(T')$ is spanned by the coroots of $G_x$ and the isomorphism $\pi_2(\widetilde{G_x^{\ad}}/T')\simeq \pi_1(T')$ which is deduced from a short exact sequence similar to \eqref{GPSES}.
\end{proof}

\begin{definition} For any $q+\Q_{R_x}\in \Q/\Q_{R_x}$, define $A^{\OO_x}_{x,q}\in \Nov{\OO_x}{\O G}$ by
\[A^{\OO_x}_{x,q}:= [\vp_q, s_{\vp=\vp_q, u_{\pm}\equiv x} ] \]
where $\vp_q:\eit\mapsto\exp(\tt\cdot q/2\pi)\in G$.
By Lemma \ref{GPNovforOO}, the assignment $q+\Q_{R_x}\mapsto A^{\OO_x}_{x,q}$ defines a bijection between $\Q/\Q_{R_x}$ and $\Nov{\OO_x}{\O G}$. We will sometimes denote $A^{\OO_x}_{x,q+\Q_{R_x}}$ to emphasize that the homotopy class is defined for cosets in $\Q/\Q_{R_x}$.
\end{definition}

\begin{remark} It is useful to remember the based point $x$ in the definition of $A^{\OO_x}_{x,q}$. If we take another based point $x'\in \OO_x$ which also lies in $\t$, then $x'=w(x)$ for some $w\in W$, and we have
\[ A^{\OO_x}_{x,q+\Q_{R_x}} = A^{\OO_x}_{x',w(q)+\Q_{R_{x'}}}.\]
If the based point is dropped, we mean that it is the unique intersection point of $\OO_x$ and $\wc_0$.
\end{remark}

\subsection{Based loop groups}\label{BScycle}
Let $\cu:[0,1]\ra \t$ be a piecewise smooth curve.
\begin{definition} \label{nice} We say that $\cu$ is \textit{nice} if the following conditions are satisfied.
\begin{enumerate}
\item If $\cu$ intersects an affine wall $V_{\a,k}$ at $t=t_0\in [0,1]$, then $\cu$ is smooth at $t_0$, the intersection is transverse, and $\cu^{-1}(V_{\a,k})=\{t_0\}$.
\item For any $t\in (0,1)$, $\cu(t)$ is contained in at most one affine wall.
\end{enumerate}
\end{definition}

\noindent Let $\cu$ be a nice curve such that $\cu(1)\in \Q$. Consider
\[ 0 =t_0< t_1<\cdots < t_{\ell} < t_{\ell +1}=1\]
where $\{t_i|~i=1,\ldots, \ell\}$ is the set of all points in $(0,1)$ for which $\cu(t_i)$ lies in an affine wall $V_{\a_i,k_i}$ for some $\a_i\in R^+$ and $k_i\in\ZZ$. Notice that by Condition (2) in Definition \ref{nice}, each pair $(\a_i,k_i)$ is well-defined, and by Condition (1), these pairs are pairwise distinct. Put $\cu_i:=\cu|_{[t_i,t_{i+1}]}$ and denote by $G_{\a_i}$ the centralizer of $V_{\a_i,0}$. It is well-known that $G_{\a_i}/T\simeq \mathbb{P}^1$ for any $i$. Define a $T^{\ell}$-action on the product $G_{\a_1}\times \cdots\times G_{\a_{\ell}}$ by
\[ (x_1,\ldots,x_{\ell})\cdot (g_1,\ldots, g_{\ell}):= (g_1x_1,x_1^{-1}g_2x_2,\ldots,x_{\ell -1}^{-1}g_{\ell} x_{\ell})\]
for any $(x_1,\ldots,x_{\ell})\in T^{\ell}$ and $ (g_1,\ldots, g_{\ell})\in G_{\a_1}\times \cdots\times G_{\a_{\ell}}$. This action is free, and the standard notation for the resulting quotient is $G_{\a_1}\times_T\cdots\times_TG_{\a_{\ell}}/T$. But since it will appear for many times, we shall denote it by $[G_{\a_1}:\cdots:G_{\a_{\ell}}]$ for simplicity. Using the aforementioned fact that $G_{\a_i}/T\simeq \mathbb{P}^1$, it is not hard to deduce that $[G_{\a_1}:\cdots:G_{\a_{\ell}}]$ is a smooth projective variety of complex dimension $\ell$ which has a structure of iterated $\mathbb{P}^1$-bundles. In particular, it carries a canonical orientation induced by the complex structure.

We introduce three more notations before defining Bott-Samelson cycles: For any $g_1,g_2\in G$, define $\O_{g_1\ra g_2}G$ to be the space of paths in $G$ from $g_1$ to $g_2$; for any $g\in G$ and $\eta\in \gg$, put $g\cdot \eta:= \Ad(g)\cdot\eta$ where $\Ad$ is the adjoint action; and for any $\eta\in \gg$, put $e^{\eta}:=\exp(\eta)\in G$.
\begin{definition} \label{BScycleDef}  The \textit{Bott-Samelson cycle} associated to $\cu$ is defined by
\[
\begin{array}{cccl}
\BS_{\cu} : &[G_{\a_1}:\cdots:G_{\a_{\ell}}] &\ra& \O_{e^{\cu(0)}\ra e} G\\ \\ [-1em]
 & [g_1:\cdots:g_{\ell}] &\mapsto &e^{\cu_0}\# e^{g_1\cdot \cu_1}\# e^{(g_1g_2)\cdot \cu_2}\#\cdots\# e^{(g_1\cdots g_{\ell})\cdot \cu_{\ell}}
\end{array}
 \]
 where $\#$ denotes path concatenations.
\end{definition}
\noindent It is clear that every element of the image of $\BS_{\cu}$ is a piecewise smooth curve in $G$ from $e^{\cu(0)}$ to $e$.

Define $\O G:=\O_{e\to e}G$. We move on to construct an explicit basis of the group $H_*(\O G;\ZZ)$. For any curve $\cu$ and any map $f$ from a space $X$ to the space of curves which start at the ending point of $\cu$, define $\cu\# f$ to be the map $X\ni x\mapsto \cu\# f(x)$. For any $\eta,\zeta\in \t$, we denote by $[\eta,\zeta]$\footnote{This is certainly not a good notation since it coincides with the Lie bracket. But as we will see, it will not make any confusion because we will rarely talk about the Lie bracket.} the line segment from $\eta$ to $\zeta$. Fix a generic element $a_0$ in the interior $\mr{\D}_0$ of the dominant alcove $\D_0$ such that for any $q\in \Q$ the segment $[-a_0,q]$ is nice.

\begin{definition} For any $q\in \Q$, define
\[ x_q:= \left[ e^{[0,-a_0]}\# \BS_{[-a_0,q]} \right]\in H_{2\deg(q)}(\O G;\ZZ)\]
where the outer bracket is the fundamental class of the cycle inside with the orientation induced by the complex structure on the domain and $\deg(q)$ is defined in Definition \ref{defforq}(2).
\end{definition}

\begin{theorem} \label{BSthm} \cite{BS} The set $\{x_q\}_{q\in \Q}$ is a $\ZZ$-basis of $H_*(\O G;\ZZ)$. \hfill$\square$
\end{theorem}

In order to prove one of our main theorems (Theorem \ref{B}), we have to construct, for each $q$, a cycle representing $x_q$ by applying the above construction to the segment $[0,q]$. However, $[0,q]$ may not be nice, and even if it is nice, $\BS_{[0,q]}$ may not represent $x_q$. For example, take $q$ to be the coroot associated to the unique highest root in type $A_2$. Observe that the construction of $\BS_{\cu}$ depends only on the collection of affine walls intersected by $\cu$ and an ordering of the moments when these intersections take place. In our case, they are specified as follows. Let $q\in \Q$ be given. Choose $a\in \mr{\D}_0$ such that the segment $[-sa,q]$ is nice for any $s\in (0,1]$. Such a point exists, since the set of $a'\in\mr{\D}_0$ for which $[-a',q]$ is not nice is contained in a finite union of proper affine subspaces. Then the collection of affine walls intersected by $[-sa,q]$ (in the interior) as well as their ordering are independent of $s\in (0,1]$. We denote this collection by $\mathcal{V}$. By letting $s\to 0^+$, we see that $[0,q]$ also intersects every member of $\mathcal{V}$ once and transversely. Moreover, the ordering of $\mathcal{V}$ determined by positive $s$ is preserved, provided that we allow $[0,q]$ to intersect more than one affine walls at the same time, and possibly at $t=0$. With this ordering, we are able to define the Bott-Samelson cycle as in Definition \ref{BScycleDef}.
\begin{definition}\label{BSq} Define $\BS_q$ to be the cycle thus constructed.
\end{definition}

\begin{example} Consider the above $A_2$ example, that is, $G=SU(3)$. Let $\a_1,\a_2$ be the simple roots. Then the highest root $\a_0$ is equal to $\a_1+\a_2$. Define $q$ to be the coroot associated to $\a_0$. If we take $a$ to be a point in $\mr{\D}_0$ closer to $\a_1^{\vee}$ than $\a_2^{\vee}$, then the affine walls intersected by the segment $[-a,q]$ are $V_{\a_2,0}, V_{\a_0,0}, V_{\a_1,0}, V_{\a_0,1}$ in this order. It follows that the corresponding $\BS_q$ is given by
\[
\begin{array}{cccl}
\BS_q : &[G_{\a_2}:G_{\a_0}:G_{\a_1}:G_{\a_0}] &\ra& \O SU(3)\\ \\ [-1em]
 & [g_1:g_2:g_3:g_4] &\mapsto & e^{(g_1g_2g_3)\cdot \left[0,\frac{q}{2}\right]}\# e^{(g_1g_2g_3g_4)\cdot \left[\frac{q}{2},q\right]}
\end{array}.
 \]
\end{example}

\begin{remark} Notice that $\BS_q$ may not be unique as different choices of $a$ may yield different orderings. (In the above $A_2$ example, if $a$ is closer to $\a_2^{\vee}$ than $\a_1^{\vee}$, then the ordering will be $V_{\a_1,0}, V_{\a_0,0}, V_{\a_2,0}, V_{\a_0,1}$.) Nevertheless, it does not affect our application, namely the proof of Theorem \ref{B}. In fact, every possible $\BS_q$ is a solution to the min-max problem considered in the theorem.
\end{remark}

\noindent It is clear that $s\mapsto e^{[0,-sa]}\# \BS_{[-sa,q]}$ defines a homotopy between $\BS_q$ and $e^{[0,-a]}\# BS_{[-a,q]}$ so that these cycles are homologous. However, it is not obvious that they represent $x_q$, due to the ordering issue we have encountered. Fortunately, this discrepancy vanishes in homology level.

\begin{proposition}\label{indepofnice} If $\cu^0$ and $\cu^1$ are nice curves such that $\cu^0(0)=\cu^1(0)$ and $\cu^0(1)=\cu^1(1)\in \Q$, then $[\BS_{\cu^0}]=[\BS_{\cu^1}]$.
\end{proposition}

\begin{proof} Denote by $x$ and $y$ the common starting point and ending point of the curves $\cu^0$ and $\cu^1$ respectively. Define
\[ S:=\{(\a,k)|~\a\in R^+, ~k\in\ZZ, ~(\a(x)-k)(\a(y)-k)<0 \}.\]
This set labels the affine walls which are intersected by $\cu^0$ (resp. $\cu^1$) in their interior. For any subset $S'\subseteq S$, define an affine subspace $V_{S'}:=\bigcap_{(\a,k)\in S'} V_{\a,k}$. Take a homotopy $\{\cu^s\}_{s\in [0,1]}$ with fixed endpoints joining $\cu^0$ and $\cu^1$ such that every $\cu^s$ satisfies Condition (1) of Definition \ref{nice}, and as a map $[0,1]^2\ra \t$, this homotopy intersects $V_{S'}$ transversely for any $S'\subseteq S$. Call a point $s\in (0,1)$ \textit{bad} if $\cu^s$ is not nice. By the transversal condition, the set of bad points is finite, and for any bad $s$, $\cu^s$ intersects $V_{S'}$ for at least one $S'$ with $\codim V_{S'} =2$ and for no $S'$ with $\codim V_{S'} \geqslant 3$. It is not hard to see that $\BS_{\cu^{s_1}}$ is homotopic to $\BS_{\cu^{s_2}}$ whenever $[s_1,s_2]$ contains no bad points. Thus, it suffices to examine what happens when $s$ crosses a bad point.

Let $s_0$ be bad. For simplicity, assume $\cu^{s_0}$ intersects only one affine subspace $V_{S'}$ of codimension two. The arguments we going to present work well for the general case. Notice that $S'$ is not unique but is unique if we assume it is maximal among all subsets of $S$ giving the same affine subspace. Then
\[R':=\{\pm\a|~\a\in R^+,~\exists k\in\ZZ\text{ s.t. }(\a,k)\in S'\}.\]
is a rank-two subroot system of $R$, and hence it can only be $A_1\times A_1, A_2, B_2$ or $G_2$. Let us deal with the $A_2$ case and leave other cases to the readers.

In this case, $|S'|=3$. For any $s$ near but not equal to $s_0$, consider, as before, the moments at which $\cu^s$ intersects an affine wall:
\[ 0= t_0^s< t_1^s<\cdots <t_{\ell}^s< t_{\ell +1}^s =1.\]
Notice that for every $i$, the pair labelling the unique affine wall containing $\cu^s(t_i^s)$ is locally constant in $s$ away from $s_0$. Thus, we simply denote by $(\a_i^-,k_i^-)$ (resp. $(\a_i^+,k_i^+)$) the pair for $s<s_0$ (resp. $s>s_0$). At $s=s_0$, the numbers $t_i^{s_0}$ are in fact well-defined but they coincide for precisely three indices which are consecutive. Let these indices be $r,r+1,r+2$. Then $(\a_i^-,k_i^-)=(\a_i^+,k_i^+)$ for $i\ne r,r+2$ and $(\a_r^{\pm},k_r^{\pm})=(\a_{r+2}^{\mp},k_{r+2}^{\mp })$. Therefore, the Bott-Samelson cycle $\BS_{\cu^s}$ with $s$ slightly smaller than $s_0$ is homotopic to the cycle
\begin{align}\label{BScycleminus}
& X^-:= [G_{\a^-_1}:\cdots:G_{\a^-_r}:G_{\a^-_{r+1}}:G_{\a^-_{r+2}}:\cdots: G_{\a^-_{\ell}}] \ra \O_{e^x\ra e} G \\
 & [g_1:\cdots:g_{\ell}] \mapsto e^{\cu^{s_0}_0}\# \cdots \# e^{(g_1\cdots g_{r-1})\cdot \cu^{s_0}_{r-1}} \#  e^{(g_1\cdots g_{r-1} g_r g_{r+1}g_{r+2})\cdot \cu^{s_0}_{r+2}} \#  \cdots \# e^{(g_1\cdots g_{\ell})\cdot \cu^{s_0}_{\ell}} \nonumber
\end{align}
where $\cu_i^{s_0}:=\cu^{s_0}|_{[t_i^{s_0},t_{i+1}^{s_0}]}$, and the Bott-Samelson cycle $\BS_{\cu^s}$ with $s$ slightly larger than $s_0$ is homotopic to the cycle
\begin{align}\label{BScycleplus}
& X^+:= [G_{\a^-_1}:\cdots:G_{\a^-_{r+2}}:G_{\a^-_{r+1}}:G_{\a^-_r}:\cdots: G_{\a^-_{\ell}}] \ra \O_{e^x\ra e} G \\
 & [g_1:\cdots: g_{r+2}: g_{r+1}: g_r: \cdots: g_{\ell}] \mapsto e^{\cu^{s_0}_0}\# \cdots \#  e^{(g_1\cdots g_{r-1} g_{r+2} g_{r+1}g_r)\cdot \cu^{s_0}_{r+2}} \#  \cdots \# e^{(g_1\cdots g_{\ell})\cdot \cu^{s_0}_{\ell}}. \nonumber
\end{align}
The proof is complete by noticing that both cycles \eqref{BScycleminus} and \eqref{BScycleplus} factor through the map
\begin{align*}
& X_{S'}:= [G_{\a^-_1}:\cdots:G_{\a^-_{r-1}}:G_{S'}:G_{\a^-_{r+3}}:\cdots: G_{\a^-_{\ell}}] \ra \O_{e^x\ra e} G \\
 & [g_1:\cdots: g_{r-1}: g: g_{r+3}: \cdots: g_{\ell}] \mapsto e^{\cu^{s_0}_0}\# \cdots \#  e^{(g_1\cdots g_{r-1} g)\cdot \cu^{s_0}_{r+2}} \#  \cdots \# e^{(g_1\cdots g_{r-1}gg_{r+3}\cdots g_{\ell})\cdot \cu^{s_0}_{\ell}}
\end{align*}
where $G_{S'}$ is the centralizer of $\bigcap_{(\a,k)\in S'}V_{\a,0}=V_{\a^-_r,0}\cap V_{\a^-_{r+1},0}\cap V_{\a^-_{r+2},0}$, and the fact that the multiplication maps
\[ \begin{array}{rcl}
 X^- &\ra & X_{S'} \\  \\ [-1em]
 [g_1:\cdots: g_{r}: g_{r+1}: g_{r+2}: \cdots: g_{\ell}] &\mapsto & [g_1:\cdots: g_{r} g_{r+1} g_{r+2}: \cdots: g_{\ell}]
\end{array} \]
and
\[ \begin{array}{rcl}
 X^+&\ra & X_{S'}  \\  \\ [-1em]
 [g_1:\cdots: g_{r+2}: g_{r+1}: g_{r}: \cdots: g_{\ell}] &\mapsto & [g_1:\cdots: g_{r+2} g_{r+1} g_{r}: \cdots: g_{\ell}]
\end{array} \]
are birational\footnote{This can be deduced from the case without $G_{\a^-_i}$ for $i\neq r,r+1,r+2$ which is nothing but the Bott-Samelson resolutions $[G_{\a^-_r}:G_{\a^-_{r+1}}:G_{\a^-_{r+2}}]\ra G_{S'}/T\simeq F\ell(1,2;3)$ etc.}.
\end{proof}

\begin{remark} The use of a homotopy intersecting every affine wall transversely can be found in \textit{Tits building theory}.
\end{remark}

\begin{corollary}\label{xqindep} For any $q\in \Q$, $x_q$ is independent of the generic element $a_0\in \mr{\D}_0$ and is equal to $[\BS_q]$. \hfill$\square$
\end{corollary}

The following multiplication formula is due to Magyar \cite{Magyar} and is essential for the proof of Theorem \ref{computeG/Tintro} and hence Theorem \ref{B}. We reproduce his proof here because our definition of $x_q$ is different from his, although one can show that the two definitions are equivalent.

\begin{proposition} Let $q_0,q_1\in \Q$. If $q_0\in \wc_0$, then we have
\begin{equation}\label{multiBS}
x_{q_0}\bulletsmall x_{q_1} = x_{q_1+w_{q_1}(q_0)}
\end{equation}
where $\bulletsmall $ is the Pontryagin product and $w_q\in W$ is defined in Definition \ref{defforq}.
\end{proposition}

\begin{proof}
We first assume $q_0$ lies in the interior of $\wc_0$. Put $q_2:=q_1+w_{q_1}(q_0)$. Let $a\in\mr{\D}_0$ be a generic element. For any $\e>0$, define $y_{\e}:= (1-\e)q_1-\e a$. Then $y_{\e}\in q_1+w_{q_1}(-\mr{\D}_0)$ for small $\e$.

\vspace{0.3cm}
\noindent \textbf{Claim.} The curve $\cu^{\e}:=[-a,y_{\e}]\#[y_{\e},q_2]$ is nice.
\begin{proof}
The only non-trivial part is to show that there does not exist $(\a,k)$ such that $-a,q_2$ lie in the same side of $V_{\a,k}$ and $y_{\e}$ lies in the other side. Suppose the contrary that such a pair exists. We have either
\begin{equation}\label{side1}
-\a(a), \a(q_2) < k < (1-\e)\a(q_1)-\e\a(a)
\end{equation}
or
\begin{equation}\label{side2}
(1-\e)\a(q_1)-\e\a(a)< k < -\a(a), \a(q_2).
\end{equation}
Observe also that $q_1$ and $w_{q_1}(q_0)$ lie in the same Weyl chamber $w_{q_1}\wc_0$, and $w_{q_1}(q_0)$ even lies in the interior. It follows that $\a(q_1)\ne\a(q_2)$. Moreover, they lie in $\RR_{\geqslant 0}$ or $\RR_{\leqslant 0}$ simultaneously, and $\a(q_2)$ is always farther away from $0$ than $\a(q_1)$. In case \eqref{side1}, we have $\a(q_2)<\a(q_1)$ (since $\e$ is small) so that $\a(q_1)\leqslant 0$. But then $(1-\e)\a(q_1)-\e\a(a)\leqslant -\e\a(a)$, and hence $-\a(a)<k<-\e\a(a)$, a contradiction, since $\a(a)\in (0,1)$. The case \eqref{side2} is similar.
\end{proof}

By Proposition \ref{indepofnice}, $x_{q_2}=[e^{[0,-a]}\#\BS_{\cu^{\e}}]$. Put $a_{\e}:=\e w_{q_1}^{-1}(q_1+a)$. Then $a_{\e}\in\mr{\D}_0$ for sufficiently small $\e$. Observe that $[y_{\e},q_2]=q_1+w_{q_1}[-a_{\e},q_0]$. Since $G$ is connected and $\exp(q_1)=e$, we have
\[x_{q_0}=[e^{[0,-a_{\e}]}\#\BS_{[-a_{\e},q_0]}]=[e^{q_1+w_{q_1}[0,-a_{\e}]}\#\BS_{q_1+w_{q_1}[-a_{\e},q_0]}]=[e^{[q_1,y_{\e}]}\#\BS_{[y_{\e}, q_2]}].\]
Thus it suffices to show, for small $\e$,
\begin{equation}\label{wts}
[\BS_{\cu^{\e}}]= [\BS_{[-a,q_1]}]\bulletsmall [e^{[q_1,y_{\e}]}\#\BS_{[y_{\e},q_2]}].
\end{equation}
Consider the moments
\[ 0=t_0<t_1\leqslant \cdots \leqslant t_{\ell}<t_{\ell+1}=1\]
at which the curve $\cu^0=[-a,q_1]\#[q_1,q_2]$ intersects those affine walls intersected by $\cu^{\e}$ (in its interior). These numbers are not pairwise distinct, and among those with multiplicities $>1$, we look at the one which is sent to $q_1$ under $\cu^0$:
\begin{equation}\label{muti=n}
t_{r+1}=t_{r+2}=\cdots=t_{r+n}.
\end{equation}
Since $q_0\in \mr{\wc}_0$, $n$ is equal to the number of positive roots. Denote by $(\a_i,k_i)$ the corresponding pairs. Due to the coincidence of some $t_i's$, an ordering of $(\a_i,k_i)$ has to be specified, and we choose the one determined by the segment $[y_{\e},q_2]$ for small $\e$. (See how $\BS_q$ is defined in Definition \ref{BSq}.) Then $R^+=\{\a_{r+1},\ldots, \a_{r+n}\}$.

Letting $\e\to 0$, the LHS of \eqref{wts} is represented by
\begin{equation}\label{left}
 \begin{array}{rcl}
 [G_{\a_1}:\cdots:G_{\a_{\ell}}]&\ra &\O_{e^{-a}\ra e}G  \\  \\ [-1em]
 [g_1:\cdots: g_{\ell}] &\mapsto & \ds\left( \bighash_{i=0}^r e^{(g_1\cdots g_i)\cdot \cu_i^0}\right) \# \left(\bighash_{i=r+n}^{\ell} e^{\vec{g} (g_{r+1}\cdots g_i)\cdot \cu_i^0} \right)
\end{array}
\end{equation}
where $\vec{g}:=g_1\cdots g_r$, and the RHS of \eqref{wts} is represented by
\begin{equation}\label{right}
 \begin{array}{rcl}
 [G_{\a_1}:\cdots: G_{\a_r}] \times [G_{\a_{r+1}}:\cdots: G_{\a_{\ell}}]&\ra &\O_{e^{-a}\ra e}G  \\  \\ [-1em]
 ([g_1:\cdots:g_r],[g_{r+1}:\cdots: g_{\ell}] )&\mapsto & \ds\left( \bighash_{i=0}^r e^{(g_1\cdots g_i)\cdot \cu_i^0}\right) \# \left(\bighash_{i=r+n}^{\ell} e^{ (g_{r+1}\cdots g_i)\cdot \cu_i^0} \right) .
\end{array}
\end{equation}
Consider the multiplication map
\[ m: [G_{\a_1}:\cdots: G_{\a_{\ell}}]\ra [G_{\a_1}:\cdots: G_{\a_r}:G:G_{\a_{r+n+1}}:\cdots : G_{\a_{\ell}}].\]
Then $m$ is birational. Moreover, \eqref{left} is the composite of $m$ with the map
\begin{equation}\label{left1}
 \begin{array}{rcl}
 [G_{\a_1}:\cdots: G_{\a_r}:G:G_{\a_{r+n+1}}:\cdots: G_{\a_{\ell}}]&\ra &\O_{e^{-a}\ra e}G  \\  \\ [-1em]
 [g_1:\cdots: g_r:g:g_{r+n+1}:\cdots: g_{\ell}] &\mapsto & \ds\left( \bighash_{i=0}^r e^{(g_1\cdots g_i)\cdot \cu_i^0}\right) \# \left(\bighash_{i=r+n}^{\ell} e^{\vec{g}g (g_{r+n+1}\cdots g_i)\cdot \cu_i^0} \right).
\end{array}
\end{equation}
Similarly, the multiplication map
\[ m': [G_{\a_{r+1}}:\cdots: G_{\a_{\ell}}]\ra [G:G_{\a_{r+n+1}}:\cdots: G_{\a_{\ell}}]\]
is birational and \eqref{right} is the composite of $\id\times m'$ with the map
\begin{equation}\label{right1}
 \begin{array}{rcl}
 [G_{\a_1}:\cdots: G_{\a_r}] \times [G:G_{\a_{r+n+1}}:\cdots: G_{\a_{\ell}}]&\ra &\O_{e^{-a}\ra e}G  \\  \\ [-1em]
 ([g_1:\cdots:g_r],[g: g_{r+n+1}:\cdots: g_{\ell}] )&\mapsto & \ds\left( \bighash_{i=0}^r e^{(g_1\cdots g_i)\cdot \cu_i^0}\right) \# \left(\bighash_{i=r+n}^{\ell} e^{ g(g_{r+n+1}\cdots g_i)\cdot \cu_i^0} \right) .
\end{array}
\end{equation}
Observe that the map
\begin{equation}\nonumber
 \begin{array}{rcl}
 [G_{\a_1}:\cdots: G_{\a_r}] \times [G:G_{\a_{r+n+1}}:\cdots: G_{\a_{\ell}}]&\ra & [G_{\a_1}:\cdots: G_{\a_r}:G:G_{\a_{r+n+1}}:\cdots: G_{\a_{\ell}}] \\  \\ [-1em]
 ([g_1:\cdots:g_r],[g: g_{r+n+1}:\cdots: g_{\ell}] )&\mapsto & ([g_1:\cdots:g_r: (\vec{g})^{-1}g:g_{r+n+1}:\cdots: g_{\ell}] ).
\end{array}
\end{equation}
is a homeomorphism and commutes with \eqref{left1} and \eqref{right1}. This completes the proof for the case $q_0\in \mr{\wc}_0$.

For the general case, pick $q_0'\in \Q\cap\mr{\wc}_0$ so that $q_0+q_0'\in \Q\cap\mr{\wc}_0$. By what we have just proved, the map $x_{q_0'}\bulletsmall -$ from $H_*(\O G;\ZZ)$ into itself is injective, $x_{q_0+q_0'}=x_{q_0}\bulletsmall x_{q_0'}$, and
\begin{equation}\label{multilast}
x_{q_0+q_0'}\bulletsmall x_{q_1} = x_{q_1+w_{q_1}(q_0+q_0')}.
\end{equation}
Since $w_{q_1}=w_{q_1+w_{q_1}(q_0)}$, the RHS of \eqref{multilast} is equal to $x_{q_1+w_{q_1}(q_0)+w_{q_1+w_{q_1}(q_0)}(q_0')}=x_{q_0'}\bulletsmall x_{q_1+w_{q_1}(q_0)}$, and hence
\[x_{q_0'}\bulletsmall (x_{q_0}\bulletsmall x_{q_1})=(x_{q_0'}\bulletsmall x_{q_0})\bulletsmall x_{q_1}= x_{q_0+q_0'}\bulletsmall x_{q_1} = x_{q_1+w_{q_1}(q_0+q_0')}= x_{q_0'}\bulletsmall x_{q_1+w_{q_1}(q_0)}.\]
The proof of the proposition is complete.
\end{proof}
\subsection{Computation of Abbondandolo-Schwarz isomorphism for G}\label{MCcircF}
We first recall what was proved in \cite{BCL} for $\O G$. Define $\rho:=\sum_{\a\in R^+}\a$, regarded as an element of $\t$ via the $\Ad$-invariant metric fixed at the beginning. Pick $x_0\in \RR_{>0}\rho$. It is well-known that $\rho$ lies in $\mr{\wc_0}$ so that $\OO_{x_0}$ is diffeomorphic to $G/T$. Moreover, the KSS symplectic form on $\OO_{x_0}$ is monotone. Let $a\in\mr{\D}_0$ be a generic element. Define $L:=T_e^*G$ and $L':=T_{e^{-a}}^*G$. Our computation of $PSS\circ \MC\circ \FF$ for $\OO_{x_0}$ depends on the similarly-constructed Abbondandolo-Schwarz isomorphism
\[\FF' : H_{-*}(\O' G;\ZZ) \ra HW^*(L',L) \]
where $\O' G:=\O_{e^{-a}\ra e}G$.

\begin{remark} \label{piimpliesnopi} Any result about $\FF'$ yields the same result about $\FF$. Observe that $H_{-*}(\O'G;\ZZ)$ (resp. $HW^*(L',L)$) is a module over the ring $H_{-*}(\O G;\ZZ)$ (resp. $HW^*(L,L)$), and $\FF'$ is a module homomorphism with respect to the ring homomorphism $\FF$. Moreover, associated to the shortest geodesic $e^{[-a,0]}$, the elements $(x')_0^{BS}$ and $(x')_0^F$ of these modules defined below are free generators as modules over the corresponding rings. It is obvious that $\FF'((x')_0^{BS})=(x')_0^F$. Therefore, these elements induce isomorphisms (by multiplication) $H_{-*}(\O G;\ZZ)\simeq H_{-*}(\O'G;\ZZ)$ and $HW^*(L,L)\simeq HW^*(L',L)$ which are compatible with $\FF$ and $\FF'$.
\end{remark}

\noindent We define additive generators of these modules and rings as follows. Let us start with the modules.

\begin{itemize}
\item \underline{$H_{-*}(\O'G;\ZZ)$:} For any $q\in \Q$, define $(x')_q^{BS}:=[\BS_{[-a,q]}]$. By Theorem \ref{BSthm} and the fact that $e^{[0,-a]}\# -$ induces an isomorphism $H_{-*}(\O'G;\ZZ)\simeq H_{-*}(\O G;\ZZ)$, the set $\{(x')_q^{BS}\}_{q\in\Q}$ is a $\ZZ$-basis of $H_{-*}(\O'G;\ZZ)$\footnote{In fact, Bott-Samelson proved this result first and used it to deduce Theorem \ref{BSthm}.}.
\item \underline{$HW^*(L',L)$:} The $\Ad$-invariant metric on $\gg$ gives rise to a bi-invariant metric on $G$ and hence a quadratic \Ham $H:=\frac{1}{2}|-|^2$ on $T^*G$. The set of \Ham chords of $H$ from $L'$ to $L$ is in natural bijective correspondence with the set of geodesics in $T$ from $e^{-a}$ to $e$. For generic $a$, the latter set is given by
\[ \{\cu_q:=e^{[-a,q]}|~q\in \Q\}.\]
Every \Ham chord is non-degenerate and has even Floer degree so that the Floer differential of $CW^*(L',L)$ is zero. Denote by $(x')_q^F\in HW^*(L',L)$ the cohomology class represented by the chord associated to $\cu_q$. Then the set $\{(x')_q^F\}_{q\in\Q}$ is a $\ZZ$-basis of $HW^*(L',L)$.
\end{itemize}
For the rings $H_{-*}(\O G;\ZZ)$ and $HW^*(L,L)$, we define $x_q^{BS}$ and $x_q^F$ to be the elements which correspond to $(x')_q^{BS}$ and $(x')_q^F$ under the isomorphisms induced by $(x')_0^{BS}$ and $(x')_0^F$ respectively. Notice that $x_q^{BS}$ agrees with $x_q$ defined in the last subsection.

We now state (a special case of) the main theorem of \cite{BCL}. Recall the function $\ell':W\ra \RR:w\mapsto\sum_{a\in R^+}\{\a(w^{-1}(a))\}$ defined in Definition \ref{perturblengthdef}.

\begin{theorem} \label{BCLthm} For any $q\in \Q$, we have
\[ (PSS\circ \MC)(x_q^F)\in \s_{w_q(x_0)} \TT^{A_{w_q^{-1}(q)}}+\bigoplus_{ \substack{w'\in W\\ \ell'(w')<\ell'(w_q)}} \ZZ[\Nov{\OO_{x_0}}{\O G}] \cdot \s_{w'(x_0)}. \] \hfill$\square$
\end{theorem}
\noindent By a simple filtration argument, we have
\begin{corollary} \label{BCLinjective} $PSS\circ \MC$ is injective. \hfill$\square$
\end{corollary}
\noindent Observe that $\ell'$ attains the minimum precisely at $w=\ii$ and $\s_{x_0}=1$. Thus we have
\begin{corollary} \label{nohigher} $(PSS\circ MC)(x_q^F)=\TT^{A_q}$ for any $q\in \Q\cap \wc_0$. \hfill$\square$
\end{corollary}
\noindent In the proof of Theorem \ref{computeG/T} below, we will need the following formula which is analogous to a special case of the multiplication formula \eqref{multiBS}.
\begin{lemma} \label{multiforF} $x_{q_1}^F\bulletsmall x_{q_2}^F= x_{q_1+q_2}^F$ for any $q_1,q_2\in \Q\cap \wc_0$.  \hfill$\square$
\end{lemma}
\begin{proof} It follows from Corollary \ref{BCLinjective} and Corollary \ref{nohigher}.
\end{proof}

The following lemma is necessary for the proof of Theorem \ref{computeG/T} below.
\begin{lemma} \label{assumemetric} The $\Ad$-invariant metric on $\gg$ can be chosen such that the following conditions hold:
\begin{enumerate}[(a)]
\item all highest roots have equal length; and
\item the element $\rho:=\sum_{\a\in R^+}\a$, when regarded as an element of $\t$ via this metric, lies in $\Q_0$.
\end{enumerate}
\end{lemma}
\begin{proof} Notice that any $\Ad$-invariant metric on $\gg$ is equal to the direct sum of some $\Ad$-invariant metrics associated to the simple factors of $\gg$. For each simple factor of $\gg$, we rescale the corresponding metric such that the unique highest root has a fixed length. The resulting metric on $\gg$ will satisfy Condition (a).

For the second condition, we require the squared-length of all highest roots to be an sufficiently divisible integer (12 is OK) such that the squared-length of any other root is an even integer. Recall $\a^{\vee}:=\frac{2\a}{\langle \a,\a\rangle}$ so we have
\[\rho = \sum_{\a\in R^+}\a = \sum_{\a\in R^+}\frac{\langle\a,\a\rangle}{2} \a^{\vee}\in \Q_0. \]
\end{proof}

\begin{theorem}\label{computeG/T} For any $q\in \Q$, we have
\[ (PSS\circ \MC\circ \FF)(x_q^{BS})\in \pm \s_{w_q(x_0)} \TT^{A_{w_q^{-1}(q)}}+\bigoplus_{ \substack{w'\in W\\ \ell'(w')<\ell'(w_q)}} \ZZ[\Nov{\OO_{x_0}}{\O G}] \cdot \s_{w'(x_0)}. \]
\end{theorem}
\noindent Same as the case for $PSS\circ\MC$, we have
\begin{corollary} \label{G/Tinjective} $PSS\circ \MC\circ \FF$ is injective. \hfill$\square$
\end{corollary}
\begin{corollary} \label{G/Tnohigher} $(PSS\circ MC\circ \FF)(x_q^{BS})=\pm \TT^{A_q}$ for any $q\in \Q\cap \wc_0$. \hfill$\square$
\end{corollary}

\begin{myproof}{Theorem}{\ref{computeG/T}}
Recall from Appendix \ref{wrapped} that $\FF'((x')_q^{BS})$ counts \holo half-strips $u:(-\infty,0]\times [0,1]\ra T^*G$ such that $u(s,0)\in L'$, $u(s,1)\in L$, $u$ converges to an output \Ham chord as $s\to -\infty$ and $u|_{\{0\}\times [0,1]}$ projects to an element of $\BS_{[-a,q]}$.

\begin{lemma} \label{energyindex} For any $q_{in}\in \Q$,
\begin{equation} \label{energyindexeq}
\FF'((x')_{q_{in}}^{BS})\in \pm (x')^F_{q_{in}}+\bigoplus \ZZ\langle (x')^F_{q_{out}}\rangle
\end{equation}
where the summation is taken over all $q_{out}\in\Q$ satisfying
\begin{enumerate}
\item $q_{out}\in q_{in}+\Q_0$;
\item the index equality
\begin{equation}\label{indeq}
\langle q_{in}, w_{q_{in}}(\rho) \rangle -\ell(w_{q_{in}}) = \langle q_{out}, w_{q_{out}}(\rho) \rangle -\ell(w_{q_{out}});
\end{equation}
and
\item the energy inequality
\begin{equation}\label{enineq}
|q_{in}+a|^2>|q_{out}+a|^2.
\end{equation}
\end{enumerate}
\end{lemma}
\begin{proof}
If a solution $u$ exists, then the input and output geodesics $\cu_{q_{in}}$ and $\cu_{q_{out}}$ are homotopic rel endpoints. This proves (1), since $\pi_1(G)\simeq \Q/\Q_0$. The index equality \eqref{enineq} simply follows from Lemma \ref{lemmaforq}. To prove (3), consider the energy functional $\mathcal{S}:\O'G\ra\RR$ for the bi-invariant metric on $G$. It is Morse and its critical points are precisely the geodesics $\cu_q,~q\in \Q$. By the energy argument in \cite{AS} and the fact that $\mathcal{S}\circ \BS_{[-a,q_{in}]}\equiv \mathcal{S}(\cu_{q_{in}})$, we obtain the energy inequality \eqref{enineq} with ``$>$'' replaced by ``$\geqslant$'' and the equality holds only for constant solutions. However, constant solutions are not regular for the cycle $\BS_{[-a,q_{in}]}$, and we have to perturb it, as in \cite{BS}, to another cycle $\BS'_{[-a,q_{in}]}$ such that $\mathcal{S}\circ \BS'_{[-a,q_{in}]}<\mathcal{S}(\cu_{q_{in}})$ except at a unique critical point $x$ which is non-degenerate and satisfies $\BS'_{[-a,q_{in}]}(x)=\cu_{q_{in}}$. The existence, uniqueness and regularity of constant solutions for this new cycle now follow from the arguments in \cite{AS}, and this solution gives rise to the leading term $\pm (x')_{q_{in}}^F$.
\end{proof}

By Remark \ref{piimpliesnopi}, we obtain an expression for $\FF$ similar to \eqref{energyindexeq}.

\begin{lemma}\label{Fspeciallemma} The expression
\begin{equation}\label{Fspecial}
\FF(x_{q_{in}}^{BS})\in \pm x_{q_{in}}^F + \bigoplus_{\ell'(w_{q_{out}})<\ell'(w_{q_{in}})} \ZZ\langle x_{q_{out}}^F\rangle
\end{equation}
holds in the following situations:
\begin{enumerate}[(I)]
\item $q_{in}=w_{q_{in}}(\rho)$.
\item $q_{in}$ lies in a finite subset of $\Q\cap \mr{\wc}_0$ specified below.
\end{enumerate}
\end{lemma}
\begin{myproof2}{Situation (I)} Since there are only finitely many $q_{out}$ satisfying \eqref{enineq}, we have $|q_{in}|\geqslant|q_{out}|$ if $|a|$ is small. It follows that
\begin{equation}\label{sit1A}
\langle q_{in},w_{q_{in}}(\rho) \rangle = |q_{in}||\rho|\geqslant |q_{out}||\rho| \geqslant\langle q_{out},w_{q_{out}}(\rho) \rangle
\end{equation}
and the equality holds if and only if $q_{out}=w_{q_{out}}(\rho)$. By \eqref{indeq} and Lemma \ref{perturblengthlemma}, it suffices to show that
\begin{equation}\label{sit1B}
\langle q_{in}+a ,w_{q_{in}}(\rho) \rangle>\langle q_{out}+a,w_{q_{out}}(\rho) \rangle.
\end{equation}
If the inequality \eqref{sit1A} is strict, then \eqref{sit1B} is achieved by assuming $|a|$ to be small. If \eqref{sit1A} is an equality, then \eqref{enineq} gives $\langle a, q_{in}\rangle>\langle a,q_{out}\rangle$, and hence $\langle a, w_{q_{in}}(\rho)\rangle>\langle a,w_{q_{out}}(\rho)\rangle$ which implies \eqref{sit1B}.
\end{myproof2}
\begin{myproof2}{Situation (II)} We begin with a lemma whose proof is postponed until the end.
\begin{lemma} \label{sit2A} For any $q_0\in\Q$, there exist $K_0\in\NN$ and $\e>0$ such that if $K\geqslant K_0$, $|a|<\e$ and $q_{in}:=K\rho +q_0$, then
\begin{enumerate}[(i)]
\item $q_{in}\in \mr{\wc}_0$.
\item every possible $q_{out}$ satisfies $w_{q_{out}}=\ii$ and $|q_{out}-K\rho|\leqslant|q_0|$. \hfill$\square$
\end{enumerate}
\end{lemma}
Define $S$ to be the set of $q\in \Q\setminus\{0\}$ which have minimal norm among all elements of the same coset in $\Q/\Q_0$. Fix a large $K_0\in\NN$ such that if $|a|$ is small, then Conditions (i) and (ii) in Lemma \ref{sit2A} hold for any $q_0\in (S\cup\{0\})+(S\cup\{0\})$ and $K\geqslant K_0$. We show that expression \eqref{Fspecial} holds for $q_{in}=K_0\rho+q_0$ for any $q_0\in S$. Let $q_0\in S$. Suppose $\FF(x_{K_0\rho+q_0}^{BS})$ contains a term other than the leading term. Then, by (i) and (ii) in Lemma \ref{sit2A}, it is of the form $x_{K_0\rho+q_1}^F$ with $K_0\rho+q_1\in \mr{\wc}_0$ and $|q_1|\leqslant |q_0|$. Since $q_1\in q_0+\Q_0$ by Lemma \ref{energyindex}(1), we have either $|q_1|=|q_0|$ or $q_1=0$.
\begin{itemize}
\item \underline{Case (a): $|q_1|=|q_0|$} In this case $q_1\in S$. By \eqref{multiBS} and Lemma \ref{multiforF}, we have
\begin{align*}
x_{2K_0\rho +q_0+q_1}^{BS} &= x_{K_0\rho +q_0}^{BS}\bulletsmall x_{K_0\rho +q_1}^{BS}\\
x_{2K_0\rho+2q_1}^F &= x_{K_0\rho +q_1}^F\bulletsmall x_{K_0\rho +q_1}^F.
\end{align*}
It follows that $\FF(x_{2K_0\rho +q_0+q_1}^{BS} )$ contains the term $x_{2K_0\rho+2q_1}^F $. Notice that this term cannot be cancelled by other terms, due to the inequality in Condition (ii). By the same inequality, we have $|2q_1|\leqslant |q_0+q_1|$. Since $|q_1|=|q_0|$, the last inequality implies $q_1=q_0$, a contradiction.

\item \underline{Case (b): $q_1=0$} By Case (a), we can write
\begin{align*}
\FF(x_{K_0\rho +q_0}^{BS} )&= \pm x_{K_0\rho +q_0}^F+A x_{K_0\rho}^F\\
\FF(x_{K_0\rho -q_0}^{BS} )&= \pm x_{K_0\rho -q_0}^F+B x_{K_0\rho}^F.
\end{align*}
Since we have proved Lemma \ref{Fspeciallemma} for Situation (I), we have
\[ \FF(x_{2K_0\rho}^{BS}) = \FF((x_{\rho}^{BS})^{2K_0})= (\pm x_{\rho}^F)^{2K_0}= x_{2K_0\rho}^F.\]
It follows that
\[ x_{2K_0\rho}^F = \FF(x_{K_0\rho +q_0}^{BS})\FF(x_{K_0\rho -q_0}^{BS} ) = (\pm 1\pm AB) x_{2K_0\rho}^F \pm Ax_{2K_0\rho -q_0}^F \pm Bx_{2K_0\rho +q_0}^F,\]
and hence $A=B=0$. That is, the non-leading term $x_{K_0\rho}^F$ does not exist.
\end{itemize}
This completes the proof of Situation (II) and hence Lemma \ref{Fspeciallemma}.
\end{myproof2}

To conclude the proof of Theorem \ref{computeG/T}, we need the following lemma which will be proved at the end.
\begin{lemma} \label{SgenQ} The subset $S$ defined in Situation (II) of Lemma \ref{Fspeciallemma} generates the lattice $\Q$. \hfill$\square$
\end{lemma}

Let $q\in\Q$ be any element. There exists $q_1\in\Q\cap \wc_0$ such that $q+w_q(q_1)\in \NN w_q(\rho)$. By Lemma \ref{SgenQ}, we can write $q_1=\sum_{i=1}^rm_is_i$ for some $s_i\in S$ and $m_i\in\ZZ$. Let $K_0$ be the integer in the proof of Lemma \ref{Fspeciallemma}, Situation (II), and put $N:=K_0\sum_{i=1}^r|m_i|$. Then $q_1+N\rho=\sum_{i=1}^r|m_i|(K_0\rho\pm s_i)$ and we also have $q+w_q(q_1+N\rho)\in \NN w_q(\rho)$. Write the last element as $(M+1)w_q(\rho)$ where $M\in\ZZ_{\geqslant 0}$. We have
\begin{align*}
 & (PSS\circ \MC\circ\FF)(x_{q+w_q(q_1+N\rho)}^{BS})\\
 =~& (PSS\circ \MC\circ\FF)((x_{\rho}^{BS})^M\bulletsmall x_{w_q(\rho)}^{BS})\\
 \in~& (\pm \TT^{A_{\rho}})^M(PSS\circ \MC)\left(\pm x_{w_q(\rho)}^F+\bigoplus_{\ell'(w_{q'})<\ell'(w_q)}\ZZ\langle x_{q'}^F\rangle \right)\\
 =~&(\pm \TT^{A_{\rho}})^M \left( \pm \s_{w_q(x_0)}\TT^{A_{\rho}}+\bigoplus_{\ell'(w'')<\ell'(w_q)}\ZZ[\Nov{\OO_{x_0}}{\O G}]\cdot \s_{w''(x_0)}  \right)\\
 =~& \pm \s_{w_q(x_0)} \TT^{A_{(M+1)\rho}} + \bigoplus_{\ell'(w'')<\ell'(w_q)}\ZZ[\Nov{\OO_{x_0}}{\O G}]\cdot \s_{w''(x_0)}
\end{align*}
where the first equality follows from the multiplication formula \eqref{multiBS}, the set membership from Corollary \ref{nohigher} and Situation (I) of Lemma \ref{Fspeciallemma}, and the second last equality from Theorem \ref{BCLthm}.

On the other hand,
\begin{align*}
 & (PSS\circ \MC\circ\FF)(x_{q+w_q(q_1+N\rho)}^{BS})\\
 =~& (PSS\circ \MC\circ\FF)(x_{q_1+N\rho}^{BS}\bulletsmall x_{q}^{BS})\\
 =~ &(PSS\circ \MC)\left( \prod_{i=1}^r (\pm x_{K_0\rho \pm s_i}^F)^{|m_i|}\right)\bulletsmall (PSS\circ \MC\circ\FF)(x_q^{BS})\\
 =~& \left[\prod_{i=1}^r \left(\pm \TT^{A_{K_0\rho\pm s_i}} \right)^{|m_i|}\right] \bulletsmall (PSS\circ \MC\circ\FF)(x_q^{BS}).
\end{align*}
where the first equality follows from the multiplication formula \eqref{multiBS}, the second from Situation (II) of Lemma \ref{Fspeciallemma} and the multiplication formula \eqref{multiBS}, and the last from Corollary \ref{nohigher}.
Therefore,
\begin{align*}
& ~(PSS\circ \MC\circ\FF)(x_q^{BS}) \\
\in & \left[\prod_{i=1}^r \left(\pm \TT^{A_{K_0\rho\pm s_i}}\right)^{-|m_i|}\right] \cdot \left(\pm \s_{w_q(x_0)} \TT^{A_{(M+1)\rho}} + \bigoplus_{\ell'(w'')<\ell'(w_q)}\ZZ[\Nov{\OO_{x_0}}{\O G}]\cdot \s_{w''(x_0)} \right)\\
= &\pm \s_{w_q(x_0)} \TT^{A_q} +  \bigoplus_{\ell'(w'')<\ell'(w_q)}\ZZ[\Nov{\OO_{x_0}}{\O G}]\cdot \s_{w''(x_0)}.
\end{align*}
This completes the proof of Theorem \ref{computeG/T}.
\end{myproof}

\begin{myproof}{Lemma}{\ref{sit2A}} The existence of $K_0$ for which (i) holds is obvious. For (ii), notice that $w_{q_{in}}=\ii$, and hence, by \eqref{indeq},
\[\langle q_{out},w_{q_{out}}(\rho)\rangle \geqslant\langle q_{in},w_{q_{in}}(\rho)\rangle\]
and the equality holds if and only if $w_{q_{out}}=\ii$. It follows that
\begin{align*}
0 & \leqslant \langle q_{out},w_{q_{out}}(\rho)\rangle -\langle q_{in},w_{q_{in}}(\rho)\rangle\\
& \leqslant |\rho| |q_{out}+a| +|\rho||a| -\langle K\rho +q_0,\rho\rangle\\
&< |\rho| (|K\rho+q_0+a|-K|\rho|) +|\rho||a|-\langle q_0,\rho\rangle.
\end{align*}
(The second inequality follows from the Cauchy-Schwarz inequality and the third from the energy inequality \eqref{enineq}.) Observe that $\lim_{K\to +\infty}\left( |K\rho+q_0+a|-K|\rho| \right) =\frac{\langle q_0+a,\rho\rangle}{|\rho|}$. It follows that for sufficiently large $K$ and sufficiently small $|a|$,
\[ 0\leqslant \langle q_{out},w_{q_{out}}(\rho)\rangle -\langle q_{in},w_{q_{in}}(\rho)\rangle<1. \]
Since the middle term is an integer, it must be zero, and hence $w_{q_{out}}=\ii$. This shows the first part of (ii). For the second part, consider
\begin{align*}
|q_{out}+a-K\rho|^2 & = |q_{out}+a|^2 -2K\langle q_{out}+a,\rho\rangle +K^2|\rho|^2\\
&< |q_{in}+a|^2 -2K\langle q_{in}+a,\rho\rangle +K^2|\rho|^2\\
&= |q_{in}+a-K\rho|^2\\
&= |q_0+a|^2.
\end{align*}
As in the beginning of the proof of Situation (I) of Lemma \ref{Fspeciallemma}, we have $|q_{out}-K\rho|\leqslant |q_0|$ if $|a|$ is smaller than a positive real number depending only on $q_0$. This gives the second part and hence completes the proof.
\end{myproof}

\begin{myproof}{Lemma}{\ref{SgenQ}} Clearly $S$ intersects every coset in $\Q/\Q_0$. Thus it suffices to show that $S_0:=S\cap \Q_0$ generates the lattice $\Q_0$. Since the Weyl group $W$ preserves $S_0$, we have $S_0\cap \wc_0\ne \emptyset$. Let $q$ be an element of this subset. We claim that $q=\a^{\vee}$ for some highest root $\a$. Indeed, suppose the contrary, we have $|q|\leqslant |q-\a^{\vee}|$ for any highest root $\a$. Then $\a(q)\leqslant 1$, since $\a^{\vee}$ is the reflection of $0$ across the affine wall $V_{\a,1}=\{\a=1\}$. It follows that $q\in \D_0\cap \Q_0$. But this set is equal to $\{0\}$, by Lemma \ref{stdfacts}(1), a contradiction. Conversely, every $\a^{\vee}$ ($\a$ highest) lies in $S_0$, by Lemma \ref{assumemetric}. Now, the sublattice $\Span_{\ZZ}(S_0)\subseteq \Q_0$ is invariant under $W$ and translations $t_{\a^{\vee}}$ for any highest root $\a$. Since the affine Weyl group $W^{af}$ is generated by these group elements by Lemma \ref{stdfacts}(2), we conclude that $\Span_{\ZZ}(S_0)$ contains the orbit of $0$ under $W^{af}$ which is equal to $\Q_0$ by Lemma \ref{stdfacts}(3).
\end{myproof}
\subsection{From G/T to G/L}\label{fromG/T}
Theorem \ref{computeG/Lintro} is proved by applying Theorem \ref{GPmain} to the \Ham fibration $\pi:G/T\ra G/L$. In order to make it Floer-theoretically possible, these symplectic manifolds have to be monotone. We have chosen $x_0\in\RR_{>0}\rho$ in the last subsection such that $G/T:=\OO_{x_0}$ is monotone. For $G/L$, we denote by $\a_1,\ldots,\a_r\in R^+$ the positive roots defining $\wc_0$, i.e. they are pairwise distinct and $\wc_0=\bigcap_{i=1}^k\{\a_i\geqslant 0\}$. These roots are called \textit{simple roots}. For any subset $I\subseteq \{1,\ldots,r\}$, define $\t_I:=\bigcap_{i\in I}\{\a_i=0\}$, $R_I:=\{\a\in R|~\a|_{\t_I}\equiv 0\}$, $R_I^+:=R_I\cap R^+$ and $\rho_I:=\sum_{\a\in R^+\setminus R_I^+}\a$ which we regard as an element of $\t$ via the $\Ad$-invariant metric. Notice that $\rho_I$ lies in the interior of $\t_I\cap \wc_0$ in $\t_I$. Then any coadjoint orbit must correspond to a unique subset $I$. Now fix a subset $I$, choose any $y_0\in \RR_{>0}\rho_I$ and define $G/L:=\OO_{y_0}$. It is well-known that the KSS form on $G/L$ is monotone.

Define $\pi:G/T\ra G/L$ to be the unique $G$-equivariant smooth map sending $x_0$ to $y_0$. Recall there are $G$-invariant integrable almost complex structures $I_{G/T}$ and $I_{G/L}$ on $G/T$ and $G/L$ which are compatible with the KSS forms.
\begin{lemma}\label{GPKleiman} Put $X:=G/T$ and $Y:=G/L$. For any smooth cycle $\vp:\G\ra \O G$, define $I^{\g}_X\equiv I_{G/T}$ and $I^{\g}_Y\equiv I_{G/L}$. Then Assumption \ref{GPassume} holds for $I_X=\{I^{\g}_X\}_{\g\in\G}$ and $I_Y=\{I^{\g}_Y\}_{\g\in\G}$, and for any smooth cycles $h_X$ and $h_Y$.
\end{lemma}
\begin{proof} (1) is proved by recalling the definition of $I_{G/T}$ and $I_{G/L}$. See \eqref{CScoadj}. For (2) and (3), it suffices to prove the result for $G/L$. The case for $G/T$ follows from the case for $G/L$ by putting $I=\emptyset$. Since $(G/L,I_{G/L})$ is a convex manifold, the moduli space $\M^{simple}(A,\{pt\},I_{G/L})$ is regular for any $A\in \pi_2(G/L)$. These moduli spaces admit a natural $G$-action such that the evaluation maps at any point are $G$-equivariant. Since $G$ acts on $(G/L,I_{G/L})$ holomorphically and transitively (it is also the reason for the convexity of $G/L$), these evaluation maps are submersive. The result follows from the induction on the length of the chains of holomorphic spheres.
\end{proof}

Although Assumption \ref{GPassume} holds for any smooth cycles in $G/T$ and $G/L$, by Lemma \ref{GPKleiman}, we have to construct particular ones which are compatible with the \Ham fibration $\pi:G/T\ra G/L$. Recall we have used the Morse function $\langle a,-\rangle$ to construct pseudocycles $\UU_x^{G/T}$ and $\UU_y^{G/L}$ for any $x\in\crit_{G/T}$ and $y\in\crit_{G/L}$. For any $y\in\crit_{G/L}$, define $F_y:=\pi^{-1}(y)$. Then the orthogonal projection $\pi_{\t^{\perp}_y}:\t\ra\t_y^{\perp}$ induces a $G_y$-equivariant symplectomorphism from $F_y$ onto a coadjoint orbit of $G_y^{\ad}$ of maximum dimension, i.e. diffeomorphic to $G_y^{\ad}/\exp(t_y^{\perp})$. The function $\langle a,-\rangle$ restricts to a Morse function on $F_y$ with critical set
\[ \crit_{F_y}=\{x\in\crit_{G/T}|~\pi(x)=y\}.\]
We define $\UU_x^{F_y}$, $x\in\crit_{F_y}$ in a similar manner.

\begin{lemma} For any $y\in \Crit_{G/L}$ and $x\in \Crit_{F_y}$, the restriction $\pi|_{\UU^{G/T}_x}:\UU^{G/T}_x\ra \UU^{G/L}_y$ forms a fiber bundle whose fiber over $y$ is equal to $\UU^{F_y}_x$.
\end{lemma}
\begin{proof}
Notice that for any of $F_y$, $G/T$ and $G/L$, the gradient flow is equal to the multiplication by $e^{-\sqrt{-1}t a}\in G_{\CC}$ . For any $\a\in R$, define
\[\gg_{\a}^{\CC}:=\{ \eta\in \gg\otimes\CC|~\ad(z)\eta =2\pi\sqrt{-1}\a(z)\eta\text{ for any }z\in\t\}.\]
It is well-known that $\UU^{F_y}_x$, $\UU^{G/T}_x$ and $\UU^{G/L}_y$ are the orbits passing through $x$, $x$ and $y$ under the actions of the unipotent subgroups with Lie algebras
\[ \bigoplus_{\a(y)=0,~\a(x)<0}\gg_{\a}^{\CC}, \bigoplus_{\a(x)<0}\gg_{\a}^{\CC}\text{ and }\bigoplus_{\a(x)<0}\gg_{\a}^{\CC}\]
respectively. The stabilizers are the exponentials of
\[ \bigoplus_{\a(y)=0,~\a(x)<0,~\a(a)>0}\gg_{\a}^{\CC}, \bigoplus_{\a(x)<0,~\a(a)>0}\gg_{\a}^{\CC}\text{ and }\bigoplus_{\a(x)<0,~(\a(y)=0\text{ or }\a(a)>0)}\gg_{\a}^{\CC}\]
respectively. The result follows from the canonical isomorphism
\[\bigoplus_{\a(x)<0,~\a(y)=0\text{ or }\a(a)>0}\gg_{\a}^{\CC}\left/\bigoplus_{\a(x)<0,~\a(a)>0}\gg_{\a}^{\CC} \right.\simeq\bigoplus_{\a(x)<0,~\a(y)=0}\gg_{\a}^{\CC}\left/ \bigoplus_{\a(x)<0,~\a(y)=0,~\a(a)>0}\gg_{\a}^{\CC}\right. . \]
\end{proof}

\begin{remark} In Remark \ref{BSresolve}, we pointed out that the closure of every unstable submanifold in a coadjoint orbit admits a Bott-Samelson resolution. It is possible to construct resolutions which are compatible with the fiber bundle $\pi|_{\UU^{G/T}_x}:\UU^{G/T}_x\ra \UU^{G/L}_y$ in the sense that the resulting smooth cycles form a fiber bundle as well. To avoid introducing unnecessary notations, the smooth cycles $h_{F_y}$, $h_{G/T}$ and $h_{G/L}$ will be replaced by the pseudocycles $\UU^{F_y}_x$, $\UU^{G/T}_x$ and $\UU^{G/L}_y$ in the remaining discussion.
\end{remark}

Recall the notions $\psi_A$ and $\PP_A$ defined in Section \ref{Leray} and that we have determined $\Nov{X}{\O G}$ for any coadjoint orbit $X$ (Lemma \ref{GPNovforOO}).

\begin{proposition} \label{peterG/L} Let $x\in\crit_{G/T}$. Put $y:=\pi(x)$. For any $q+\Q_{R_y}\in \Q/\Q_{R_y}\simeq\Nov{G/L}{\O G}$,
\[\psi_{A^{G/L}_{y,q+\Q_{R_y}}} = \pi_{\t_y^{\perp}}(q) +\Q_{R_y}\in P^{\vee}_{R_y}/\Q_{R_y}\simeq \pi_1(G_y^{\ad})\simeq \pi_0(\O G_y^{\ad})\]
and for any $\tilde{q}\in P^{\vee}_{R_y}\simeq \Nov{F_y}{\O G_y^{\ad}}$ with $\tilde{q}\in \pi_{\t_y^{\perp}}(q) +\Q_{R_y}$,
\[\PP_{A^{G/L}_{y,q+\Q_{R_y}}}(A^{F_y}_{x,\tilde{q}}) = A^{G/T}_{x,\tilde{q}-\pi_{\t_y^{\perp}}(q)+q}.  \]
\end{proposition}

\begin{proof} In general, given $A\in \Nov{G/L}{\O G}$ represented by $(\vp,s_{\vp,u_{\pm}})$ for some $\vp\in \O G$ and maps $u_{\pm}:D_{\pm}\ra G/L$ satisfying
\[ u_+(\eit)=\vp(\eit)\cdot u_-(e^{i\tt}),\]
we have to trivialize the \Ham fibrations $D_{\pm}\times_{(u_{\pm},\pi)}G/T$ over $D_{\pm}$ and see how the fibers over $\pt D_-$ is identified with those over $\pt D_+$. In our case, every $A$ is of the form $A^{G/L}_{y,q+\Q_{R_y}}$, corresponding to $\vp(\eit)=\exp(\tt\cdot q/2\pi)$ and $u_{\pm}\equiv y$. Since $u_{\pm}$ are constant, there are canonical trivializations of the aforementioned \Ham fibrations. Therefore, the fibers over $\pt D_{\pm}$ are identified by $\vp$ acting on $F_y$, and hence for any $q+\Q_{R_y}\in\Q/\Q_{R_y}$,
\[ \psi_{A^{G/L}_{y,q+\Q_{R_y}}} = \pi_{\t_y^{\perp}}(q) +\Q_{R_y}\in P^{\vee}_{R_y}/\Q_{R_y}. \]
Since $F_y$ is a coadjoint orbit of $G_y^{\ad}$, we have, by Lemma \ref{GPNovforOO},
\[ \pi_2^{section}(P_{\psi_{A^{G/L}_{y,q+\Q_{R_y}}}}(F_y))\simeq \{ A^{F_y}_{x,\tilde{q}}|~ \tilde{q}\in \pi_{\t_y^{\perp}}(q)+\Q_{R_y}\}.\]
Given $\tilde{q}\in \pi_{\t_y^{\perp}}(q)+\Q_{R_y}$. $\PP_{A^{G/L}_{y,q+\Q_{R_y}}}(A^{F_y}_{x,\tilde{q}})$ is equal to $A^{G/T}_{x,p}$ for some $p\in \Q$. The element $p$ has to be projected to $q\in \Q/\Q_{R_y}$ under the quotient map $\Q\ra \Q/\Q_{R_y}$ and to $\tilde{q}\in P^{\vee}_{R_y}$ under $\pi_{\t_y^{\perp}}$. It is straightforward to see that $p$ is equal to $\tilde{q}-\pi_{\t_y^{\perp}}(q)+q$, as desired.
\end{proof}

\begin{definition}\label{petersonlift} Let $y\in \crit_{G/L}$ and $q+\Q_{R_y}\in\Q/\Q_{R_y}$.
\begin{enumerate}
\item Define the \textit{Peterson lift} of $q+\Q_{R_y}$ to be the unique element $\tilde{q}$ of the coset $q+\Q_{R_y}\subset \Q$ such that the dimension of the Bott-Samelson cycle in $\O G_y^{\ad}$ associated to $\pi_{\t_y^{\perp}}(\tilde{q})\in P^{\vee}_{R_y}$ is zero.
\item Define the \textit{associated lift} of $y$ with respect to $q+\Q_{R_y}$ to be the unique element $x\in \crit_{F_y}$ such that $\a(x)$ and $\a(\tilde{q}+a)$ have the same sign for any $\a\in R_y$, where $\tilde{q}$ is the Peterson lift of $q+\Q_{R_y}$ and $a\in\mr{\wc}_0$ is an element sufficiently close to the origin.
\end{enumerate}
\end{definition}

\begin{proposition} \label{compare} Given $y\in \crit_{G/L}$ and $q+\Q_{R_y}\in\Q/\Q_{R_y}$. Suppose $\vp:\G\ra\O G$ is a smooth cycle such that
\begin{equation}\label{comparedim}
\dim\G+\dim \UU_y^{G/L}+2c_1^v\left(A^{G/L}_{y,q+\Q_{R_y}}\right) =0.
\end{equation}
We have
\begin{equation}\label{pfG/L3}
\langle \See{G/T}{\O G}([\vp]), [\UU_x^{G/T}] \rangle_{A^{G/T}_{x,\tilde{q}}} = \pm \langle \See{G/L}{\O G}([\vp]), [\UU_y^{G/L}]\rangle_{A^{G/L}_{y,q+\Q_{R_y}}}
\end{equation}
where $\tilde{q}$ is the Peterson lift of $q+\Q_{R_y}$ and $x$ is the associated lift of $y$ with respect to $q+\Q_{R_y}$.
\end{proposition}
\begin{proof} Notice that $\pi(x)=y$. Put $q':=\pi_{\t_y^{\perp}}(\tilde{q})$. By Proposition \ref{peterG/L},
\begin{equation}\label{pfG/L2}
\PP_{A^{G/L}_{y,q+\Q_{R_y}}}\left(A^{F_y}_{x,q'}\right) = A^{G/T}_{x,\tilde{q}}.
\end{equation}
By the same proposition, $\PP_{A^{G/L}_{y,q+\Q_{R_y}}}$ is injective, and hence $A^{F_y}_{x,q'}$ is the unique element for which \eqref{pfG/L2} holds. Since $F_y$ is a coadjoint orbit of $G_y^{\ad}$ passing through a regular point, we have, by Theorem \ref{computeG/T},
\[\langle \See{F_y}{\O G_y^{\ad}}(x_{q'}^{\O G_y^{\ad}}), [\UU_x^{F_y}]\rangle_{A^{F_y}_{x,q'}}=\pm 1\]
where $x_{q'}^{\O G_y^{\ad}}\in H_0(\O G_y^{\ad};\ZZ)$ is represented by the Bott-Samelson cycle associated to $q'\in P^{\vee}_{R_y}$. The result follows from Theorem \ref{GPmain} applied to the fiber bundle $\UU_x^{F_y}\hookrightarrow \UU_x^{G/T}\ra \UU_y^{G/L}$.
\end{proof}

For any $q\in \Q$, define
\[\deg^{L/T}(q):=\sum_{\substack{\a\in R_{w_q(y_0)}\\ \a(w_q(x_0))>0}} \lfloor\a(q+a)\rfloor\]
where $a\in\mr{\wc}_0$ is an element sufficiently close to the origin. It can be checked that $\deg^{L/T}(q)$ is the complex dimension of the Bott-Samelson cycle for $\O G_{w_q(y_0)}^{\ad}$ associated to $\pi_{\t^{\perp}_{w_q(y_0)}}(q)\in P^{\vee}_{R_{w_q(y_0)}}$.
\begin{theorem} \label{computeG/L} Let $q\in \Q$.
\begin{enumerate}[(a)]
\item If $\deg^{L/T}(q)=0$, then
\[ \See{G/L}{\O G}(x_q) = \pm \s_{w_q(y_0)} \TT^{A^{G/L}_{w_q^{-1}(q)+\Q_{R_{y_0}}}}+\cdots \]
where $\cdots$ is a finite sum of terms which do not cancel with the first term.
\item If $\deg^{L/T}(q)=0$ and $w_q=\ii$, then
\[ \See{G/L}{\O G}(x_q) = \pm  \TT^{A^{G/L}_{q+\Q_{R_{y_0}}}}.\]
\end{enumerate}
\end{theorem}

\begin{proof}
\begin{enumerate}[(a)]
\item Put $y:=w_q(y_0)$. Then $q$ is the Peterson lift of $q+\Q_{R_y}$ and $x:=w_q(x_0)$ is the associated lift of $y$ with respect to $q+\Q_{R_y}$. Moreover, \eqref{comparedim} holds if we take $\vp$ to be the Bott-Samelson cycle in $\O G$ associated to $q$. Indeed, one can show that
\begin{align*}
\dim\G &=2\deg(q) = 2\sum_{\a(x)>0}\lfloor \a(q+a)\rfloor\\
\dim \UU_y^{G/L} &= -2\sum_{\a(y)>0}\lfloor \a(a)\rfloor\\
2c_1^v\left(A^{G/L}_{y,q+\Q_{R_y}}\right) &= -2\sum_{\a(y)>0}\a(q).
\end{align*}
It follows that
\[
\dim\G+\dim \UU_y^{G/L}+2c_1^v\left(A^{G/L}_{y,q+\Q_{R_y}}\right) = 2\sum_{\substack{\a(x)>0\\ \a(y)=0}}\lfloor \a(q+a)\rfloor = 2\deg^{L/T}(q)= 0,
\]
by assumption. Therefore, the result follows from Proposition \ref{compare} and Theorem \ref{computeG/T}.
\item By Corollary \ref{G/Tnohigher}, the LHS of \eqref{pfG/L3} vanishes unless $x=x_0$ and $\tilde{q}=q$. It follows that $\See{G/L}{\O G}(x_q)$ contains only the term from part (a).
\end{enumerate}
\end{proof}
\appendix
\section{Wrapped Floer theory on cotangent bundles}\label{wrapped}
Let $\P$ be a compact connected smooth manifold. Its cotangent bundle $T^*\P$ is a Liouville manifold with the standard Liouville form. For any $p\in\P$, denote by $L_p$ the cotangent fiber $T_p^*\P$. We may drop the subscript of $L_p$ if the point $p$ is not specified. By the construction in \cite{Ab_JSG}, we have two cohomology groups $HW_b^*(L,L)$ and $SH_b^*(T^*\P)$, where $b\in H^2(T^*\P;\ZZ_2)$ is the background class given by the pull-back of $w_2(\P)$ via the canonical projection $\pi:T^*\P\ra \P$.

Consider the based loop space $\O\P$ and free loop space $\Lo\P$ of $\P$. In \cite{AS}, Abbondandolo-Schwarz constructed two isomorphisms
\begin{align*}
\FF &: H_{-*}(\O\P)\ra HW_b^*(L,L)\\
\GG &: H_{-*}(\Lo\P)\ra SH_b^*(T^*\P).
\end{align*}
\noindent For our purposes, we recall how these maps are defined on homology classes represented by smooth cycles. Let $f:\G\ra \O\P$ and $g:\G'\ra \Lo\P$ be two smooth cycles. For any $\g\in\G$ (resp. $\g\in\G'$), we denote by $f_{\g}$ (resp. $g_{\g}$) the corresponding loops in $\P$. Denote by $[f]$ and $[g]$ the corresponding homology classes. Let $H$ and $H'$ be the Hamiltonians which define $HW_b^*(L,L)$ and $SH_b^*(T^*\P)$ respectively\footnote{Recall that $H$ is a Hamiltonian quadratic at infinity and $H'=H+F$ is a perturbation of $H$ by a uniformly bounded time-dependent function $F:S^1\times T^*\P\ra \RR$.}. Denote by $\X$ the set of \Ham chords of $H$ from $L$ to $L$, and by $\Y$ the set of \Ham orbits of $H'$.

\begin{definition}\label{ASdef} $~$
\begin{enumerate}
\item Define
\begin{equation}\nonumber
\FF([f]):=\sum_{x\in\X}\#\M_{AS}^0(x,f) x
\end{equation}
where $\M_{AS}^0(x,f)$ is the zero-dimensional component of the moduli space of pairs $(\g,u)$ with $\g\in\G$ and
\begin{equation}\nonumber
u:(-\infty,0]\times[0,1]\ra T^*\P
\end{equation}
satisfying the boundary and asymptotic conditions
\begin{equation}\nonumber
\left\{ \begin{array}{l}
u(s,0),u(s,1)\in L\text{ for any }s\leqslant 0;\\
f_{\g}=(\pi\circ u)(0,\cdot);\\
\ds\lim_{s\ra -\infty} u(s,t)=x(t)\text{ for any }t;
\end{array}\right.
\end{equation}
solving the perturbed Cauchy-Riemann equation
\begin{equation}\nonumber
(du-X_H\otimes dt)^{0,1}=0
\end{equation}
with respect to a domain-dependent compatible almost complex structure which is of contact type, and having finite energy:
\[ \int |du-X_H\otimes dt|^2 <+\infty.\]

\item Define
\begin{equation}\nonumber
\GG([g]):=\sum_{y\in\Y}\#\M_{AS}^0(y,g) y
\end{equation}
where  $\M_{AS}^0(y,g)$ is the zero-dimensional component of the moduli space of pairs $(\g,u)$ with $\g\in\G'$ and
\begin{equation}\nonumber
u:(-\infty,0]\times S^1\ra T^*\P
\end{equation}
satisfying
\begin{equation}\nonumber
\left\{ \begin{array}{l}
g_{\g}=(\pi\circ u)(0,\cdot);\\
\ds\lim_{s\ra -\infty} u(s,e^{i \tt})=y(e^{i \tt})\text{ for any }\tt;
\end{array}\right.
\end{equation}
solving the perturbed Cauchy-Riemann equation
\begin{equation}\nonumber
(du-X_{H'}\otimes d\tt)^{0,1}=0,
\end{equation}
and having finite energy:
\[ \int |du-X_{H'}\otimes d\tt|^2<+\infty.  \]
\end{enumerate}
\end{definition}

\begin{remark} Abbondandolo-Schwarz constructed $\FF$ and $\GG$ using Morse models on $\O \P$ and $\Lo \P$ respectively. To see that their maps are equal to what we have just defined, consider the moduli space of triples $(\g,u,T)$ with $T\in [0,+\infty)$ and $(\g,u)$ satisfying the conditions in Definition \ref{ASdef} except that $u(0,\cdot)$ is now projected to the time-$T$ gradient flow of $f_{\g}$ or $g_{\g}$. Its one-dimensional component gives the desired equality, modulo an exact element of $CW_b^*(L,L)$ or $SC_b^*(T^*\P)$.
\end{remark}

\begin{proposition} \label{commwrapped} The following diagram is commutative.
\begin{equation}\nonumber
\begin{tikzcd}
H_{-*}(\O \P)  \arrow{r}{\FF} \arrow{d}[left]{\iota} &[2em]   HW_b^*(L,L)   \arrow{d}{\mathcal{OC}} \\
 H_{-*}(\Lo \P) \arrow{r}{\GG} &[2em] SH_b^*(T^*\P)
\end{tikzcd}
\end{equation}
where $\iota$ is the map induced by the canonical inclusion $\O \P\hookrightarrow \Lo\P$ and $\mathcal{OC}$ is the length-zero part of the open-closed map.
\end{proposition}
\begin{proof}
It follows from a combination of \cite[Proposition 1.6]{Ab_Adv}, \cite[Lemma 5.1]{Ab_JSG} and its analogue for $\Lo\P$, and the fact that the Goodville's isomorphism in \cite{Ab_Adv} restricts to $\iota$ on $H_{-*}(\O \P)$.
\end{proof}


\section{Gluing: annulus-shrinking}\label{gluing}
This appendix recalls a gluing result of Wehrheim-Woodward \cite{WW2} about \holo quilted maps with an annulus patch which is shrunk to become a loop at the limit. For our application, we need a slightly more general version which requires the quilted maps to satisfy parametrized seam conditions. The new ingredient is a more general energy-index relation which we will prove to be satisfied in our situation.

Let $M_0, M_1, M_2$ be monotone symplectic manifolds with the same monotonicity constant $\monoconst >0$. Suppose $\G$ is a compact smooth manifold and there are smooth families $\{L_{01}^{\g,\tt}\}_{(\g,\eit)\in \G\times S^1}$, $\{L_{12}^{\g,\tt}\}_{(\g,\eit)\in \G\times S^1}$ of Lagrangian correspondences
\[ L_{01}^{\g,\tt}: M_0\ra M_1\quad\text{ and }\quad L_{12}^{\g,\tt}:M_1\ra M_2 \]
such that the geometric composition $L_{02}^{\g,\tt}:=L_{01}^{\g,\tt}\circ L_{12}^{\g,\tt}$ is embedded for any $\g,\tt$ and $L_{01}^{\g,\tt}$, $L_{12}^{\g,\tt}$, $L_{02}^{\g,\tt}$ are monotone.

While the theorem we are going to recall holds for general quilted surfaces, we will only restrict ourselves to $\ul{\S}:=(\S_0,\S_2)$ where $\S_0$ is a unit disk and $\S_2$ is a unit disk with an interior marked point. The unique seam circle $C_{02}$ of $\ul{\S}$ is formed by identifying the boundary curves $\pt\S_0=S^1$ and $\pt\S_2=S^1$ via $\eit\leftrightarrow e^{-i\tt}$. For any $\cl>0$, defined $\ul{\S}^{\cl}$ to be the quilted surface consisting of three patches $\S_0$, $\S_1^{\cl}:=S^1\times [-\cl,0]$ and $\S_2$ with two seams $C_{01}^{\cl}$ and $C_{12}^{\cl}$, where $C_{01}^{\cl}$ (resp. $C_{12}^{\cl}$) is formed by identifying $\pt\S_0$ with $S^1\times \{0\}\subset \pt\S_1^{\cl}$ (resp. $S^1\times \{-\cl\}\subset \pt\S_1^{\cl}$ with $\pt\S_2$) via $\eit\leftrightarrow (e^{-i\tt},0)$ (resp. $(\eit,-\cl)\leftrightarrow\eit$). See Figure 7.
\begin{center}
\vspace{0.2cm}
\begin{tikzpicture}
\tikzmath{\x1 = 0; \x2 = 0.2; \x3=1; \x4=1; \x5=3; \x6=1;\x9=0.45; \y1=0.45; \y3=1.2*\x3;}
\tikzmath{\z1= 0.15*\x3 ; \z2=0.2; \z3=1.3*\x3; \z4=0;}

\draw [<->, line width=\z2mm] (\x1,-\z3) -- (\x1+\x5,-\z3);
\draw [line width=\z2mm] (\x1,-\z3+\z1) -- (\x1,-\z3-\z1);
\draw [line width=\z2mm] (\x1+\x5,-\z3+\z1) -- (\x1+\x5,-\z3-\z1);

\begin{scope}
\clip(\x1,-\y3)--(\x1,\y3)--(\x1+1.1*\x2,\y3)--(\x1+1.1*\x2,-\y3)--cycle;
\draw [line width=\y1mm] (\x1,0) ellipse (\x2 and \x3);
\end{scope}

\begin{scope}
\clip(\x1,-\y3)--(\x1,\y3)--(\x1-1.1*\x2,\y3)--(\x1-1.1*\x2,-\y3)--cycle;
\draw [dashed, line width=\y1mm] (\x1,0) ellipse (\x2 and \x3);
\end{scope}


\begin{scope}
\clip(\x1+\x5,-\y3)--(\x1+\x5,\y3)--(\x1+1.1*\x2+\x5,\y3)--(\x1+1.1*\x2+\x5,-\y3)--cycle;
\draw [line width=\y1mm] (\x1+\x5,0) ellipse (\x2 and \x3);
\end{scope}

\begin{scope}
\clip(\x1+\x5,-\y3)--(\x1+\x5,\y3)--(\x1-1.1*\x2+\x5,\y3)--(\x1-1.1*\x2+\x5,-\y3)--cycle;
\draw[dashed, line width=\y1mm] (\x1+\x5,0) ellipse (\x2 and \x3);
\end{scope}

\draw [line width=\y1mm] (-\x6,\x3) arc (90:270:\x3);
\draw [line width=\y1mm] (\x5+\x6,\x3) arc (90:-90:\x3);
\draw [line width=\x9mm] (0,-\x3) -- (-\x6,-\x3);
\draw [line width=\x9mm] (0,\x3) -- (-\x6,\x3);
\draw [line width=\x9mm] (\x5,-\x3) -- (\x5+\x6,-\x3);
\draw [line width=\x9mm] (\x5,\x3) -- (\x5+\x6,\x3);
\draw [line width=\x9mm] (0,-\x3) -- (\x5,-\x3);
\draw [line width=\x9mm] (0,\x3) -- (\x5,\x3);

\node at  (-\x6-\x3,0) [circle,fill,inner sep=2pt]{};


\node[anchor = west] at (\x1+0.4*\x5,0) {$\S_1^{\cl}$};
\node[anchor = west] at (-1.3*\x6,0) {$\S_2$};
\node[anchor = east] at (\x5+1.5*\x6,0) {$\S_0$};
\node[anchor = west] at (\x1+1*\x2,0.6*\x3) {$C_{12}^{\cl}$};
\node[anchor = west] at (\x1+1*\x2+\x5,0.6*\x3) {$C_{01}^{\cl}$};
\node[anchor = north] at (\x1+\x5/2,-\z3) {$\cl$};
\node at (\x5+\x3+\x6+0.2,0) {};
\node at (-\x3-\x6-0.2,0) {};


\end{tikzpicture}

\vspace{0.1cm}
\noindent \hypertarget{fig7}{FIGURE 7.}
\vspace{.2cm}
\end{center}

\begin{definition} \label{moduliforgluing} Let $h:N\ra M_2$ be a pseudocycle.
\begin{enumerate}
\item Define $\M_0(h)$ to be the moduli space of pairs $(\g,\ul{u})$ where
\[ \left\{\begin{array}{l}
\g\in \G\\
\ul{u}=(u_0,u_2)\text{ with }u_i:\S_i\ra M_i,~i=0,2
\end{array}\right.\]
satisfying
\[ \left\{\begin{array}{l}
u_2\text{ hits the pseudocycle }h\text{ at the interior marked point};\\
(u_0(e^{-i\tt}),u_2(\eit))\in L_{02}^{\g,\tt}\text{ for any }\tt;\\
\ul{u}\text{ solves a perturbed Cauchy-Riemann equation; and}\\
\ul{u}\text{ has finite energy.}
\end{array}\right.\]
\item Define $\M_{\cl}(h)$ to be the moduli space of pairs $(\g,\ul{u})$ where
\[ \left\{\begin{array}{l}
\g\in \G\\
\ul{u}=(u_0,u_1,u_2)\text{ with }u_i:\S_i\ra M_i,~i=0,2\text{ and }u_1:\S_1^{\cl}\ra M_1
\end{array}\right.\]
satisfying
\[ \left\{\begin{array}{l}
u_2\text{ hits the pseudocycle }h\text{ at the interior marked point};\\
(u_0(e^{-i\tt}),u_1(\eit,0))\in L_{01}^{\g,\tt}\text{ for any }\tt;\\
(u_1(\eit,-\cl),u_2(\eit))\in L_{12}^{\g,\tt}\text{ for any }\tt;\\
\ul{u}\text{ solves a perturbed Cauchy-Riemann equation; and}\\
\ul{u}\text{ has finite energy.}
\end{array}\right.\]
\item For any integer $i$, define $\M_0^i(h)$ (resp. $\M_{\cl}^i(h)$) to be the component of $\M_0(h)$ (resp. $\M_{\cl}(h)$) which has virtual dimension $i$.
\end{enumerate}
\end{definition}

\noindent In order to achieve compactness for the moduli spaces we have just defined, it is necessary to impose the following energy-index relation for $\ul{u}$, in addition to the monotonicity of $M_i$ and $L_{ij}^{\g,\tt}$.
\begin{assume} \label{assumeeind} There exist continuous families $\{g_{01}^{\g,\tt}\}_{(\g,\eit)\in\G\times S^1}$, $\{g_{12}^{\g,\tt}\}_{(\g,\eit)\in\G\times S^1}$ of continuous functions
\[g_{01}^{\g,\tt}:L_{01}^{\g,\tt}\ra\RR\quad\text{ and }\quad g_{12}^{\g,\tt}:L_{12}^{\g,\tt}\ra\RR,\]
and a continuous function $C:\G\ra\RR$ such that for any $\cl>0$, $\g\in\G$ and $\ul{u}=(u_0,u_1,u_2)$ satisfying the seam condition in Definition \ref{moduliforgluing}(2), we have
\begin{equation} \label{eindrelation}
\int\ul{u}^*\ul{\w} +\int_0^{2\pi} g_{01}^{\g,\tt}(u_0(e^{-i\tt}),u_1(\eit,0))+g_{12}^{\g,\tt}(u_1(\eit,-\cl),u_2(\eit))~d\tt = \monoconst\cdot \ind(\ul{u}) + C(\g)
\end{equation}
where $\ind(\ul{u})$ is the Fredholm index of $\ul{u}$.
\end{assume}

\noindent Clearly, if \eqref{eindrelation} holds for a fixed value of $\cl>0$, then it holds for any $\cl$. Moreover, it implies that for any $\ul{u}=(u_0,u_2)$ satisfying the seam condition in Definition \ref{moduliforgluing}(1), we have
\[\int\ul{u}^*\ul{\w} +\int_0^{2\pi} g_{02}^{\g,\tt}(u_0(e^{-i\tt}),u_2(\eit))~d\tt = \monoconst\cdot \ind(\ul{u}) + C(\g)\]
where $g_{02}^{\g,\tt}:L_{02}^{\g,\tt}\ra\RR$ is defined as follows. Given $(x_0,x_2)\in L_{02}^{\g,\tt}$. Since $L_{02}^{\g,\tt}=L_{01}^{\g,\tt}\circ L_{12}^{\g,\tt}$ is embedded, there exists a unique $x_1\in M_1$ such that $(x_0,x_1)\in L_{01}^{\g,\tt}$ and $(x_1,x_2)\in L_{12}^{\g,\tt}$. We define
\[ g_{02}^{\g,\tt}(x_0,x_2):= g_{01}^{\g,\tt}(x_0,x_1)+g_{12}^{\g,\tt}(x_1,x_2).\]

Let $h:N\ra M_2$ be a pseudocycle. Fix a smooth map $h':N'\ra M_2$ with $\dim N'\leqslant \dim N -2$ which covers the limit set of $h$.
\begin{theorem}\label{annulusshrink} Assume $\M_0^i(h)$ is regular $i\leqslant 0$ and $\M_0^i(h')$ is regular $i\leqslant -2$. Then for any sufficiently small $\cl>0$ and $i\leqslant 0$, $\M_{\cl}^i(h)$ is regular. Moreover, there exists a bijection
\[ \M_0^{i=0}(h) \simeq \M_{\cl}^{i=0}(h).\]
If some of $M_i$ or $L_{ij}^{\g,\tt}$ are non-compact, the conclusion still holds provided that an a priori $C^0$-bound for $\ul{u}\in \M_{\cl}(h)$ is given.
\end{theorem}

Theorem \ref{annulusshrink} is proved by the same arguments in \cite{WW2}. An important issue is the elimination of \textit{figure-eight bubbles} arising from the Gromov compactness argument when one proves the surjectivity of the gluing map. Wehrheim-Woodward handled these bubbles as follows. They established a positive lower bound for the energy of figure-eight bubbles so that every sequence of solutions contains a convergent subsequence in the sense of Gromov, by the standard rescaling argument. If a Gromov limit contains a bubble (disk/sphere/figure-eight), there is loss of energy of the main component. By the energy-index relation, its Fredholm index decreases, and hence this Gromov limit does not exist generically. In our case, we take the energy to be the LHS of \eqref{eindrelation}. Notice that if $L_{ij}^{\g,\tt}$ are compact, or an a priori $C^0$-bound for $\ul{u}$ exists, the integral $\int_0^{2\pi} g_{01}^{\g,\tt}(u_0(e^{-i\tt}),u_1(\eit,0))+g_{12}^{\g,\tt}(u_1(\eit,-\cl),u_2(\eit))~d\tt$ vanishes for any disk or figure-eight bubble. It follows that our energy becomes the standard energy for these bubbles, and hence the argument in \cite{WW2} carry over our case.

In this paper, Theorem \ref{annulusshrink} will be applied to the following situation. Consider the set-up in Section \ref{QFTsection}. Take $M_0:=pt$, $M_1:=T^*G$ and $M_2:=X^-\times X$. Let $\vp:\G\ra\Lo G$ be a smooth cycle. Define $L_{01}^{\g,\tt}:=T^*_{\vp_{\g}(\eit)}G$ and $L_{12}^{\g,\tt}:=C$. Then $L_{02}^{\g,\tt}=\D_{\vp_{\g}(\eit)}$ is embedded. By Lemma \ref{monolemma}, these Lagrangian correspondences are monotone. In order to apply Theorem \ref{annulusshrink}, there are two issues to settle, the energy-index relation and the existence of an a priori $C^0$-bound for $\ul{u}$. The first issue is resolved by the following
\begin{lemma} Define $g_{01}^{\g,\tt}:  T^*_{\vp_{\g}(\eit)}G\ra \RR$ to be the linear map
\[ g_{01}^{\g,\tt}(\eta):=-\langle\eta, \pt_{\tt}\vp_{\g}(\eit)\rangle\]
and $g_{12}^{\g,\tt}\equiv 0$. Then Assumption \ref{assumeeind} holds.
\end{lemma}
\begin{proof}
Let $\ul{u}=(u_0,u_1,u_2)\in \M_{\cl}(h)$. Denote by $\l_G$ the canonical Liouville form on $T^*G$ and put $\w_G:=d\l_G$. By Stokes' theorem,
\[ \int u_1^*\w_G = \int (u_1|_{S^1\times \{0\}})^*\l_G - \int (u_1|_{S^1\times \{-\cl\}})^*\l_G .\]
Notice that $\int (u_1|_{S^1\times \{0\}})^*\l_G=- \int_0^{2\pi} g_{01}^{\g,\tt}(u_1(\eit,0))~d\tt$. The rest follows from Lemma \ref{monolemma}.
\end{proof}

The second issue follows from the convexity argument in \cite{Ab_IHES}. Let $\ul{u}\in\M_{\cl}(h)$ be an element. We require $u_1$ to solve the perturbed Cauchy-Riemann equation
\[ (du_1-X_{H'})^{0,1}=0\]
with respect to a domain-dependent $\w_G$-compatible almost complex structure of contact type where $H'=H+F$ is a \Ham with $F$ uniformly bounded\footnote{We also require that $F$ vanishes over a sequence of necks in the cylindrical end of $T^*G$ which diverges to infinity. See \cite{Ab_IHES} for more detail.} and $H$ quadratic at infinity, i.e. $H$ equals half squared-length of cotangent vectors with respect to a Riemannian metric on $G$. The key point is that as long as the image of a portion of $S^1\times \{0\}$ under $u_1$ is sufficiently away from the zero section, the integral $\int (u_1|_{S^1\times \{0\}})^*\l_G- H'\circ (u_1|_{S^1\times \{0\}}) ~d\tt $ is negative, since the image of the smooth cycle $\vp$ is compact.

\section{An alternative proof of Peterson-Woodward's formula}\label{PWproof}
We give here an alternative proof of Peterson-Woodward's formula \cite{Peter, Wformula} based on what we have developed. It is well-known that this formula is a consequence of Lam-Shimozono's theorem \cite{LS} which states that $\See{G/L}{\O G}(x_q)$ in Theorem \ref{computeG/Lintro}(a) does not contain any term other than the leading term and it is zero if $\deg^{L/T}(q)\ne 0$. A proof for this implication is given by Huang-Li \cite{HL}. Although we are not able to prove Lam-Shimozono's result, we are still able to obtain the formula. The key point is that in the proof of Huang-Li, the Bott-Samelson cycles, which represent $x_q$, do not play any role.

\begin{proposition}\label{image=leading} The image of $\See{G/T}{\O G}$ is equal to the direct sum of the leading terms, i.e.
\[ \im\left( \See{G/T}{\O G}\right) = \bigoplus_{\substack{(w,q)\in W\times \Q\\ q+w^{-1}(a)\in \wc_0}} \ZZ\langle \s^{G/T}_{w(x_0)} \TT^{A^{G/T}_q}\rangle \]
where $a\in\mr{\wc}_0$ is an element sufficiently close to the origin.
\end{proposition}
\begin{proof} WLOG, assume $G$ is of adjoint type. By Theorem \ref{computeG/T} and a filtration argument, it suffices to verify the inclusion $\subseteq$. Let $\a\in H_{-*}(\O G;\ZZ)$ be a homogeneous element. Suppose the coefficient of a term $\s^{G/T}_{w(x_0)} \TT^{A^{G/T}_q}$ in $\See{G/T}{\O G}(\a)$ with $q+w^{-1}(a)\not\in\wc_0$ is non-zero. Then there exists a simple root $\a_0$ (i.e. the positive root which defines a boundary wall of $\wc_0$) such that $\a_0(q+w^{-1}(a))<0$. We choose an element $q_1\in\Q\cap \wc_0$ as follows. If $\a_0(w^{-1}(a))<0$, then $\a_0(q)\leqslant 0$, and we choose $q_1$ such that $\a_0(q+q_1)=1$; if $\a_0(w^{-1}(a))>0$, then $\a_0(q)\leqslant -1$, and we choose $q_1$ such that $\a_0(q+q_1)=0$. In any case, we have $\a_0(q_1)\geqslant 1$. By Corollary \ref{G/Tnohigher} and Theorem \ref{MC=Se}, the coefficient of $\s^{G/T}_{w(x_0)} \TT^{A^{G/T}_{q+q_1}}$ in $\See{G/T}{\O G}(x_{q_1}\bulletsmall \a)$ is non-zero.

Now we take $G/L$ to be the coadjoint orbit passing through an interior point\footnote{More precisely, this point must lie in the ray passing through $\rho-\a_0$ in order for $G/L$ to be monotone.} $y_0$ of $\{\a_0=0\}\cap \wc_0$ in $\{\a_0=0\}$. Our choice of $q_1$ implies that $x:=w(x_0)$, $y:=w(y_0)$ and $q_0+\Q_{R_y}:=w(q+q_1)+\Q_{R_y}$ satisfy the conditions in Proposition \ref{compare}. Therefore, we conclude that the coefficient of $\s^{G/L}_{w(y_0)} \TT^{A^{G/L}_{q+q_1+\Q_{R_{y_0}}}}$ in $\See{G/L}{\O G}(x_{q_1}\bulletsmall \a)$ is non-zero. However, we have, by the same argument used in the proof of Theorem \ref{computeG/L}(b), $\See{G/L}{\O G}(x_{q_1})=0$, a contradiction.
\end{proof}

\begin{corollary} \label{pseudoPeter} There exists a $\ZZ$-basis $\{x_q'\}_{q\in \Q}$ of $H_*(\O G;\ZZ)$, possibly different from $\{x_q\}_{q\in\Q}$, such that for any monotone coadjoint orbit $G/L$ and any $q\in \Q$,
\[ \See{G/L}{\O G}(x_q') = \left\{
\begin{array}{cc}
\pm \s^{G/L}_{w_q(y_0)} \TT^{A^{G/L}_{w_q^{-1}(q)+\Q_{R_{y_0}}}} & ,~\deg^{L/T}(q)=0\\
0 &,~\deg^{L/T}(q)\ne 0
\end{array}
\right. .\]
\end{corollary}
\begin{proof}
We take $\{x'_q\}_{q\in \Q}$ to be the unique $\ZZ$-basis for which the above conclusion holds for $G/T$ (with positive sign for the non-vanishing term). Its existence follows from Proposition \ref{image=leading} and the injectivity of $\See{G/T}{\O G}$ which is a consequence of Corollary \ref{G/Tinjective} and Theorem \ref{MC=Se}. The case for general $G/L$ follows from the same arguments used for the proof of Theorem \ref{computeG/L}.
\end{proof}

As proved by Huang-Li, the Peterson-Woodward's formula is a formal consequence of Corollary \ref{pseudoPeter}. For the sake of completeness, we give an exposition of their argument. Denote by $\star$ the quantum cup product on $QH^*(G/L;\ZZ[\pi_2(G/L)])$. For any $y_1,\ldots, y_k\in\crit_{G/L}$, write
\[ \s^{G/L}_{y_1}\star\cdots\star  \s^{G/L}_{y_k} =\sum_{\substack{y\in \crit_{G/L}\\ A\in \pi_2(G/L)}} C^{G/L,A}_{y_1,\ldots,y_k;y} \s^{G/L}_y\TT^A.\]
It is well-known that the structure constant $C^{G/L,A}_{y_1,\ldots,y_k;y}$ is always non-negative.
\begin{theorem} \label{PeterWoodward} (Peterson-Woodward's comparison formula \cite{Peter, Wformula}) Given $y_1,\ldots, y_k, y\in\crit_{G/L}$ and $q+\Q_{R_y}\in \Q_0/\Q_{R_y}$. Define $x_i\in\crit_{F_{y_i}}$ to be the element with shortest length $\ell(x_i)$ (i.e. $\dim\UU_{x_i}^{F_{y_i}}=0$), $\tilde{q}\in\Q_0$ to be the Peterson lift of $q+\Q_{R_y}$ (Definition \ref{petersonlift}(1)), and $x\in\crit_{F_y}$ to be the associated lift of $y$ with respect to $q+\Q_{R_y}$ (Definition \ref{petersonlift}(2)). Then we have
\[ C^{G/L, A^{G/L}_{y,q+\Q_{R_y}}}_{y_1,\ldots, y_k; y} = C^{G/T, A^{G/T}_{x,\tilde{q}}}_{x_1,\ldots, x_k; x}. \]
\end{theorem}
\begin{proof}
Denote by $w_i\in W$ (resp. $w\in W$) the unique element such that $w_i(x_0)=x_i$ (resp. $w(x_0)=x$). Choose, for each $i=1,\ldots, k$, $q_i\in \Q\cap\wc_0$ which lies in the interior of $\t_{y_0}\cap \wc_0$ in $\t_{y_0}$, where $\t_{y_0}:=\bigcap_{\a(y_0)=0}\{\a=0\}$, such that
\begin{equation}\label{PWproof1}
\a\left(\tilde{q}+\sum_{i=1}^kw_i(q_i) \right)>0\text{ for any }\a\in R\text{ with }\a(y)>0.
\end{equation} We have, by Corollary \ref{pseudoPeter},
\begin{align*}
\See{G/T}{\O G}(x'_{w_i(q_i)} )&=\pm \s^{G/T}_{x_i} \TT^{A^{G/T}_{q_i}} \\
\See{G/L}{\O G}(x'_{w_i(q_i)} )&=\pm \s^{G/L}_{y_i} \TT^{A^{G/L}_{q_i+\Q_{R_{y_0}}}}
\end{align*}
Put $q':= \tilde{q}+\sum_{i=1}^kw_i(q_i)$. Observe that $q'$ is the Peterson lift of $q+ \sum_{i=1}^kw_i(q_i)+\Q_{R_y}$ and $w_{q'}=w$. We have, by the same corollary,
\begin{align*}
\See{G/T}{\O G}(x'_{q'}) &= \pm \s^{G/T}_x \TT^{A^{G/T}_{x,q'}}\\
\See{G/L}{\O G}(x'_{q'}) &= \pm \s^{G/L}_y \TT^{A^{G/L}_{y,q'+\Q_{R_y}}}.
\end{align*}
It follows that $C^{G/L, A^{G/L}_{y,q+\Q_{R_y}}}_{y_1,\ldots, y_k; y}$ and $C^{G/T, A^{G/T}_{x,\tilde{q}}}_{x_1,\ldots, x_k; x}$ are both equal to, up to sign, the coefficient of $x'_{q'}$ in the expression $x'_{w_1(q_1)}\bulletsmall\cdots\bulletsmall x'_{w_k(q_k)}$. The proof is complete by the semi-positivity of these structure constants.
\end{proof}



\end{document}